\documentclass[11pt,twoside]{article}
\usepackage{latexsym}
\usepackage{amssymb,amsbsy,amsmath,amsfonts,amssymb,amscd}
\usepackage{mathrsfs}
\usepackage{epsfig, graphicx}
\usepackage{graphics,epsfig}
\usepackage{color}
\usepackage{enumerate}
\usepackage{epstopdf}
\usepackage{titlesec}
\usepackage[titletoc,toc,title]{appendix}
\usepackage{subfigure}
\usepackage[utf8]{inputenc}
        \usepackage[T1]{fontenc}
        \usepackage{lmodern}
        \usepackage{graphicx}
        \usepackage{caption}
         \usepackage{floatrow}%
\usepackage{tikz}
\usetikzlibrary{arrows,positioning, arrows}

\usepackage{amsthm}
\usepackage{amsfonts}
\usepackage{graphicx}
\usepackage{caption}
\usepackage{floatrow}
\usepackage{amsmath}
\usepackage{amssymb}
\usepackage{amsmath}
\usepackage{bm}
\usepackage{subeqnarray}
\usepackage[linesnumbered,ruled]{algorithm2e}
\usepackage[noend]{algpseudocode}
\usepackage{cases}
\theoremstyle{plain}
    \newtheorem{theorem}{Theorem}
    \newtheorem{lemma}{Lemma}
    
    \newtheorem{proposition}{Proposition}

\theoremstyle{definition} 
    
    \newtheorem{remark}{Remark}

\textheight 9in \textwidth 6.5in \numberwithin{equation}{section}

\newcommand{\eps}{\varepsilon}

\newcommand{\half}{\frac{1}{2}}
\newcommand{\RR}{\mathbb{R}}
\newcommand{\rd}{\mathrm{d}}
\newcommand{\equilibrium}{\mathcal{M}}
\newcommand{\Lop}{\mathcal{L}}

\newcommand{\Amat}{\mathsf{A}}
\newcommand{\Dmat}{\mathsf{D}}

\newcommand{\Mmat}{\mathsf{M}}
\newcommand{\Kmat}{\mathsf{K}}

\newcommand{\Cmat}{\mathsf{C}}
\newcommand{\Lmat}{\mathsf{L}}
\newcommand{\Pmat}{\mathsf{P}}
\newcommand{\Imat}{\mathsf{I}}

\newcommand{\vecf}{\mathsf{f}}
\newcommand{\FT}{\mathsf{F}}
\newcommand{\IFT}{\mathsf{F}^{-1}}
\newcommand{\FL}{\mathsf{L}_s}
\newcommand{\tr}{\tilde{\rho}}

\newcommand{\feven}{\tilde{f}}

\newcommand{\average}[1]{\left\langle#1\right\rangle}

\graphicspath{{../}}

\title{An asymptotic preserving scheme for L\'{e}vy-Fokker-Planck equation with fractional diffusion limit}
\date{}
\author{Wuzhe Xu and Li Wang \thanks{School of Mathematics, University of Minnesota, Twin cities, MN 55455. W.Xu (xu000355@umn.edu), L.Wang (wang8818@umn.edu).}}

\begin{document}
\maketitle
\begin{abstract}
    In this paper, we develop a numerical method for the L\'evy-Fokker-Planck equation with the fractional diffusive scaling. There are two main challenges. One comes from a two-fold nonlocality, that is, the need to apply the fractional Laplacian operator to a power law decay distribution. The other arises from   long-time/small mean-free-path scaling, which introduces stiffness to the equation. To resolve the first difficulty, we use a change of variable to convert the unbounded domain into a bounded one and then apply the Chebyshev polynomial based pseudo-spectral method. To treat the multiple scales, we propose an asymptotic preserving scheme based on a novel micro-macro decomposition that uses the structure of the test function in proving the fractional diffusion limit analytically. Finally, the efficiency and accuracy of our scheme are illustrated by a suite of numerical examples.
\end{abstract}

\section{Introduction}
We consider the L\'{e}vy-Fokker-Planck (LFP) equation 
\begin{equation} \label{eqn:VLFP}
\begin{cases}{}
\partial_t f + v \cdot \nabla_x f = \nabla_v \cdot (v f) - (-\Delta_v)^s f :=\Lop^s (f)\,, \qquad s\in (0,1)\,,
\\ f(0,x,v) = f_{in}(x,v)\,,
\end{cases}
\end{equation}
where $f(t,x,v): (0, \infty) \times \RR^d \times \RR^d  \mapsto  \RR^+$ is the distribution function of a cloud of particles in plasma, which undergoes a free transport describing by the convection on the left hand side, and an interaction with the background, described by the LFP operator on the right. Here $(-\Delta_v)^s$ is the fractional Lapacian operator that models the L\'evy processes at the microscopic level. Among various  equivalent ways of defining the fractional Laplace operator \cite{kwasnicki2017ten}, we only mention the one that will be used throughout the paper:   
\begin{equation}
    (-\Delta_v)^s f(v) : = C_{s,d} \text{P.V.} \int_{\mathbb{R}^d} \frac{f(v)-f(w)}{|v-w|^{d + 2s}} \rd w,
\label{def_fl}
\end{equation}
where $P.V.$ denotes Cauchy principal value and $C_{s,d} = \frac{4^s \Gamma(d/2+s)}{\pi^{d/2}|\Gamma(-s)|}$.

When $s=1$, the right hand side of \eqref{eqn:VLFP} reduces to the to classical Fokker-Planck operator that models the Brownian motion of the microscopic particles. In contrast, $s\in (0,1)$ allows particles to make long jumps at the microscopic scale, and hence leads to the nonlocal effect at the mesoscopic scale. Consequently, as opposed to the Gaussian Maxwellian to the Fokker-Planck operator, the L\'{e}vy-Fokker-Planck operator admits an equilibrium that has a fat tail. More precisely, there is a unique normalized distribution $\equilibrium(v)$, such that \cite{gentil2008levy}:
\begin{equation} \label{eqn:equilibrium}
\Lop^s(\equilibrium) = 0, \quad \int_{\RR^d} \equilibrium (v) \rd v = 1, \quad \equilibrium(v) \sim \frac{C}{|v|^{d+2s}} ~\text{as}~ |v| \rightarrow \infty\,.
\end{equation}
Moreover, the convergence toward the equilibrium is shown to be exponential in the sense of relative entropy. In particular, let $\Phi$ be a convex smooth function, and define the relative entropy to the equilibrium as 
\begin{equation*}
   H^\Phi(f|\equilibrium) :=\int \Phi(f) \equilibrium  \rd v -\Phi\left(\int f \equilibrium \rd v \right)\,,
\end{equation*}
then from the Theorem 2 in \cite{gentil2008levy}, one has 
\begin{equation*}
H^\Phi (f|\equilibrium)\left(\frac{f(t)}{\equilibrium}\right) \leq e^{-\frac{t}{C}} H^\Phi (f_{in}|\equilibrium)\left(\frac{f(t)}{\equilibrium}\right), \quad t \geq 0\,,
\end{equation*}
for some constant $C$.

In the long time and small mean free path scaling, that is, rescaling the time and space as 
\begin{equation} \label{scaling}
x  \mapsto \eps x, \quad t\mapsto \eps^{2s} t,
\end{equation}
equation \eqref{eqn:VLFP} can be rewritten into the following form:
\begin{equation} \label{eqn:111}
\begin{cases}{}
\eps^{2s}\partial_t f^\eps + \eps v \cdot \nabla_x f^\eps = \mathcal L^s (f^\eps) ,
\\ f^\eps(0,x,v) = f_{in}(x,v)\,.
\end{cases}
\end{equation}
Sending $\eps$ to zero will give rise to a fractional diffusion equation for the density of particles. More precisely, we cite the following theorem from \cite{cesbron2012anomalous}.
\begin{theorem} 
Assume that $f_0 \in L^2 (\RR^N, \equilibrium (v)^{-1}\rd v \rd x)$, where $\equilibrium(v)$ is the unique normalized equilibrium distribution that satisfies \eqref{eqn:equilibrium}.
Then, up to a subsequence, the solution $f^\eps$ of \eqref{eqn:111} converges weakly in $L^\infty(0, T; L^2(\RR^{2d}, \equilibrium(v)^{-1} \rd v \rd x))$ to $\rho(t,x) \equilibrium (v)$ as $\eps \rightarrow 0$, where $\rho(t,x)$ solves
\begin{equation}\label{eqn: limit_system}
\begin{cases}{}
\partial_t \rho + (-\Delta_x)^s \rho= 0\,,
\\ \rho(0,x) = \rho_{in}(x) := \int_{\RR^d} f_{in}(x,v)\rd v\,.
\end{cases}
\end{equation}
\end{theorem}

In the classical case (i.e., $s$=1) when $\equilibrium$ is a fast decaying function such as Gaussian, one rescales $t$ as $t \mapsto \eps^2 t$ and the resulting macroscopic equation is the diffusion equation \cite{poupaud2000parabolic}:
\[
\partial_t \rho + \nabla_x \cdot ( D \nabla_x \rho)  = 0, 
\]
where $D$ is the diffusion matrix 
\[
D = \int v \otimes  v \equilibrium \rd v\,.  
\]
Clearly the fat tail equilibrium \eqref{eqn:equilibrium} renders the above integral unbounded and therefore invalids the classical diffusion limit. Conversely, the anomalous scaling \eqref{scaling} is necessary. Similar scaling has also been investigated in the framework of linear Boltzmann equation, see \cite{mellet2011fractional, abdallah2011fractional} for a reference.

Numerically computing \eqref{eqn:111} has two major challenges. One comes from necessity to apply the fractional Laplacian operator to a slow decay function, in which case one needs to consider {\it infinite} computational domain, since any truncation would lose the important information carried by the tail and leads to erroneous result. To this end, two kinds of numerical methods have been developed to approximate the fractional Laplacian operator. One hinges upon the integral form of $(-\Delta_v)^s$ in \eqref{def_fl} and splits the infinite domain into a computable body part and a compensate tail part that can be integrated exactly, see especially \cite{huang2014numerical, huang2016finite}. This method heavily relies on analytic expression of the tail, which is not known for our case except when the solution has reached equilibrium. But since our goal is to simulate the dynamics, this method cannot be adopted without a major modification. The other uses the spectral method with {\it nonlocal} basis. For instance, the Hermite spectral method is developed in \cite{mao2017hermite} which takes advantage of the fact that Hermite polynomials are invariant under Fourier transform and therefore can be efficiently computed when taking the fractional Laplacian. In \cite{sheng2020fast}, mapped Chebyshev polynomials are used as basis, then the fractional Laplacian is calculated via the celebrated Dunford-Taylor formula, and therefore is very efficient in higher dimensions. Another approach, proposed in \cite{cayama2019pseudospectral}, also starts from the Chebyshev polynomial basis, but it then uses a change of variable and even extension of the function, and therefore boils down the problem to computing the Fractional Laplacian of the Fourier basis on a finite domain, which can be approximated numerically with high accuracy. All these spectral methods were developed to address the nonlocality of the fractional Lapalcian operator, but when they apply to a slow decay function, additional nonlocality is introduced and a large number of basis are expected in order to meet a certain accuracy. See also \cite{lischke2020fractional, bonito2018numerical} for a review on numerical issues related to fractional diffusion. For our problem, we find that the approach in \cite{cayama2019pseudospectral} gives the best result within the computational budget when some tuning parameters are chosen appropriately. 

Another challenge arises from the diffusive scaling, which introduces stiffness to the system. Our goal is to develop a numerical solver with uniform performance across different regimes, i.e., $\eps$ varies in magnitude by several orders. In particular, the scheme for \eqref{eqn:111} is expected to  reduce to the solver for \eqref{eqn: limit_system} automatically with unresolved mesh. This is the so-called asymptotic preserving (AP) scheme. There has been an extensive study on the AP scheme for kinetic equations with various scalings, see \cite{jin2010asymptotic, hu2017asymptotic} for a review. When it comes to anomalous diffusive scaling, we cite specially \cite{wang2016asymptotic, crouseilles2016numerical, CHL16, wang2019asymptotic}, which all deal with the linear Boltzmann type equation. These works and the current paper share the same equilibrium \eqref{eqn:equilibrium} and diffusion limit \eqref{eqn: limit_system}, but the different nature between Fokker Planck type and Boltzmann type operator lead to very difference convergence mechanism and therefore hinders the application of the methods developed there. In fact, when the collision is of the Boltzmann type, a strong convergence toward the anomalous diffusion is obtained \cite{mellet2011fractional, abdallah2011fractional}, which plays a significant role in designing the AP method. On the contrary, with the L\'evy-Fokker-Planck operator in our case, only a weak convergence is available, which gives very limited knowledge on how the scheme can be constructed. Nevertheless, the special choice of the test function that aids the proof of Theorem 1 in \cite{cesbron2012anomalous} inspires our macro-micro decomposition, which sets the base of our AP scheme.

The rest of the paper is organized as follows. In the next section, we detail the computation of the collision operator $\Lop^s$ and combine it with backward Euler scheme to solve the spatially homogeneous case. Section 3 is devoted to the design of AP scheme, along with a rigorous proof of the AP property and a detailed guide in implementation. In section 4, extensive numerical examples are given to test the performance of our AP scheme. Finally the paper is concluded in section 5.

\section{Computation of the collision operator $\Lop^s$}
Aside from multiple scales that appear in equation \eqref{eqn:111}, the collision operator $\Lop^s$ itself poses severe computational difficulties, which is attributed to the non-locality of the operator $(-\Delta_v)^s$. As already mentioned in \cite{huang2014numerical, huang2016finite}, if $f$ is compactly supported in $v$, then the computation of $(-\Delta_v)^s f$ can be fulfilled by the Fourier transform. However, in our case, the interplay between the two operators in $\Lop^s$ leads to an equilibrium that has only a power law decay, see \eqref{eqn:equilibrium}. As a result, a more sophisticated treatment is needed for fractional Laplacian, as will be detailed in the following section. Here for notation simplicity, we assume $f$ only depends on $v$ throughout this section.

\subsection{Change of variable}
As mentioned above, one of the major difficulties of fractional Laplacian operator $(-\Delta_v)^s$ is its non-locality, especially when it applies to a slow decaying function like $\equilibrium$ in \eqref{eqn:equilibrium}. In order to treat the fat tail distribution on an unbounded domain, two approaches can be taken: one is to truncate the computation domain and introduce suitable boundary conditions; the other is to use a change of variable that maps the infinite domain into a finite one and then use a spectral method. In this paper, we take the latter approach, which is also termed as the rational spectral method. Below we briefly introduce the rational Chebyshev spectral method, more details can be found in \cite{cayama2019pseudospectral}.

Consider an algebraic mapping that maps the unbounded domain $(- \infty, \infty)$ into $[-1,1]$, i.e.\,,
\begin{equation*}
\xi = \frac{v}{\sqrt{L_v^2+v^2}} \in (-1,1) \Longleftrightarrow v = \frac{L_v \xi}{ \sqrt{1 - \xi^2}} \in (-\infty, \infty)\,,
\end{equation*}
where $L_v$ is a scaling parameter that is chosen for the sake of accuracy and mass conservation. In principle, it can be chosen adaptive as mentioned in \cite{ma2005hermite}, see also Remark~\ref{rmk: mass} for more discussion.
Take the Chebyshev polynomials of the first kind as the basis on $[-1,1]$,
\begin{equation} \label{eqn22}
T_k(\xi) = \cos (k \arccos (\xi))\,, \qquad \xi \in [-1,1]\,,
\end{equation}
then 
\[
TB_k(v) = T_k (\frac{v}{\sqrt{L_v^2+v^2}} )\,, \qquad  v\in \mathbb{R}\,,
\]
is the so-called Chebyshev rational polynomials on infinite domain. It has been pointed out in \cite{boyd1987spectral} that Chebyshev rational polynomials are appropriate for approximating the algebraically decay function, and has also been used in \cite{sheng2020fast} to compute the fractional diffusion operator. 

We now concentrate on the finite domain $[-1,1]$ and employ a further change of variable.
Let $q = \arccos (\xi) \in [0, \pi]$, then 
\begin{equation}\label{change_var}
    v = \frac{L_v \xi}{ \sqrt{1 - \xi^2}} = L_v \cot (q) ,
\end{equation}
and $T_k(\xi)$ in \eqref{eqn22} reduces to 
\[
T_k(\xi) = \cos(kq).
\]
Therefore, expanding $f(q)$ in terms of Chebyshev functions is equivalent to a cosine expansion. 

In the new variable $q \in [0, \pi]$, the fractional Laplacian can be rewritten as \cite{cayama2019pseudospectral} 
\[
(-\Delta_q)^{s} f(q)=
\left\{ \begin{array}{cc}
-\frac{1}{L_v \pi} \int_{0}^{\pi} \frac{f'(p)}{\cot (q)-\cot (p)} \rd p, \qquad & s=\half\,,
\\ \frac{C_{s,d}}{2L_v^{2s} s(1-2s)} 
\int_{0}^{\pi} \frac{\sin ^{2}(p) f''(p)+2 \sin (p) \cos (p) f'(p)}{|\cot (q)-\cot (p)|^{2s-1}} \rd p, \qquad & s \neq \half \,.
\end{array} \right.
\]
And one can reformulate \eqref{eqn:VLFP} into 
\begin{equation} \label{eqn:VLFP_cv}
\begin{cases}{}
\partial_t f + L_v \cot(q)  \partial_x f = f - \cos(q) \sin(q) \partial_q f(q) - (-\Delta_q)^s f :=\Lop^s_q (f)\,, 
\\ f(0,x,q) = f_{in}(x,q)\,,
\end{cases}
\end{equation}
where $f$ now depends on $t$, $x$ and $q\in[0,\pi]$. In the rest of this section, we will use \eqref{eqn:VLFP_cv} as our target equation, and discretize $q$ in the following way:
\begin{equation} \label{eqn:qpoint}
q_j = \frac{\pi(2j+1)}{2N_v}, \qquad 0\leq j \leq N_v-1\,,
\end{equation}
with $\Delta q = \frac{\pi}{N_v}$.

\subsection{Computation of $(-\Delta_q)^s f$}\label{sec:LFP}
To further ease the computation of $(-\Delta_q)^s f$, we conduct the even extension of $f$ at $\pi$, i.e., 
\[
\feven(q) = \left\{ \begin{array}{cc} f(q), \quad & q \in [0,\pi]\,, \\
  f(2\pi - q), \quad & q \in [\pi, 2\pi]\,. \end{array} \right.
\]
This way, according to the relation
\begin{equation*}
\int_0^{2 \pi} \feven(q) \cos(kq) \rd q = \int_0^{2 \pi} \feven(q) e^{-ikq} \rd q \,,
\end{equation*}
one can compute the coefficients of cosine expansion of $\feven$ via the Fast Fourier Transform. More specifically, denote 
\begin{equation} \label{ftildev}
\tilde{\vecf} = [\tilde f(q_0), \tilde{f}(q_1), \cdots,\tilde f(q_{N_v-1}),f(q_{2N_v-1}), \tilde f(q_{2N_v-2}), \cdots,  \tilde f(q_{N_v}) ]^T, 
\end{equation}
then its discrete Fourier transform takes the form 
\[
 f(q_j) = \sum_{k=-N_v}^{N_v-1}  \hat{\tilde f}_k  e^{ikq_j}\,, \qquad j = 0, \cdots, 2N_v-1\,, 
\]
where $\hat{\tilde f}_k = \frac{1}{2N_v} \sum_{j = -N_v}^{N_v-1} \tilde f(q_j) e^{-ik q_j}$, and hence
\[
(-\Delta_q)^s f(q_j) = \sum_{k=-N_v}^{N_v-1}  \hat{\tilde f}_k (-\Delta_q)^s e^{ikq_j}\,,\qquad  j = 0, \cdots, 2N_v-1\,.
\]
Then the question boils down to calculating $(-\Delta_q)^s e^{ikq_j}$, and we directly cite the result from \cite{cayama2019pseudospectral}. 
\begin{theorem}
Let $s \in (0,0.5) \cup (0.5,1) $, then
\begin{equation}\label{eqn: fl_fourier}
(-\Delta_q)^{s}\left(e^{i m q}\right)=\left\{\begin{array}{cc}
\frac{c_{s,1}|\sin (q)|^{2s-1}}{8 L^{2s} \tan \left( \pi s \right)} \sum_{l=-\infty}^{\infty} e^{i 2 l q}\left((1-2s) m^{2}-4 m l\right) \\
\times \frac{\Gamma\left(\frac{-1+2s}{2}+|l|\right) \Gamma\left(\frac{-1-2s}{2}+\left|\frac{m}{2}-l\right|\right)}{\Gamma\left(\frac{3-2s}{2}+|l|\right) \Gamma\left(\frac{3+2s}{2}+\left|\frac{m}{2}-l\right|\right)}, & m \text { even}\,, \\
i \frac{c_{s,1}|\sin (q)|^{2s-1}}{8 L^{2s}} \sum_{l=-\infty}^{\infty} e^{i 2 l q}\left((1-2s) m^{2}-4 m l\right)  \\
\times \operatorname{sgn}\left(\frac{m}{2}-l\right) \frac{\Gamma\left(\frac{-1+2s}{2}+|l|\right) \Gamma\left(\frac{-1-2s}{2}+\left|\frac{m}{2}-l\right|\right)}{\Gamma\left(\frac{3-2s}{2}+|l|\right) \Gamma\left(\frac{3+2s}{2}+\left|\frac{m}{2}-l\right|\right)}, & m \text { odd }\,.
\end{array}\right.
\end{equation}
Moreover, when $s=0.5$, 
\begin{equation*}
(-\Delta_q)^{0.5}\left(e^{i m q}\right)=\left\{\begin{array}{ll}
\frac{|m| \sin ^{2}(q)}{L} e^{i m q}, & m \text { even}\,, \\
\frac{i m}{L \pi}\left(\frac{-2}{m^{2}-4}-\sum_{l=-\infty}^{\infty} \frac{4 \operatorname{sgn}(l) e^{i 2 l q}}{(m-2 l)\left((m-2 l)^{2}-4\right)}\right), & m \text { odd }\,.
\end{array}\right.
\end{equation*}
\end{theorem}
To implement it numerically, one truncates $l$ by setting $l = l_1 N_v + l_2$, where $l_{2} \in\{-N_v / 2, \ldots, N_v / 2-1\}$, and $l_{1} \in\left\{-l_{l i m}, \ldots, l_{l i m}\right\}$. Then   \eqref{eqn: fl_fourier} is approximated as
\begin{equation}\label{eqn: fl_fourier_approx}
(-\Delta_q)^{s}\left(e^{i m q_{j}}\right) \approx\left\{\begin{array}{l}
\frac{c_{s,1}\left|\sin \left(q_{j}\right)\right|^{2s-1}}{8 L^{2s} \tan \left(\frac{\pi 2s}{2}\right)} \sum_{l_{2}=-N_v / 2}^{N_v / 2-1}\left[\sum_{l_{1}=-l_{l i m}}^{l_{l i m}} a_{m, l_{1}, l_{2}}\right] e^{i 2 l_{2} q_{j}}, \quad m \text { even } \,,\\
i \frac{c_{s,1}\left|\sin \left(q_{j}\right)\right|^{2s-1}}{8 L^{2s}} \sum_{l_{2}=-N_v / 2}^{N_v / 2-1}\left[\sum_{l_{1}=-l_{l i m}}^{l_{l i m}} a_{m, l_{1}, l_{2}}\right] e^{i 2 l_{2} q_{j}}, \quad m \text { odd }\,,
\end{array}\right.
\end{equation}
where
\[
a_{m, l_{1}, l_{2}}=\left\{\begin{array}{ll}
(-1)^{l_{1}}\left((1-2s) m^{2}-4 m\left(l_{1} N_v+l_{2}\right)\right) \\
\times \frac{\Gamma\left(\frac{-1+2s}{2}+\left|l_{1} N_v+l_{2}\right|\right) \Gamma\left(\frac{-1-2s}{2}+\left|\frac{m}{2}-l_{1} N_v-l_{2}\right|\right)}{\Gamma\left(\frac{3-2s}{2}+\left|l_{1} N_v+l_{2}\right|\right) \Gamma\left(\frac{3+2s}{2}+\left|\frac{m}{2}-l_{1} N_v-l_{2}\right|\right)}, & m \text { even }\,, \\
(-1)^{l_{1}}\left((1-2s) m^{2}-4 m\left(l_{1} N_v+l_{2}\right)\right) \operatorname{sgn}\left(\frac{m}{2}-l_{1} N_v-l_{2}\right) \\
\times \frac{\Gamma\left(\frac{-1+2s}{2}+\left|l_{1} N_v+l_{2}\right|\right) \Gamma\left(\frac{-1-2s}{2}+\left|\frac{m}{2}-l_{1} N_v-l_{2}\right|\right)}{\Gamma\left(\frac{3-2s}{2}+\left|l_{1} N_v+l_{2}\right|\right) \Gamma\left(\frac{3+2s}{2}+\left|\frac{m}{2}-l_{1} N_v-l_{2}\right|\right)}, & m \text { odd }\,.
\end{array}\right.
\]
Note that $l_{lim}$ here is an adjustable parameter, and it is obvious that larger $l_{lim}$ gives better approximation. For all our numerical examples in this paper, we use $l_{lim} = 300$.  

Now let the Matrix $\Mmat \in \mathbb{C}^{2N_v \times 2N_v}$ be
\[
\Mmat_{m,n} = \left\{ \begin{array}{cc} (-\Delta_q)^s e^{i(m-N_v) q_{n-1}} & n\leq N_v  \,, \\ (-\Delta_q)^s e^{i(m-N_v) q_{2N_v-n}} & N_v+1\leq n \leq 2N_v\,, \end{array} \right.
\]
then we have 
\[
 (-\Delta_q)^s \tilde \vecf   = (\Mmat \times \FT)~  \tilde \vecf \,,
\]
where $\tilde \vecf$ is defined in \eqref{ftildev}, and $\FT$ denotes the $2N_v$-periodic discrete Fourier transform. Confining the above calculation of $\tilde \vecf $ to $\vecf$, i.e., 
\begin{equation} \label{vectorf}
\vecf = (f(q_0), f(q_1), \cdots, f(q_{N_v-1}))^T. 
\end{equation}
we can write down the matrix representation of  $(-\Delta_q)^s \vecf$ as follows: 
\begin{equation}\label{LsD}
(-\Delta_q)^s  \vecf  = \FL \vecf \,,
\end{equation}
where $\FL = (\Mmat \times \FT)(1:N_v, 1:N_v) + (\Mmat \times \FT) (1:N_v, N_v+1:2N_v)$.

As already mentioned in the introduction, the way we treat the fractional Laplacian is not unique. We just choose the one that performs the best in our case in terms of accuracy and efficiency.

\subsection{Spatially homogeneous case} \label{sec:shc}
In this section, we detail the computation of the spatially homogeneous case of \eqref{eqn:VLFP_cv}:
\[
\partial_t f = \Lop^s_q(f) := f - \cos(q) \sin(q) \partial_q f - (-\Delta_q)^s f \,.
\]
Here the fractional Laplacian term is treated via the aforementioned pseudospectral method, and  $\partial_q f$ is discretized using the Fourier spectral method. Still using the vector form of the discrete $f$ defined in \eqref{vectorf},  we have the discretization of $\Lop^s_q f$ as 
\begin{equation} \label{lfp}
 \Pmat^s \vecf := (\Cmat \cdot \IFT \Kmat \FT + \Imat - \FL ) \vecf \,,
\end{equation}
where
\[
\Kmat = i
  \begin{pmatrix}
    -N_v & & \\
    & \ddots & \\
    & & N_v-1
  \end{pmatrix},
  \quad
  \Cmat =
  \begin{pmatrix}
    -\cos(q_0) \sin(q_0) & & \\
    & \ddots & \\
    & & -\cos(q_{N_v-1}) \sin(q_{N_v-1})
  \end{pmatrix} \,.
\]
Note specifically that even though the computation of $\FL$ can be expensive, it only needs to be computed once.

For time discretization, we choose the backward Euler scheme, namely, 
\begin{equation}\label{eqn: homo_scheme_ori}
\frac{1}{\Delta t} (\vecf^{n+1} - \vecf^n) = \Pmat^s \vecf^{n+1}\,,
\end{equation}
as it will be used in the spatially in-homogeneous case in the next section. 
In practice, the ODE system is solved by Matlab builtin solver ODE45 or ODE15s.

\begin{remark}[Positivity]
The scheme \eqref{eqn: homo_scheme_ori} is not positivity preserving, but it will be so after a slight modification. In fact, the positivity is lost when $q_j$ close to the boundary, $0$ or $\pi$. This is because the function value $f$ near the boundary are already small, then any small numerical error would render it negative. To resolve this issue, our idea is to use the tail information to reassign the value of $\vecf$ at the boundary. More precisely, since the equilibrium is proportional to $(1+L_v\cot(q))^{-(1+2s)}$ at its tail, then if there is an index $l$ close to the right boundary (i.e., $q=\pi$) such that $\vecf_l <0$ for the first time, namely $\vecf_j>0$ for $j<l$, then we set $\vecf_j = \frac{(1+L_v \cot(q(j-1)))^{1+2s}}{(1+L_v \cot(qj))^{1+2s}} $ for $j\geq l$.
\end{remark}

\begin{remark}[Choice of $L_v$ and $N_v$]\label{rmk: mass}
In the scheme $\eqref{eqn: homo_scheme_ori}$, $N_v$ and $L_v$ should be chosen according to $s$ for the sake of mass conservation. Since the tail of equilibrium distribution relates to $s$ via $\equilibrium \sim \frac{1}{|v|^{1+2s}}$ as $|v| \rightarrow \infty$, larger $N_v$ (with fix $L_v$) is required for smaller $s$ to capture the tail information and conserve the total mass. One should also choose the parameter $L_v$ properly. On one hand, for fixed $N_v$, $L_v$ should not be too large, otherwise accuracy is lost in approximating the body part of the distribution function. On the other hand, if $L_v$ is too small, the method lose the capability to capture the tail information. Therefore, in principle, the smaller the $s$ is, the larger number of mode $N_v$ and $L_v$ are needed. For instance, when $s \geq 0.5$, $N_v=64$ and $L_v=3$ is enough; whereas for $s=0.4$ and if $L_v = 3$, $N_v = 128$ is required to ensure the mass is conserved up to error  $\mathcal O(10^{-3})$, see Table~\ref{table:me}.
\end{remark}

\section{Asymptotic preserving scheme}
We introduce an asymptotic preserving scheme for the spatially inhomogeneous system and its implementation in this section. The main difficulty of capturing the anomalous diffusion limit is to acquire operator $(-\Delta_x)^s$ when only $(-\Delta_v)^s$ appears in original system. In the proof of Theorem 1, the authors use a test function in the form of $\phi(x+\eps v)$ to get $(-\Delta_x)^s$ from $(-\Delta_v)^s$ weakly via integration by parts, which indicates a kind of symmetry between $x$ and $v$. Inspired by this symmetry, we propose a scheme based on a novel micro-macro decomposition by requiring the macro part to respect such symmetry.

Let us first introduce some notation. Denote $\Omega_x = [-L_x, L_x]$ as our spatial domain and partition it into $N_x$ uniform grids, i.e. $x_i = -L_x +  (i+\half) \Delta x$, where $\Delta x = \frac{2L_x}{N_x}$. Periodic boundary condition in the spatial domain will be used. For velocity domain, we will work with the variable $q$ defined in \eqref{change_var} and use the same discretization as in \eqref{eqn:qpoint}. Then $f(t,x,q) = f(t,x,v(q))$, and we denote numerical approximation of $f(t_n,x_i,q_j)$ as $f^n_{i,j}$, where $t_n = n \Delta t$, $0\leq j \leq N_v-1$, $0\leq i \leq N_x-1$.

To start, we introduce the following lemmas.
\begin{lemma}\label{lemma:fl_comp}
$(-\Delta_v)^s (fg) = g (-\Delta_v)^s f  + f (-\Delta_v)^s g +I(f,g)$, where
\[
I(f,g) =C_{1,s} \int_{\RR^d} \frac{(f(v)-f(w))(g(w)-g(v))}{|v-w|^{d+2s}} \rd w\,.
\]
\end{lemma}
\begin{proof} 
\begin{eqnarray*}
    (-\Delta_v)^s (fg) & =& C_{1,s}  \int_{\RR} \frac{f(v)g(v)-f(w)g(w)}{|v-w|^{d+2s}} \rd w \, \\
    & =& C_{1,s}  \int_{\RR} \frac{f(v)g(v)-f(v)g(w)}{|v-w|^{d+2s}} \rd w + C_{1,s}  \int_{\RR} \frac{f(v)g(v)-f(w)g(v)}{|v-w|^{d+2s}} \rd w \, \\
    &&+ C_{1,s}  \int_{\RR} \frac{(f(v)-f(w))(g(w)-g(v))}{|v-w|^{d+2s}} \rd w \, \\
    & = & g (-\Delta_v)^s f  + f (-\Delta_v)^s g +I(f,g).
\end{eqnarray*}

\end{proof}

\begin{lemma}\label{lemma: interchange}
Suppose $h(x,v) \in L^1_v (\mathbb R) W^{2,1}_x(\mathbb R) \cap L_v^1(\mathbb R) C^2_x (\mathbb R)$, then $\average{(-\Delta_x)^s h(x,v)} = (-\Delta_x)^s \average{h} $, where $\average{f} := \int_{\mathbb R} f(x,v) \rd v$.
\end{lemma}
\begin{proof}
First, let us rewrite fractional Laplacian in following finite difference form,
\[
(-\Delta_x)^s h(x,v) = C_{1,s} \int_0^{\infty} (2h(x,v)-h(x-y,v)-h(x+y,v)) \nu(y) \rd y\,,
\]
where $\nu(y) = |y|^{-(1+2s)}$, which is a direct consequence of \eqref{def_fl}. Denote 
\begin{align*}
    I(x,v) & = \int_0^{\infty} |(2h(x,v)-h(x-y,v)-h(x+y,v)) \nu(y) | dy  \\
    & = \int_0^{\delta} |(2h(x,v)-h(x-y,v)-h(x+y,v)) \nu(y) | dy \\
    & \qquad \qquad + \int_{\delta}^{\infty} |(2h(x,v)-h(x-y,v)-h(x+y,v)) \nu(y) | dy \\
    & := I_1(x,v) +I_2(x,v) \,,
\end{align*}
for $\delta < 1$. 
By Taylor expansion we have 
\[
2h(x,v) - h(x+y,v)-h(x-y,v) = - \partial_{x}^2 h(\xi,v) y^2 \,,
\]
where $\xi \in (x-\delta,x+\delta)$. Since $s \in(0,1)$, we have 
$\int_0^{\delta} y^n \nu(y) \rd y < \infty$. The assumption that 
$h(x,v) \in L^1_v(\mathbb R) W^{2,1}_x(\mathbb R)$ then leads to 
$
\int_{\RR} |\partial_{x}^2 h(\xi,v)| \int_0^{\delta} y^2 \nu(y) \rd y \rd v < \infty\,.
$
On the other hand,  it's obvious that 
\[
\int_{\RR} \int_{\delta}^{\infty} |(2h(x,v)-h(x-y,v)-h(x+y,v)) \nu(y) | \rd y \rd v < \infty.
\]
Then we conclude the result by the Fubini's theorem. 

\end{proof}
\subsection{A novel micro-macro decomposition}
A typical approach in designing an asymptotic preserving method is via a micro-macro decomposition. That is, write 
\begin{equation} \label{MM00}
f(t,x,v) = \rho(t,x) \equilibrium(v) +g(t,x,v)\,,
\end{equation}
 where $\rho(t,x)$ is the macroscopic density, and $g$ is viewed as the microscopic fluctuation. See for instance \cite{lemou2008new, wang2016asymptotic, wang2019asymptotic, crouseilles2016numerical}. However, directly applying this decomposition to our case fails as it is not easy  to obtain $(-\Delta_x)^s$ when only $(-\Delta_v)^s$ appears in $\eqref{eqn:111}$. Inspired by the proof of Theorem~1 in \cite{cesbron2012anomalous}, we see that the operator $(-\Delta_x)^s$ and operator $\eps^{2s} (-\Delta_v)^s$ are related by considering a test function in variable $(x+\eps v)$. Therefore,  we propose the following novel micro-macro decomposition of distribution $f$:
\begin{equation}\label{eqn: decomp}
    f(t,x,v) = \eta(t,x, v) \equilibrium(v) +g(t,x,v)\,,
\end{equation}
where $\eta(t,x,v)$ takes the form
\begin{equation} \label{eqn:tr}
    \eta(t,x,v) = h(t,x+\eps v)\,, 
\end{equation} 
for some function $h(t,x)$, and $\equilibrium$ satisfies 
\begin{equation} \label{equilibrium000}
\partial_v (v \equilibrium) - (-\Delta_v)^s \equilibrium =0, \qquad \int_{\mathbb R} \equilibrium(v) \rd v = 1\,. 
\end{equation}
Here unlike the classical micro-macro decomposition \eqref{MM00}, we allow $\eta$ to depend on $v$, but intrinsically $\eta$ lives on a lower dimensional manifold than $f$ does, as required from \eqref{eqn:tr}. As a result, $\eta$ satisfies
\begin{equation} \label{eqn032}
\eps \partial_x \eta =  \partial_v \eta, \qquad  (-\Delta_v)^s \eta =\eps^{2s} (-\Delta_x)^s \eta \,.  
\end{equation}
Plug $\eqref{eqn: decomp}$ into $\eqref{eqn:111}$, we get:
\begin{equation} \label{eqn031}
\eps^{2s}\partial_t (\eta \equilibrium +g) + \eps v  \partial_x (\eta \equilibrium +g) = \partial_v  (v (\eta \equilibrium +g)) - (-\Delta_v)^s (\eta \equilibrium +g)\,.
\end{equation}
Using \eqref{equilibrium000},\eqref{eqn032} and Lemma~\ref{lemma:fl_comp}, the above equation simplifies to
\begin{equation}\label{eqn: decomp_sym}
    \eps^{2s}\partial_t (\eta \equilibrium +g) + \eps v  \partial_x g = \Lop^{s} (g)  - \eps^{2s} (-\Delta_x)^s \eta \equilibrium - I(\eta,\equilibrium) .
\end{equation}

To solve \eqref{eqn: decomp_sym}, we split it into the following system
\begin{equation} \label{split0}
    \begin{cases}{}
 \eps^{2s}\partial_t g + \eps v \partial_x g =\Lop^s(g) - I(\eta, \equilibrium)  \,,
 \\ \partial_t \eta = -(-\Delta_x)^s \eta\,,
    \end{cases}
\end{equation}
equipped with the initial condition
\begin{equation} \label{IC}
\eta_{in}(x,v) = \rho_{in}(x+\eps v)\, \qquad g_{in}(x,v) = f_{in}(x,v) - \eta_{in}(x,v) \equilibrium (v)\,.
\end{equation}
Upon solving \eqref{split0}, one can recover $f$ using \eqref{eqn: decomp}.

Note that the reduction of \eqref{eqn: decomp_sym} from \eqref{eqn031} is possible only when $\eta$ has the form \eqref{eqn:tr}. Therefore, it is important to make sure that such property is preserved along the dynamics of \eqref{split0}. To this end, we have the following lemma. 
\begin{lemma}
Let $\eta(t,x,v)$ solves
\begin{equation} \label{eta2}
\partial_t \eta = -(-\Delta_x)^s \eta\,, \qquad \eta(0,x,v) = \rho_{in}(x+\eps v)
\end{equation}
then there exists a function h(t,x) such that $\eta(t,x,v) = h(t,x+\eps v)$\,.
\end{lemma}
\begin{proof}
Let $h$ satisfies
\begin{equation}\label{eqnh}
\partial_t h = -(-\Delta_x)^s h\,, \qquad h(0, x,v) = h_{in}(x)\,.
\end{equation}
Then we claim that $\eta(t,x,v)= h(t,x+\eps v)$ is the solution to \eqref{eta2}. Indeed, denote $y = x + \eps v$ for any fixed $v$, we have
\begin{align*}
\partial_t h(t, x+\eps v) = \partial_t h(t, y) = -(-\Delta_y)^s h
&= -\int_{\mathbb R} \frac{h(t,y)-h(t,y')}{|y-y'|^{2s+1}} \rd y' \\
&=  - \int_{\mathbb R} \frac{h(t,x+\eps v)- h(t,x'+\eps v)}{|(x+\eps v)-(x'+\eps v)|^{2s+1}} \rd
(x'+\eps v) \\ &= -(-\Delta_x)^s h(t, x+\eps v)\,. 
\end{align*}
\end{proof}
\begin{remark}
From the above lemma, one sees that in order to solve the last equation in \eqref{split0}, one can solve the low dimensional problem \eqref{eqnh}, and then obtain $\eta$ by shifting $h$, i.e., $\eta(t,x,v) = h(t,x+\eps v)$.
\end{remark}

Next, we show that the system \eqref{split0} is energy stable. 
\begin{proposition} \label{prop:energy1}
If $(\eta, g)$ solves \eqref{split0} with initial data \eqref{IC}, then $f = \eta \equilibrium + g$ solves \eqref{eqn:111}. Both system has the energy dissipation property. That is, define the total energy 
\begin{equation} \label{Ef}
    E_f^2 = \int_{\RR} \int_\RR \frac{f^2}{\equilibrium} \rd v \rd x 
    = \int_{\RR} \int_\RR \frac{(\eta \equilibrium + g)^2}{\equilibrium} \rd v \rd x \,,
\end{equation}
then $\frac{d E_f}{dt} \leq 0$.
\end{proposition}
\begin{proof}
It is easy to show that if $(\eta, g)$ solves \eqref{split0} with initial data \eqref{IC}, then by directly adding the two equation,  $f = \eta \equilibrium + g$ solves \eqref{eqn:111}. The energy dissipation of the original system \eqref{eqn:111} follows from Proposition 2.1 in \cite{cesbron2012anomalous}. One just needs to check the same property for the split system \eqref{split0}. To this end, multiply the first equation in \eqref{split0} by $\frac{g}{\equilibrium}$, the second equation by $\frac{\eta}{\equilibrium}$, integrate in $x$ and $v$, and add them together, we get (here we omit the integration domain, which is $\RR$ for both $x$ and $v$):
\begin{align*}
& \eps^{2a} \partial_t \int \int \left[  \half \frac{g^2}{\equilibrium} + \half \eta^2 \equilibrium + g \eta \right]  \rd x \rd v + \eps \int \int v \partial_x g \frac{g}{\equilibrium} \rd x \rd v 
\\ &= \int \int  \Lop^s (g) \left(\frac{g}{\equilibrium} + \eta  \right)  - I(\eta, \equilibrium) \left(\frac{g}{\equilibrium} + \eta  \right)  - (\eta \equilibrium + g) (-\Delta_v)^s \eta - \eps v \partial_x g \eta  \rd x \rd v
\\ & = \int \int (\Lop^s(g) + \Lop^s(\eta \equilibrium)) \left(\frac{g}{\equilibrium} + \eta  \right)  - \partial_v (v \eta \equilibrium) \left(\frac{g}{\equilibrium} + \eta  \right)   + \eta (-\Delta_v)^s \equilibrium \left(\frac{g}{\equilibrium} + \eta  \right)  - \eps v  \partial_x g \eta \rd x \rd v
\\ &= \int \int (\Lop^s(g) + \Lop^s(\eta \equilibrium)) \left(\frac{g}{\equilibrium} + \eta  \right)  - v \equilibrium \partial_v \eta  \left(\frac{g}{\equilibrium} + \eta  \right) - \eps v \partial_x g \eta \rd x \rd v\,,
\end{align*}
where the second equality uses Lemma~\ref{lemma:fl_comp} and the third equality uses the fact that $\equilibrium$ satisfies $\Lop^s(\equilibrium) = 0$. Since $\partial_v \eta = \eps \partial_x \eta$ and $\int\int v \partial_x \eta g + \partial_x g \eta \rd x \rd v = \int \int v\equilibrium \partial_x \eta \eta \rd x \rd v = \int \int v \partial_x g \frac{g}{\equilibrium} \rd x \rd v =0$,
we immediately get 
\[
\eps^{2a} \partial_t \int \int \left[  \half \frac{g^2}{\equilibrium} + \half \eta^2 \equilibrium + g \eta \right]  \rd x \rd v = \int \int (\Lop^s(g) + \Lop^s(\eta \equilibrium)) \left(\frac{g}{\equilibrium} + \eta  \right)  \leq 0\,.
\]
\end{proof}

Moreover, we can bound the energy for $\eta$ and $\rho$ separately. 
\begin{proposition} \label{prop:energy2}
If $(\eta, g)$ solves \eqref{split0} with initial data \eqref{IC}, then
\begin{equation} \label{Eetag}
E_\eta^2 = \int_\RR\int_\RR \eta^2 \equilibrium \rd x \rd v, \qquad
E_g^2 = \int_\RR \int_\RR \frac{g^2}{\equilibrium} \rd x \rd v
\end{equation}
are both uniformly bounded in time. 
\end{proposition}
\begin{proof}
Multiply the second equation in \eqref{split0} with $\eta \equilibrium$ and integrate in $x$ and $v$, we get
\begin{align*}
    \frac{\eps^{2s}}{2}  \partial_t \int \int \eta^2 \equilibrium \rd x \rd v = - \int \int \eta(-\Delta_x)^s \eta  \rd x \equilibrium \rd v 
\end{align*}
From the definition \eqref{def_fl}, it is straightforward to see that $\int \eta (-\Delta_x)^s \eta \rd x = \half \int \int \frac{(\eta(x)-\eta(y))^2}{|v-w|^{1+2s}} rd x \rd y  \geq 0$. Therefore, $\frac{d}{dt} E_\eta \leq 0$. Then from 
\[
E_g^2 = \int \int \frac{(f-\eta \equilibrium)^2}{\equilibrium} \rd x \rd v \leq 2 \int \int  \frac{f^2 + \eta^2 \equilibrium^2}{ \equilibrium} \rd x \rd v = 2 E_f^2 + 2 E_\eta^2\,,
\]
we get the uniform boundedness of $E_g$. 
\end{proof}

Numerically, we propose the following semi-discrete scheme to \eqref{split0}:
\begin{subequations}\label{eqn: semi_scheme}
\begin{numcases}{}
\frac{\eps^{2s}}{\Delta t} (g^{*}-g^n)  = \Lop^s(g^*) - \gamma g^* - I(\eta^n,\equilibrium)  , \label{semis1} \\
 \frac{\eps^{2s}}{\Delta t}(g^{n+1}-g^*)  + \eps v  \partial_x g^{n+1} = \gamma  g^{n+1}  \,, \label{semis2} \\
 \frac{1}{\Delta t} (\eta^{n+1}-\eta^n)  =  - (-\Delta_x)^s \eta^n \,, \label{semi3}
\end{numcases}
\end{subequations}
where $\gamma$ is a positive constant. 
As opposed to directly applying an implicit-explicit discretization to the first equation in \eqref{split0}, i.e., 
\[
\frac{\eps^{2s}}{\Delta t} (g^{n+1}-g^n) + \eps v \partial_x g^{n+1}  = \Lop^s( g^{n+1}) - I(\eta^n,\equilibrium) \,,
\]
we conduct an operator splitting here, due to three reasons. One is that the L\'evy-Fokker-Planck operator $\Lop^s$ has nonzero null space, and therefore the inversion in this equation will become stiff for small $\eps$ (see also Table~\ref{table:cda}). The augmented term $-\gamma g^*$ will then shift the spectrum of the to-be-inverted operator and therefore remove the ill-conditioning. The second is that we need the magnitude of $g$ to remain small for small $\eps$ in order to preserve the asymptotic property, and this requirement is fulfilled in \eqref{semis2}, which warrants $g^{n+1}$ to be of order $\eps^{2s}$ as $\eps \rightarrow 0$ (More details is shown in the proof of Proposition \ref{prop:asy}). The third is the computational efficiency. Thanks to the splitting, one no longer needs to invert operators in $x$ and $v$ simultaneously. Instead, only an inversion in $v$ is needed in solving \eqref{semis1}, whereas in \eqref{semis2} only an inversion in $x$ is needed, and this inversion can be efficiently accomplished via either sweeping or the fast Fourier transform.

The rest of this subsection is devoted to proving the asymptotic property of \eqref{eqn: semi_scheme}. First let us introduce the following lemma that describes the smoothing effect of fractional diffusion equation.
\begin{lemma}\label{lemma:fd_reg}
Consider the initial value problem 
\begin{equation} \label{eqn0108}
\partial_t u + (-\Delta_x)^s u = 0 \,\qquad 
u(0,x) = u_{in}(x) \,.
\end{equation}
If $u_{in}(x) \in W^{2,1}(\RR) \cap C^2(\RR)$, then $u(t,\cdot) \in  W^{2,1}(\RR) \cap C^{\infty}(\RR) $ for $t>0$.
\end{lemma}
\begin{proof}
Using the Fourier transform for $x$, one writes down the solution to \eqref{eqn0108} as
\begin{equation*}
    \hat u(t,\xi) = \hat u_{in}(\xi) e^{-|\xi|^{2s} t} : = \hat u_{in} (\xi) \hat K_s(t,\xi)\,.
\end{equation*}
Changing back to $x$, one has 
\[
u(t,x) = \int_{\RR} K_s(t,x-y)u_0(y) \rd y\,, \qquad K_s(t,x) = \frac{1}{t^{1/2s}} F \left(\frac{|x|}{t^{1/2s}} \right)\,,
\]
where $F$ is positive and decreasing, and it behaves like $F(r) \sim r^{-(1+2s)}$ at infinity (see also \cite{vazquez2014recent} for further discussion). It is easy to see $K_s(t, \cdot) \in L^1 (\mathbb R) \cap C^\infty(\mathbb R)$ for $t>0$. Therefore $u(t,\cdot) \in C^\infty(\mathbb R)$. Further, by Young's convolution inequality, we have 
\[
\|u(t,\cdot) \|_1 \leq \|K_s \|_1 \|u_0\|_1\,,\quad 
\|\partial^2_x u(t,\cdot) \|_1 \leq  \|K_s \|_1 \| \partial^2_x u_0\|_1\,.
\]
\end{proof}

The proposition on the asymptotic property of the splitting scheme is in order. 
\begin{proposition} \label{prop:asy}
Consider system $\eqref{eqn:111}$ with initial data $\average{f_{in}} \in W^{2,1}(\RR)$. Let $\eta^n$ and $g^n$ be the solution to \eqref{eqn: semi_scheme}, where  $\eta^0 = \rho_{in}(x+\eps v)$ and $g^0 = f_{in} - \eta_{in} \equilibrium$. Then the numerical solution 
\begin{equation*} 
\rho^n = \average{f^n} = \average{\eta^n \equilibrium + g^n} 
\end{equation*}
satisfies 
\begin{equation} \label{eqn01082}
  \frac{\rho^{n+1}- \rho^n}{\Delta t}  =  (-\Delta_x)^s \rho^n  
\end{equation}
as $\eps \rightarrow 0$.
\end{proposition}

\begin{proof}
From the reconstruction formula, we have
\begin{align*}
    \frac{\rho^{n+1}- \rho^n}{\Delta t} &  = \average{\frac{\eta^{n+1}-\eta^n}{\Delta t} \equilibrium } + \average{\frac{g^{n+1}-g^n}{\Delta t}}  \\
    & = - \average{(-\Delta_x)^s \eta^n \equilibrium} + \average{\frac{g^{n+1}-g^n}{\Delta t}}  \\
    & =  (-\Delta_x)^s \average{f^{n}-g^{n}} + \average{\frac{g^{n+1}-g^n}{\Delta t}} , \\
    & = (-\Delta_x)^s \rho^n - (-\Delta_x)^s \average{g^n}  + \average{\frac{g^{n+1}-g^n}{\Delta t}} \,,
\end{align*}
where the third equality uses lemma~\ref{lemma: interchange} and lemma~\ref{lemma:fd_reg}, namely, $\average{(-\Delta_x)^s \eta ^n \equilibrium} = (-\Delta_x)^s \average{\eta^n \equilibrium}$. Then to show \eqref{eqn01082}, it amounts to show that the magnitude of $g^n$ vanishes as $\eps$ approaches zero. First, from \eqref{semis1}, one sees that 
\[
\left( \frac{\eps^{2s}}{\Delta t} + \gamma + \Lop^s \right) g^* =  \frac{\eps^{2s}}{\Delta t} g^n - I(\eta^n, \equilibrium)\,.
\]
From the contractive estimate in Corollary 3.1 of \cite{biler2003generalized} and Hille-Yosida Theorem, we have, for positive $\gamma$, 
\begin{equation} \label{estimate01}
    \|g^*(x,\cdot)\|_{L^\infty_v} \lesssim \|I(\eta^n, \equilibrium) \|_{L_v^\infty} + \frac{\eps^{2s}}{\Delta t} \|g^n\|_{L_v^\infty}\,.
\end{equation}
Further, from \eqref{semis2}, we have, for $v<0$,
\begin{align*}
g^{n+1}(x,v) = \frac{\eps^{2s-1}}{v \Delta t  } \int_{-\infty}^x e^{ \frac{ \frac{\eps^{2s}}{\Delta t}-\gamma }{\eps v} (y-x)} g^*(y,v) \rd y 
\end{align*}
Therefore,
\begin{align} \label{estimate02}
    |g^{n+1}(x,v)| \leq \frac{\eps^{2s-1}}{|v| \Delta t  }  \|g^*(\cdot, v)\|_{L^\infty_x} \int_{-\infty}^x e^{ \frac{ \frac{\eps^{2s}}{\Delta t}-\gamma }{\eps v} (y-x)} \rd y  
    \leq \frac{2 \eps^{2s}}{\Delta t}  \|g^*(\cdot, v)\|_{L^\infty_x} \,.
\end{align}
The case with $v>0$ can be estimated similarly. Then a combination of \eqref{estimate01} and \eqref{estimate02} immediately leads to $\|g^{n+1}\|_{L^\infty_{x,v}} \lesssim \mathcal{O}(\eps^{2s})$. 
\end{proof}

\subsection{Fully discrete scheme and implementation}
Now we briefly discuss the implementation of scheme $\eqref{eqn: semi_scheme}$. Using the same notation as mentioned at the beginning of this section, we denote the numerical approximations $\eta^n_{i,j}$, $g^n_{i,j}$ as $ \eta^n_{i,j} \approx \eta(t^n,x_i,q_j)$ and $g^n_{i,j} \approx g(t^n,x_i,q_j)$, with $0\leq i \leq N_x-1$, $0\leq j \leq N_v-1$. For a fixed $x$-index $i$, let 
\begin{align*}
& {\bf \eta^n_i} = (\eta^n_{i,1}, \eta^n_{i,2}, \cdots, \eta^n_{i,N_v-1})^T, \qquad
{\bf g^n_i} = (g^n_{i,1}, g^n_{i,2}, \cdots, g^n_{i,N_v-1})^T \,, \qquad
\\ &
{\bf \mathcal M} = (\mathcal M_1, \cdots, \mathcal M_{N_v-1})^T\,, \qquad 
{\bf \eta^n_i \mathcal M} = (\eta^n_{i,1}\mathcal M_1, \cdots, \mathcal \eta^n_{i,N_v-1} M_{N_v-1})^T\,.
\end{align*}
For a fixed $q$-index $j$, let 
\[
{\bf \eta^n_j} = (\eta^n_{1,j}, \eta^n_{2,j}, \cdots, \eta^n_{N_x-1,j})^T, \qquad
{\bf g^n_j} = (g^n_{1,j}, g^n_{2,j}, \cdots, g^n_{N_x-1,j})^T\,.
\]

First we compute $I(\eta^n, \equilibrium)$. According to its definition
\[
I(\eta^n, \equilibrium) = (-\Delta_v)^s (\eta^n \equilibrium) - \equilibrium (-\Delta_v)^s \eta^n  -  \eta^n (-\Delta_v)^s \equilibrium , 
\]
this is simply accomplished by applying $\Lmat^s$ defined in \eqref{LsD} to ${\bf \eta_i^n \equilibrium}$, ${\bf \eta_i^n}$, and ${\bf \equilibrium}$, respectively, {\it for a fixed spatial index $i$}. Then to treat the spatial discretization in \eqref{semis2} and \eqref{semi3}, we use the Fourier spectral method with periodic boundary condition. Note that the transport term in \eqref{semis2} is treated implicitly, so the stability is guaranteed. Moreover, unlike in $v$ direction, where one needs to consider a fat tail distribution due to the kernel of $\Lop^s$, we can consider a sufficiently decaying profile in $x$ and therefore the Fourier spectral method can be used here. 

To summarize, for given $\eta^n_{i,j}$ and $g^n_{i,j}$, we have

\textbf{Step 1} Compute $\bf g_i^*$ by
\[
{\bf g_i^*} = \left[ \left(\frac{\varepsilon^{2s}}{\Delta t} + \gamma \right) \Imat - \Pmat^s \right]^{-1} \left(\frac{\varepsilon^{2s}}{\Delta t} {\bf g_i^n} - I({\bf \eta_i}^n \equilibrium)\right) \,, \qquad {\bf i} = 0, 1, \cdots, N_x-1  \,;
\]

\textbf{Step 2} Compute ${\bf g_{j}^n}$ by inverting 
\[
\eps^{2s} \frac{{\bf \hat g_{j}^{n+1}} - {\bf \hat g_{j}^* } }{\Delta t}  + \eps v(q_j) \Dmat_x    {\bf \hat g_{j}^{n+1}}  = \frac{1}{2} {\bf \hat g_{j}^{n+1}}\,, \qquad {\bf j} = 0,1, \cdots, N_v-1\,,
\]
where ${\bf \hat g_{j}}$ is the Fourier transform of ${\bf g_j}$, and $\Dmat_x$ is the diagonal matrix with elements $k = 0, 1, \cdots, N_x-1$.

\textbf{Step 3} Compute ${\bf \eta_0^{n+1}}$ via 
\[
{\bf \hat \eta_0^{n+1}} = (\Imat + \Delta t \Dmat_s)^{-1} {\bf \hat\eta_0^n},
\qquad {\bf j} = 0,1, \cdots, N_v-1\,,
\]
where ${\bf \hat \eta_0}$ is the Fourier transform of ${\bf \eta_0}$, and $\Dmat_s$ is the diagonal matrix with elements $-|k|^{2s}$ with $k$ ranging from 0 to $N_x-1$.

\section{Numerical examples}
In this section, we present several numerical examples to check the accuracy and efficiency of schemes $\eqref{eqn: homo_scheme_ori}$ and $\eqref{eqn: semi_scheme}$.  Periodic boundary condition is always used in $x$ direction. Unless otherwise specified, we choose $L_v =3$ and $l_{lim}=300$ in computing \eqref{eqn: fl_fourier_approx}, and $\gamma=1$ in \eqref{eqn: semi_scheme}. 

\subsection{Computation of $(-\Delta)^s f$}
We first check the performance of the pseudospectral method in computing  $(-\Delta)^s f$ via \eqref{LsD}. Two specific examples will be considered here: an exponential decay function and a polynomial decay function, both of which have exact form when taken the fractional Laplacian. 
\begin{align}
1) \qquad & f(v)=(1+v^2)^{-\frac{1-2s}{2}},  \quad s \in (0,0.5),  \label{eqn: pl_decay}
\\ &  (-\Delta_v)^s f(v) = - 2^{2s}\Gamma \left(\frac{1+2s}{2} \right) \Gamma \left(\frac{1-2s}{2} \right)^{-1} (1+v^2)^{-\frac{1+2s}{2}}\,. \nonumber 
\\ 2) \qquad & f(v)=e^{-v^2}\,, \label{eqn: exp_decay}
\\ & (-\Delta_v)^s f(v) = - \frac{1}{\sqrt{\pi}}2^{2s}\Gamma \left(\frac{1+2s}{2} \right)    {}_1F_1 \left(\frac{1+2s}{2},0.5;-v^2 \right)\,. \nonumber 
\end{align}
Here $\Gamma$ is the gamma function and $_1F_1$ is the hypergeometric function.

First, a comparison of the numerical solutions with exact solutions is gathered in Fig.~\ref{pl_acc} and \ref{ep_acc}, for \eqref{eqn: pl_decay} and \eqref{eqn: exp_decay}, respectively, where good agreements are observed in both cases. 
\begin{figure}[H]
\centering
{\includegraphics[width=0.45\textwidth]{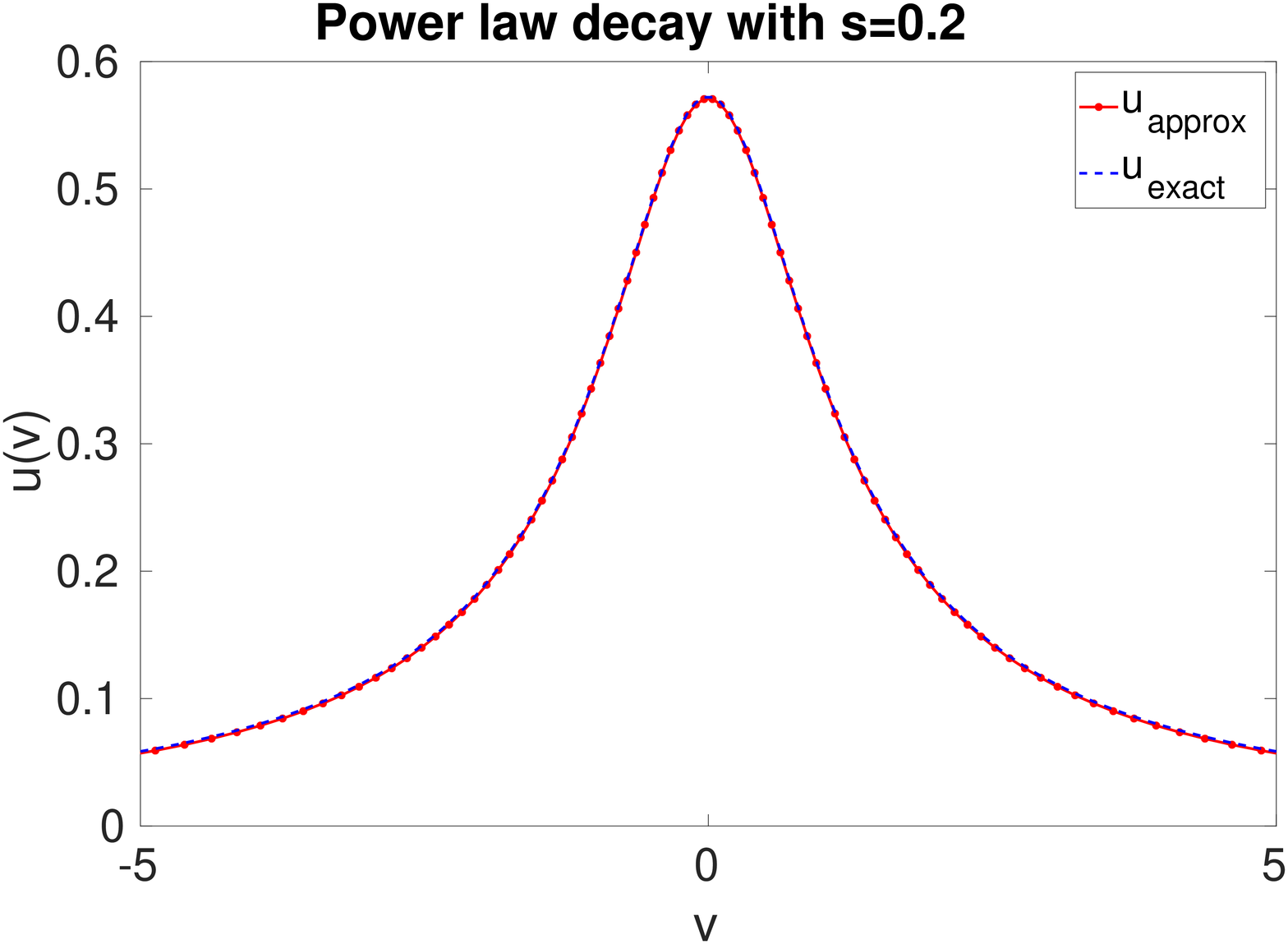}}
{\includegraphics[width=0.45\textwidth]{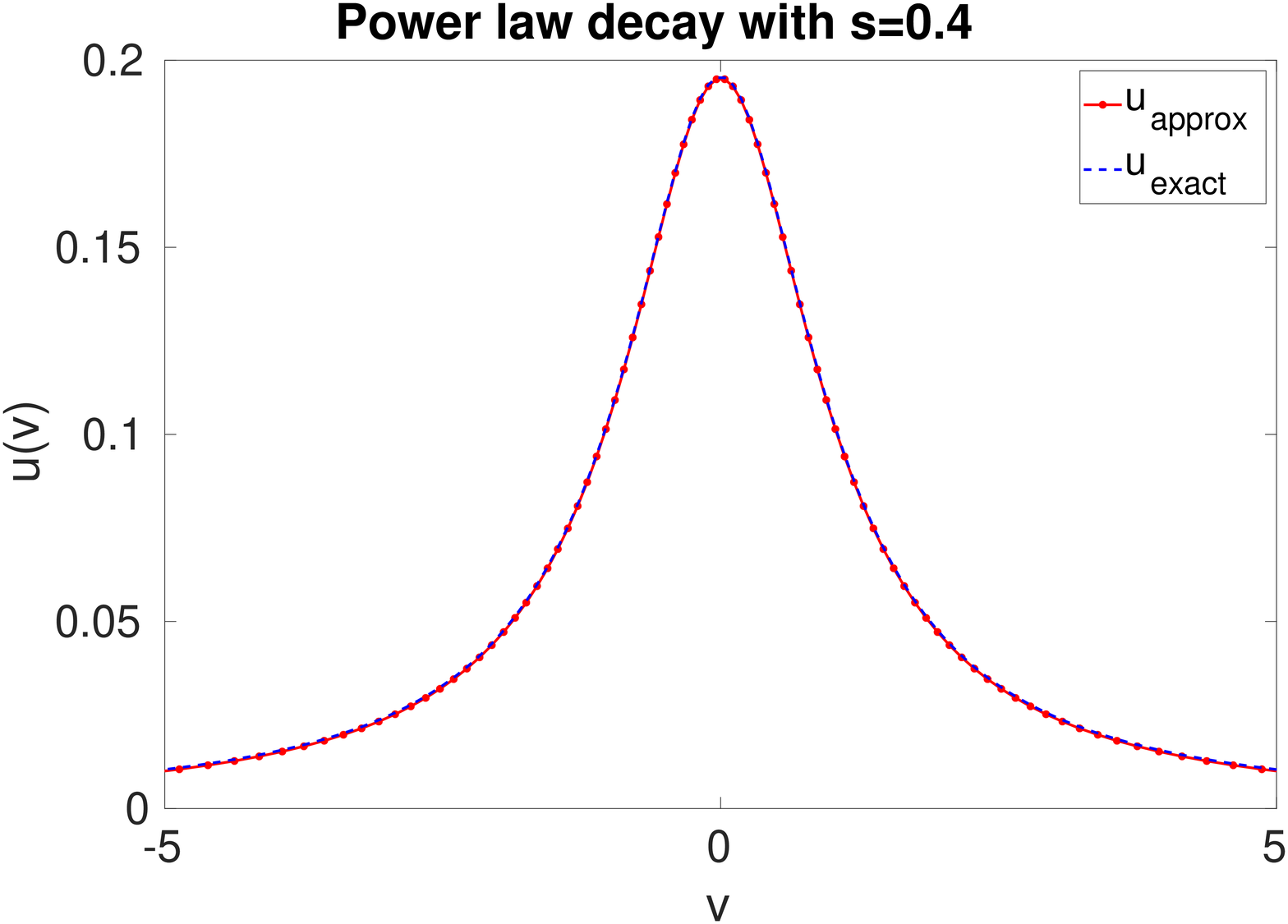}}
\caption{Comparison of numerical approximation and exact expression of $(-\Delta)^s f$ for $f$ in \eqref{eqn: pl_decay}. $N_v=128$ are used for both cases.}
\label{pl_acc}
\end{figure}

\begin{figure}[H]
\centering
{\includegraphics[width=0.31\textwidth]{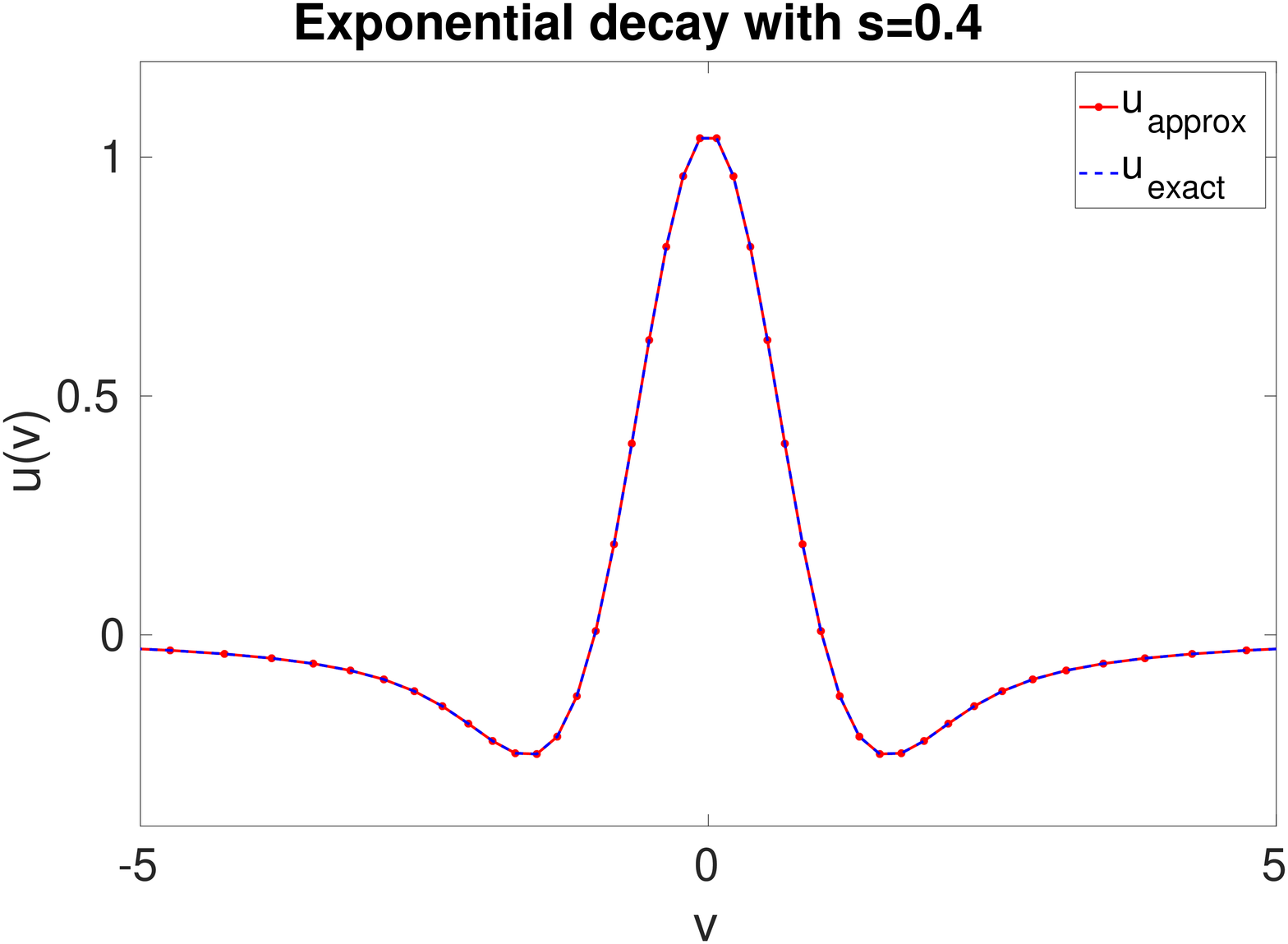}}
{\includegraphics[width=0.31\textwidth]{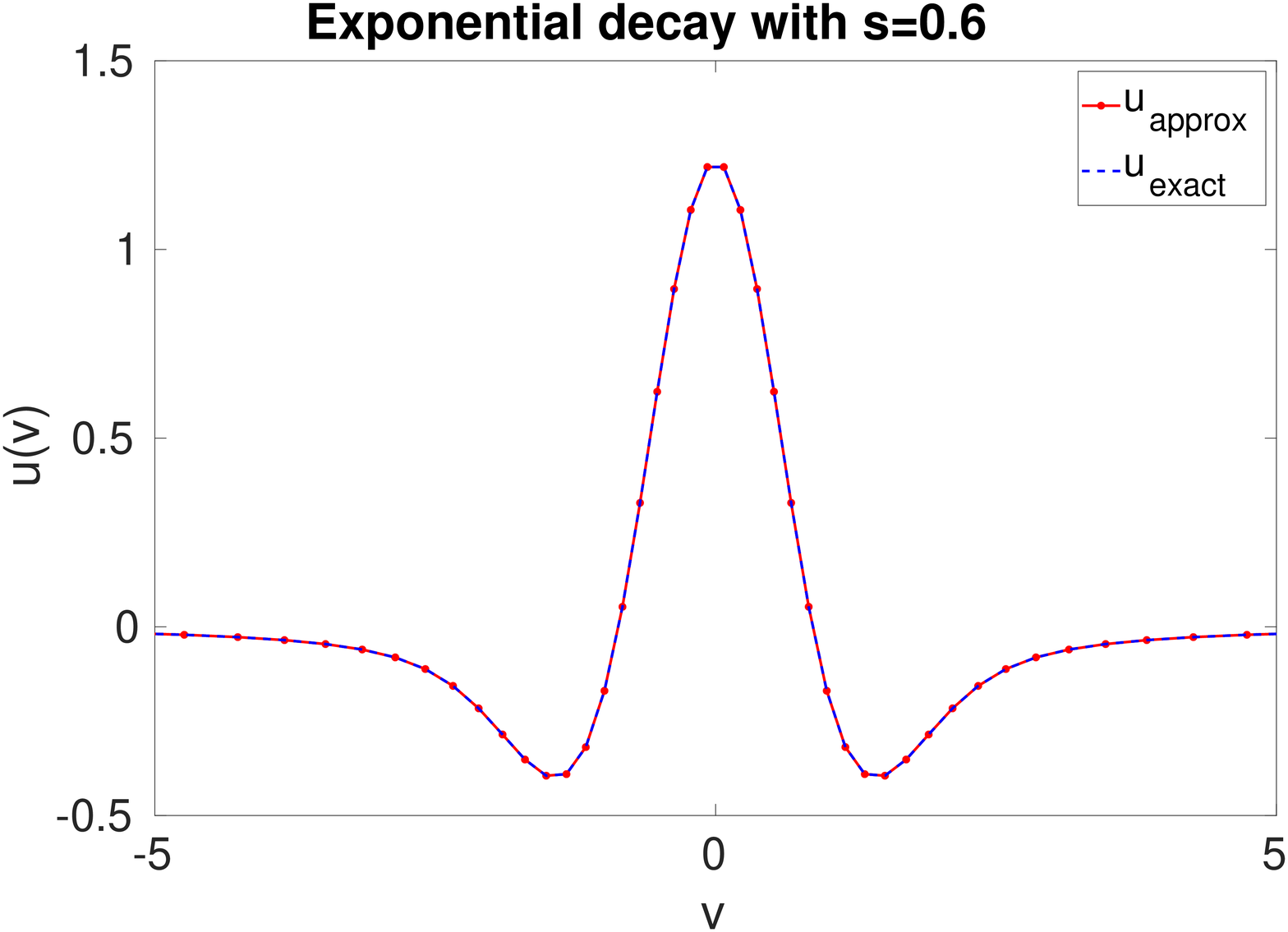}}
{\includegraphics[width=0.31\textwidth]{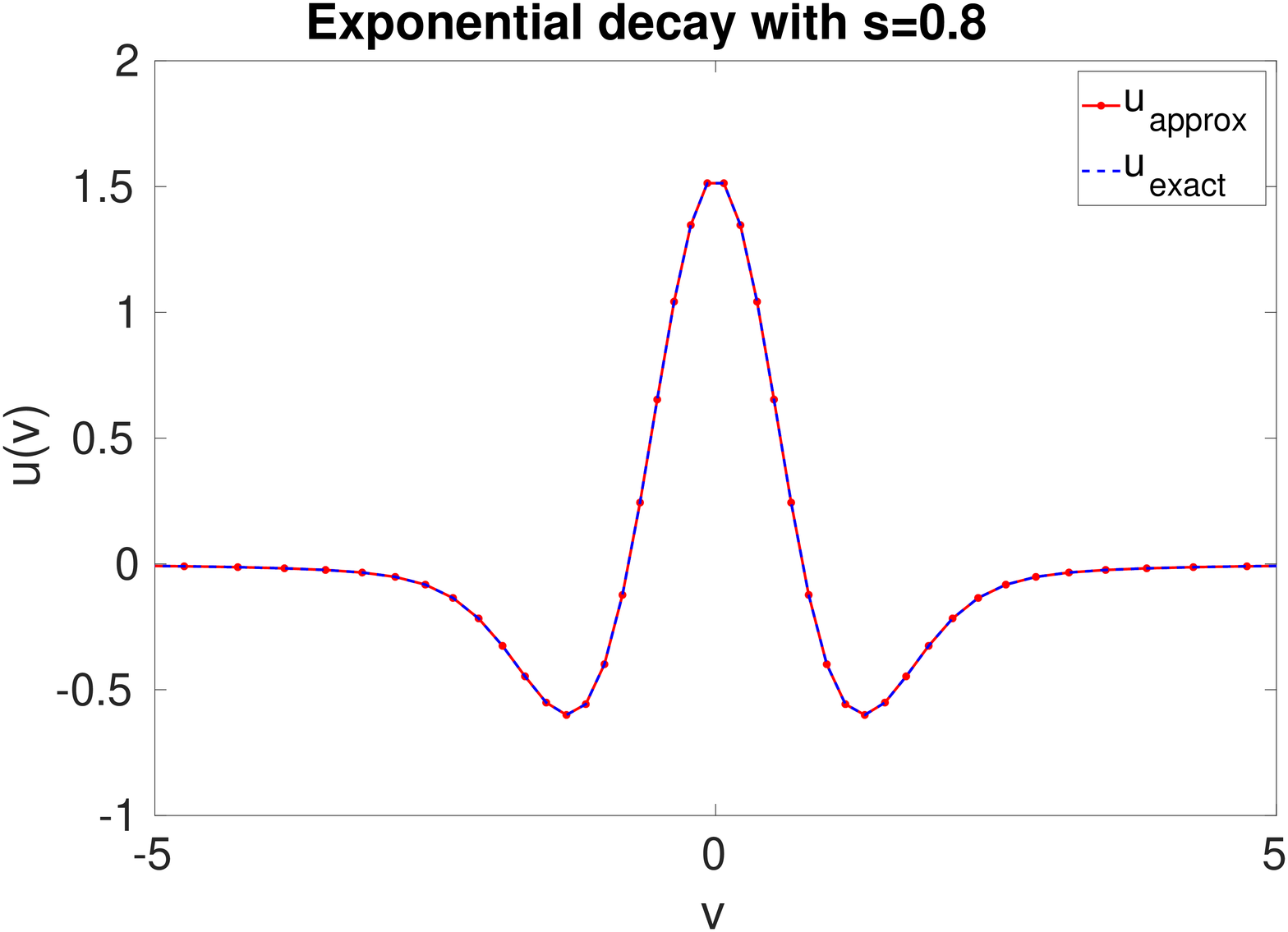}}
\caption{ Comparison of numerical approximation and exact expression of $(-\Delta)^s f$ for $f$ in \eqref{eqn: exp_decay}. $N_v=64$ are used for all cases. }
\label{ep_acc}
\end{figure}

Second, we show relationship between accuracy and the number of modes $N_v$. Given fixed $L_v=3$ and $l_{lim}=300$, Fig.~\ref{fig: error vs Nv} displays the error versus $N_v$ for different $s$, with exponential decay function on the left, and power law decay function on the right. The error is measured in $l^{\infty}$ norm, i.e.
\begin{equation} \label{error}
e_{\infty} = \max_j \| u_j - (-\Delta_v)^s f(v_j)   \|_{\infty}  
\end{equation}
where $u_j$ is the numerical approximation of $(-\Delta_v)^s f$ at $v_j$. One 
 sees that a small number of mode is adequate for the exponential decay case, whereas for the power law decay case, increasing the number of modes leads to better approximation. Moreover, a comparison between these two figures also gives a visual indication on why computing the fractional Laplacian of a slow decaying function is significantly harder than a fast decaying function, as a lot more modes are needed to reach an even lower accuracy criteria. 

\begin{figure}[H]
\centering
{\includegraphics[width=0.45\textwidth]{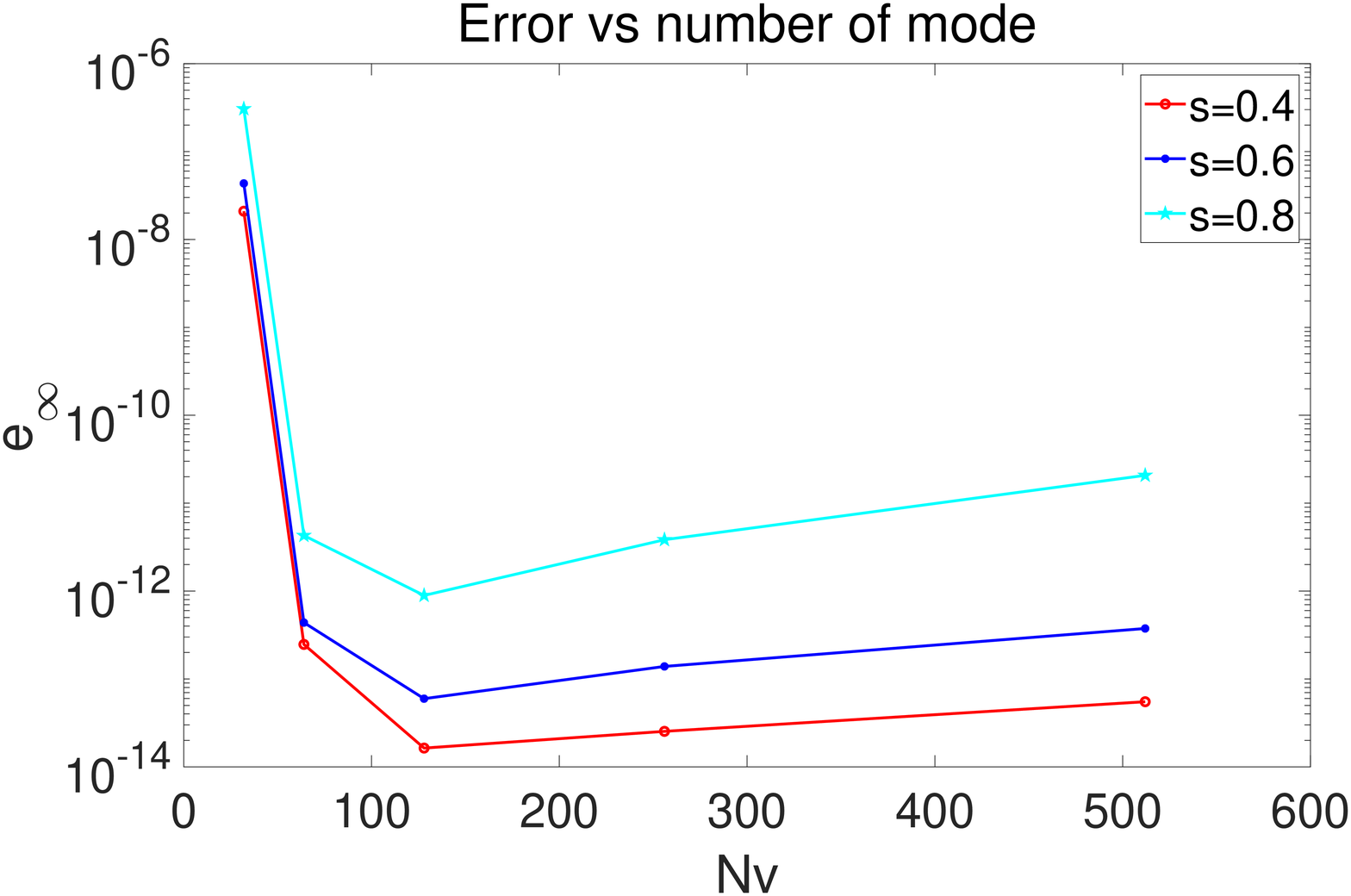}}
{\includegraphics[width=0.45\textwidth]{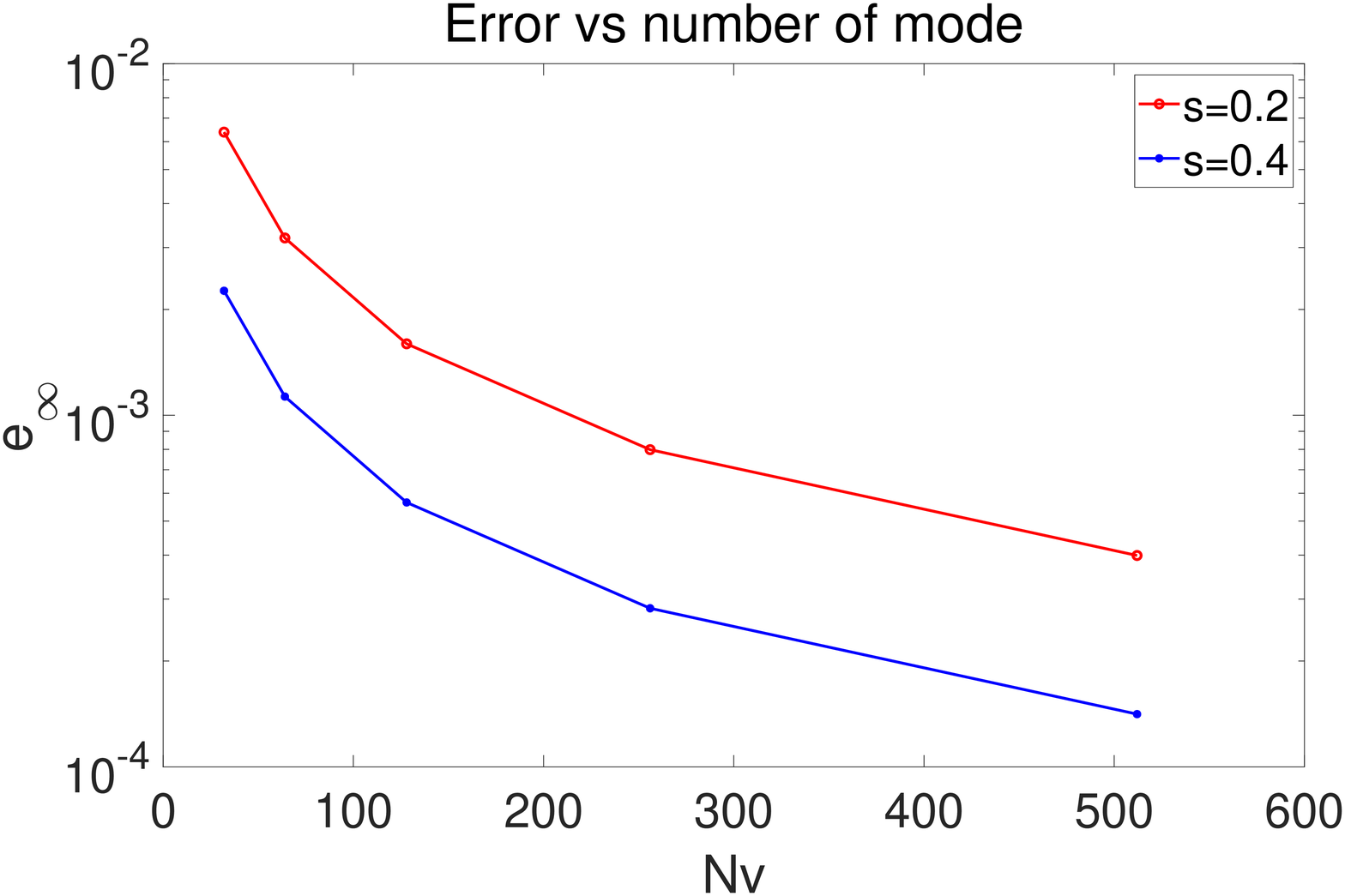}}
\caption{Error \eqref{error} versus $N_v$ with different $s$. The left figure is for exponential decay function $\eqref{eqn: pl_decay}$, and the right is for power law decay function $\eqref{eqn: exp_decay}$.}
\label{fig: error vs Nv}
\end{figure}

\subsection{Spatially homogeneous case}
We restrict our attention to the spatially homogeneous case in this section and consider the following specific example: 
\begin{equation}\label{eqn: numeriacl_ex}
 \partial_t f  = \partial_v (v f) - (-\Delta_v)^s f\,, \qquad 
       f(0,v) =e^{-v^2} \,.
\end{equation}
In order to check the performance of the scheme, two properties of the solution will be considered: 1) convergence towards equilibrium in the long time limit and 2) mass conservation.

\subsubsection{Long time behavior} \label{sec:long time}
As pointed out in \cite{gentil2008levy}, $f(t,v)$  will converge toward the equilibrium $\equilibrium$ exponentially fast. To observe this dynamics numerically, we compute the relative entropy $\int_{\mathbb R} f \ln  \frac{f}{\equilibrium}  \rd v$ as
\begin{equation} \label{eqnHn}
 H^n = \sum_{j=0}^{N_v-1} f^n(q_j) \ln \frac{f^n(q_j)}{\equilibrium(q_j)} w_j , \qquad w_j = \frac{L_v}{(\sin (q_j))^2}\,.   
\end{equation}
The numerical equilibrium, denoted as $f_{\infty}$, is obtained when the variation in the solution is negligible, i.e., 
\begin{equation} \label{Nequi}
\sum_j | f_j^{n+1} - f_j^{n} |w_j \Delta q \leq \delta, \qquad w_j = \frac{L_v}{(\sin (q_j))^2}\,,
\end{equation}
and then set $f_{\infty} = f^{n+1}$. Here $\delta$ is a small parameter and chosen to be $10^{-6}$ in our examples.

We first test the case with $s=0.5$, where an explicit form of equilibrium is available: $\equilibrium(v)=\pi^{-0.5}(1+v^2)^{-1}$. The results are collected in Fig~\ref{fig:tail_zp5}. In the left and middle figures, one sees that the numerical equilibrium coincides with analytic one, both in the bulk area and in the tail. On the right, the evolution of relative entropy \eqref{eqnHn} in time is plotted, and exponential rate of convergence is confirmed.  

\begin{figure}[!ht]
\centering
{\includegraphics[width=0.3\textwidth]{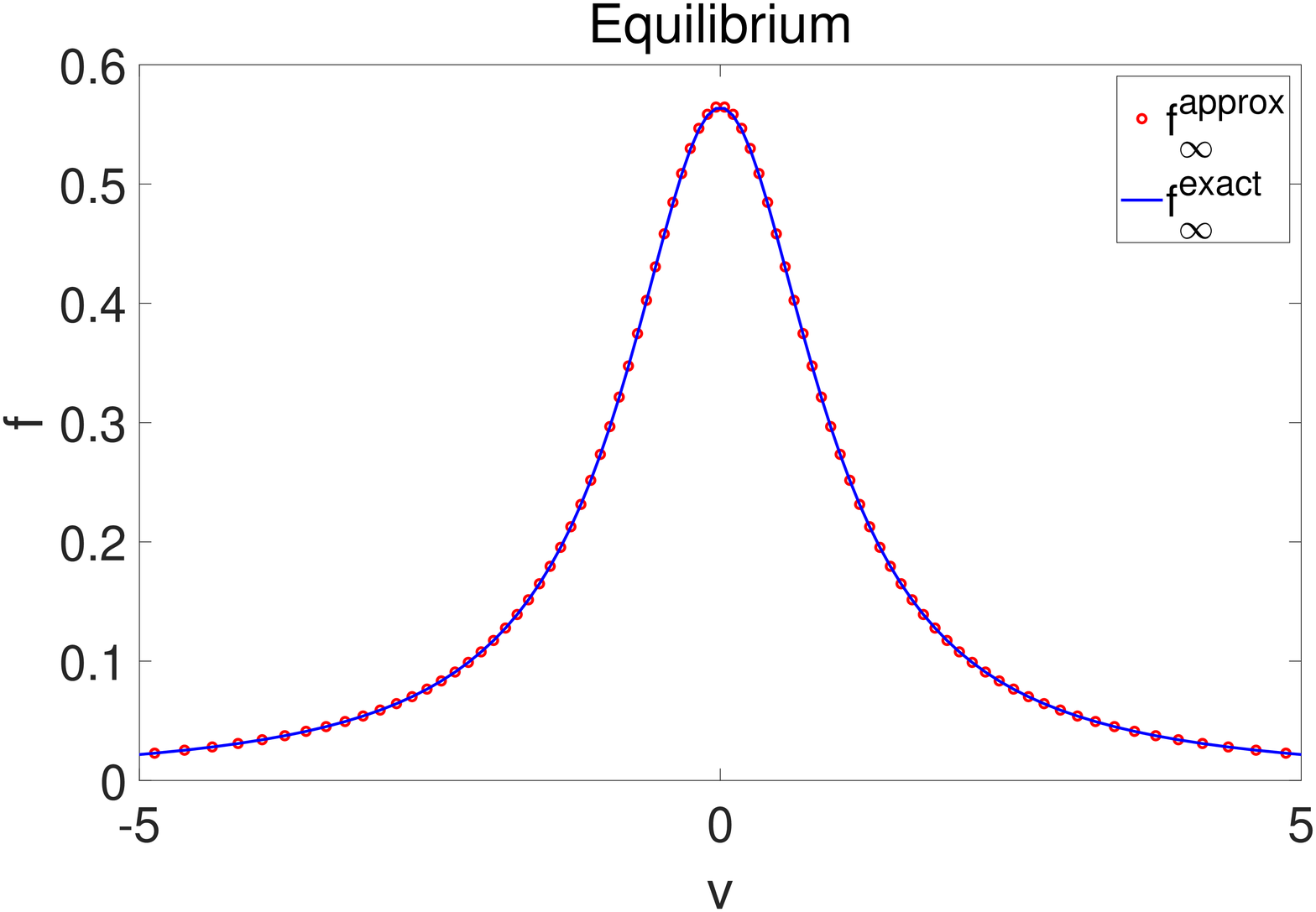}}
{\includegraphics[width=0.3\textwidth]{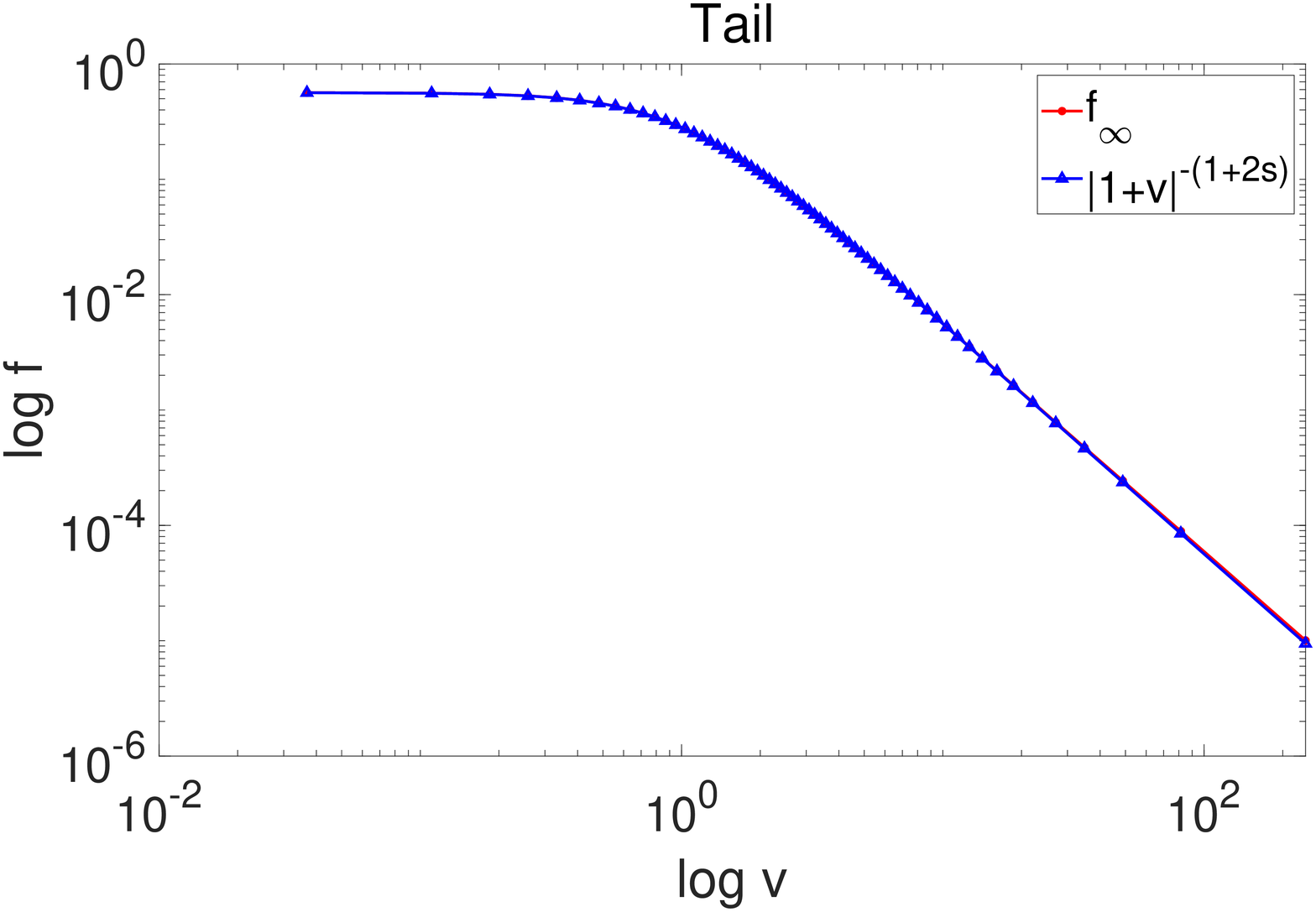}}
{\includegraphics[width=0.3\textwidth]{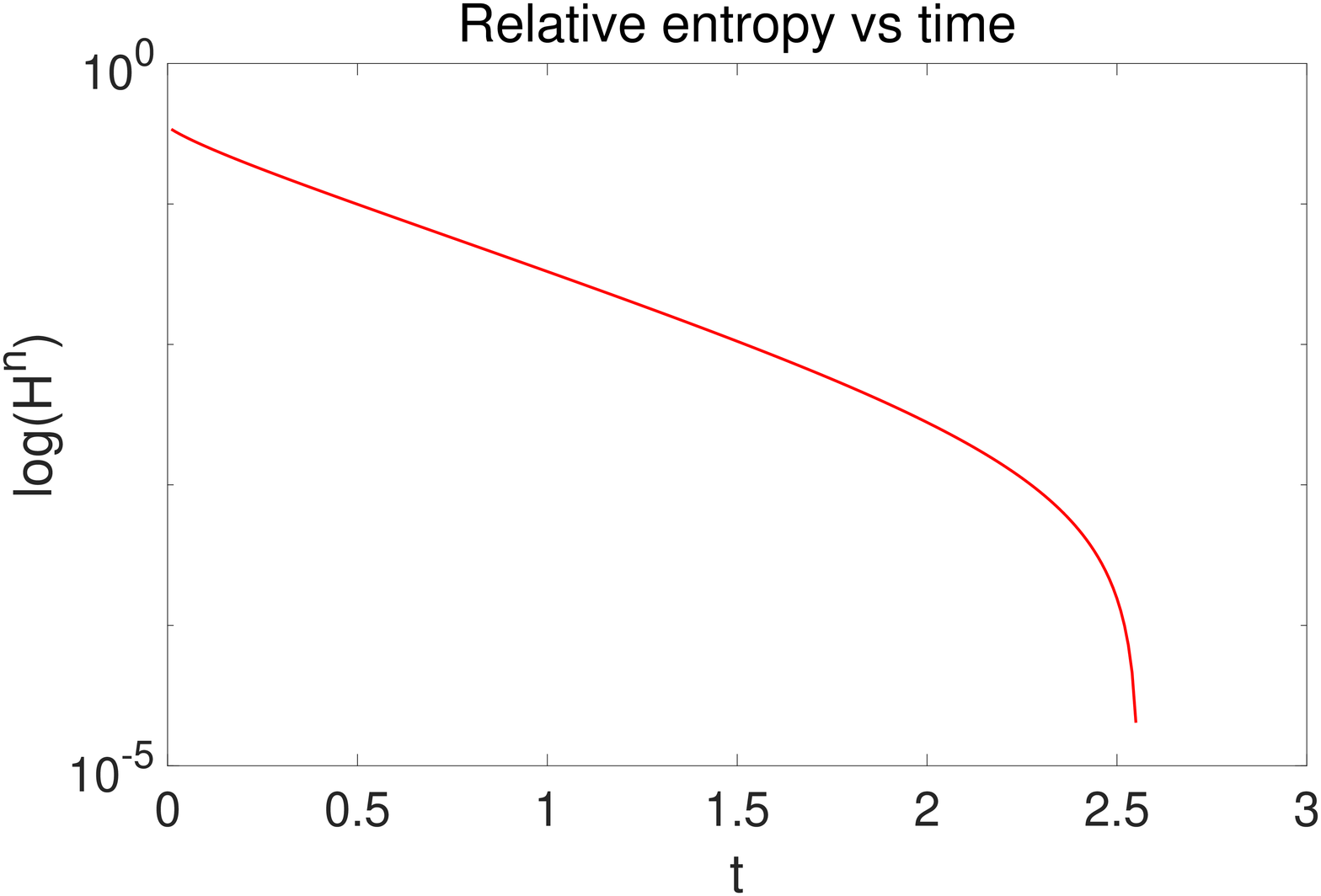}}
\caption{Computation of spatially homogeneous case \eqref{eqn: numeriacl_ex} with $s=0.5$. Left: a comparison between exact equilibrium and numerical equilibrium after converging. Middle: the tail of the equilibrium. Right: exponential convergence of the relative entropy \eqref{eqnHn}. Here $\Delta t = 0.01$ and $N_v=128$, and numerical equilibrium is reached at $t = 6.47$. }
\label{fig:tail_zp5}
\end{figure}

Next, we test two other cases $s=0.6, ~0.8$, where no explicit form of equilibrium is available, therefore we only check the tail behavior, as predicted in \eqref{eqn:equilibrium}. As shown in Fig.~\ref{fig:tail_zp5}, correct power law decay rate at tail is captured numerically and exponential convergence toward equilibrium is observed.

\begin{figure}[!ht]
\centering
{\includegraphics[width=0.4\textwidth]{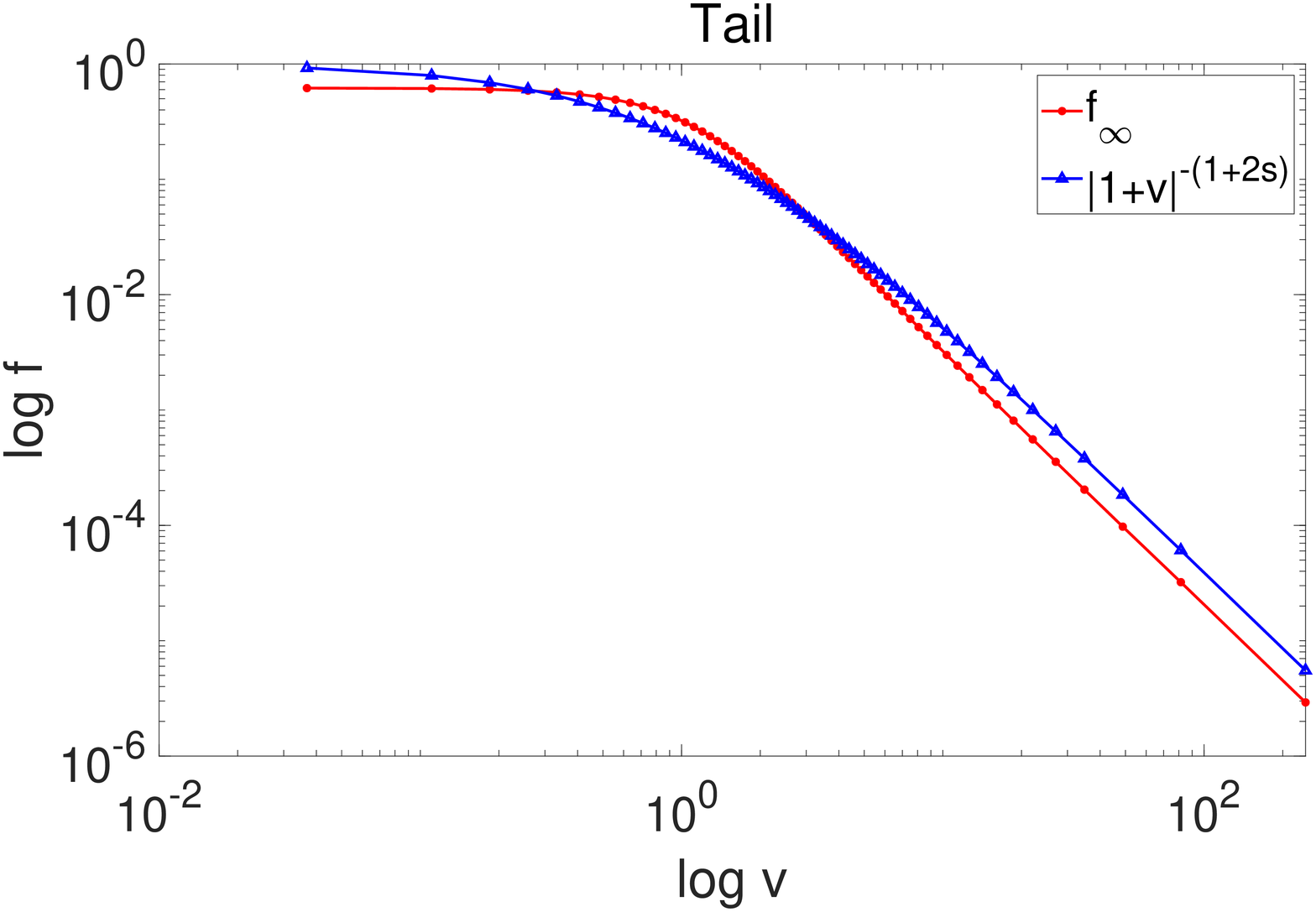}}
{\includegraphics[width=0.4\textwidth]{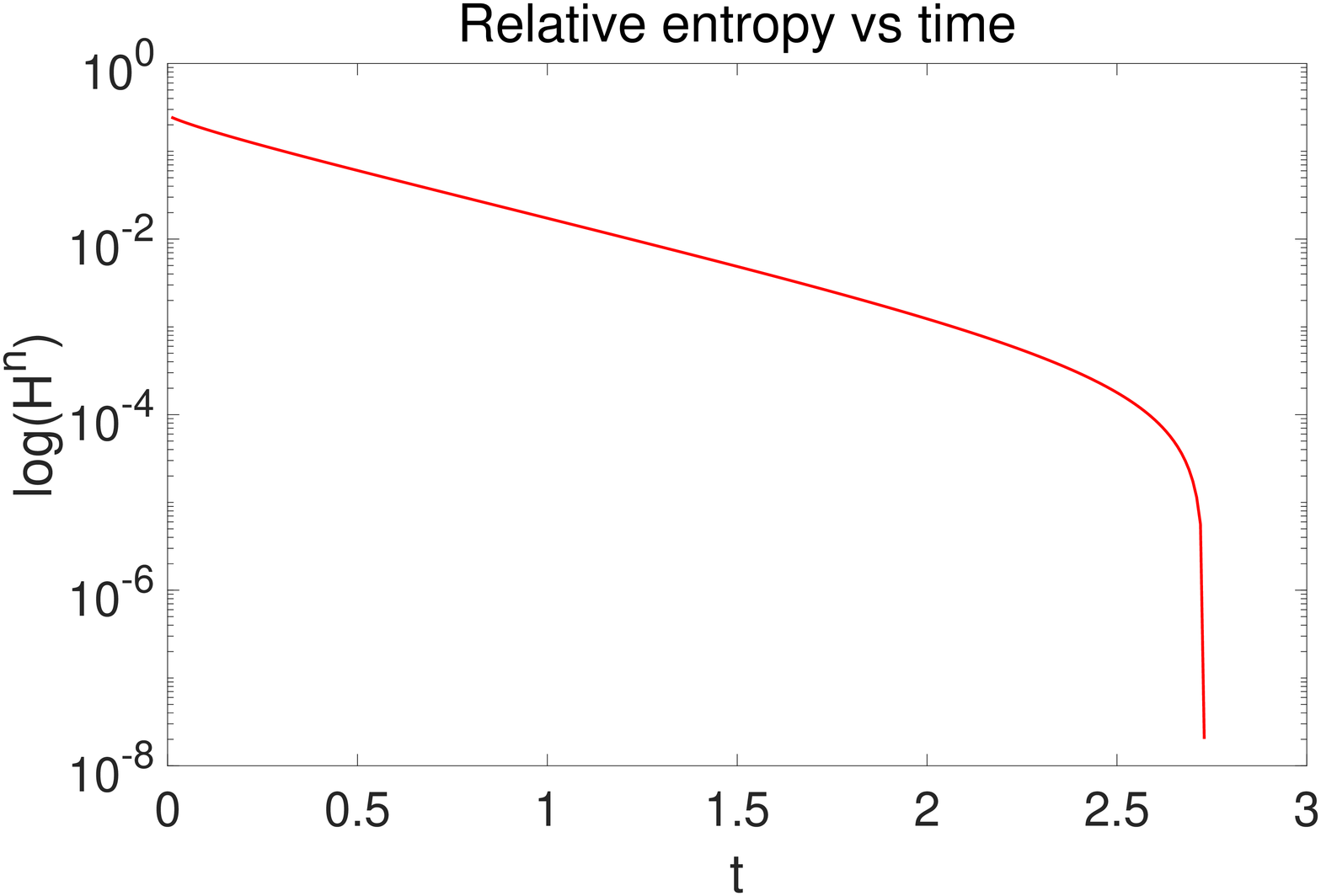}}
\vfill
{\includegraphics[width=0.4\textwidth]{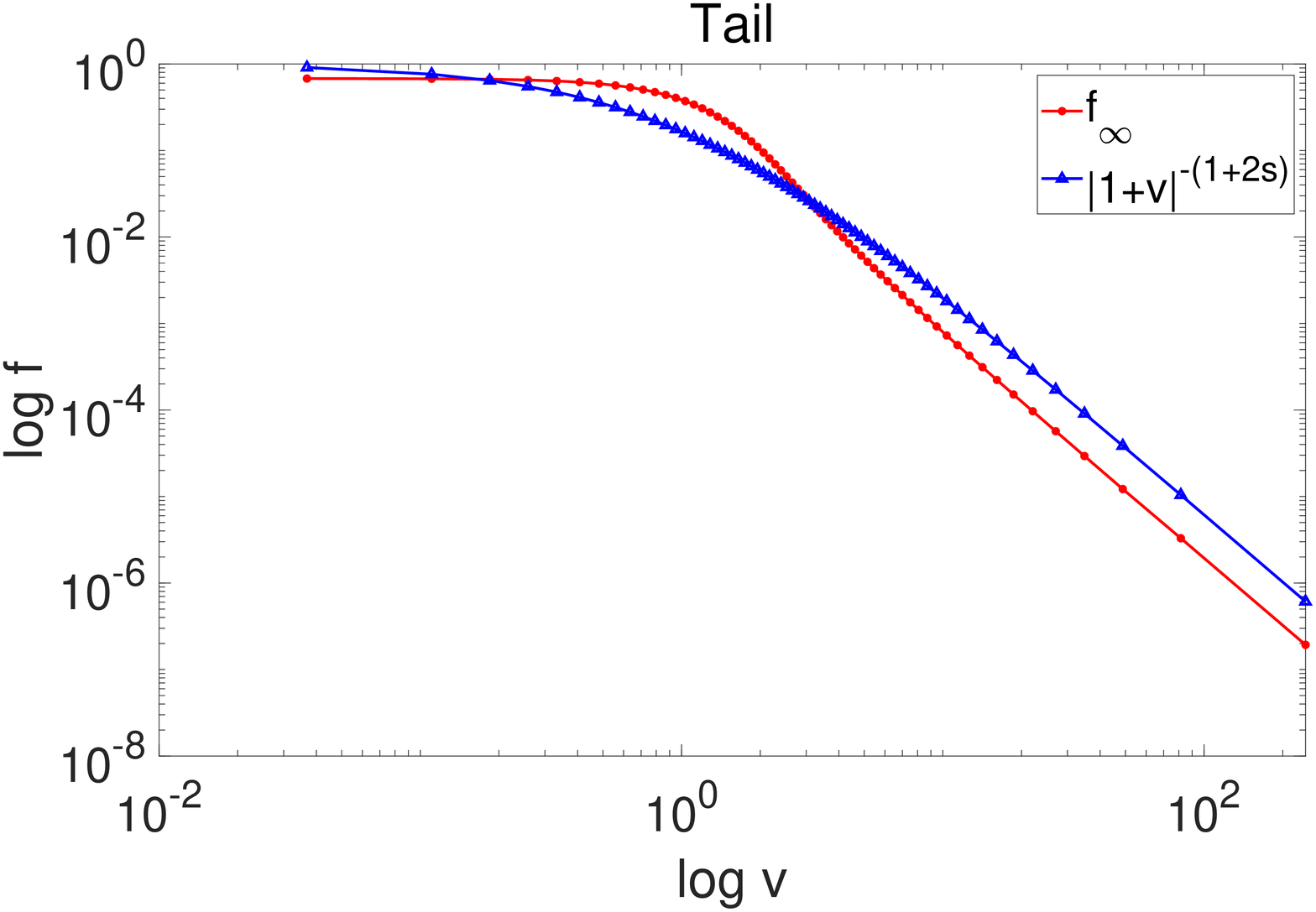}}
{\includegraphics[width=0.4\textwidth]{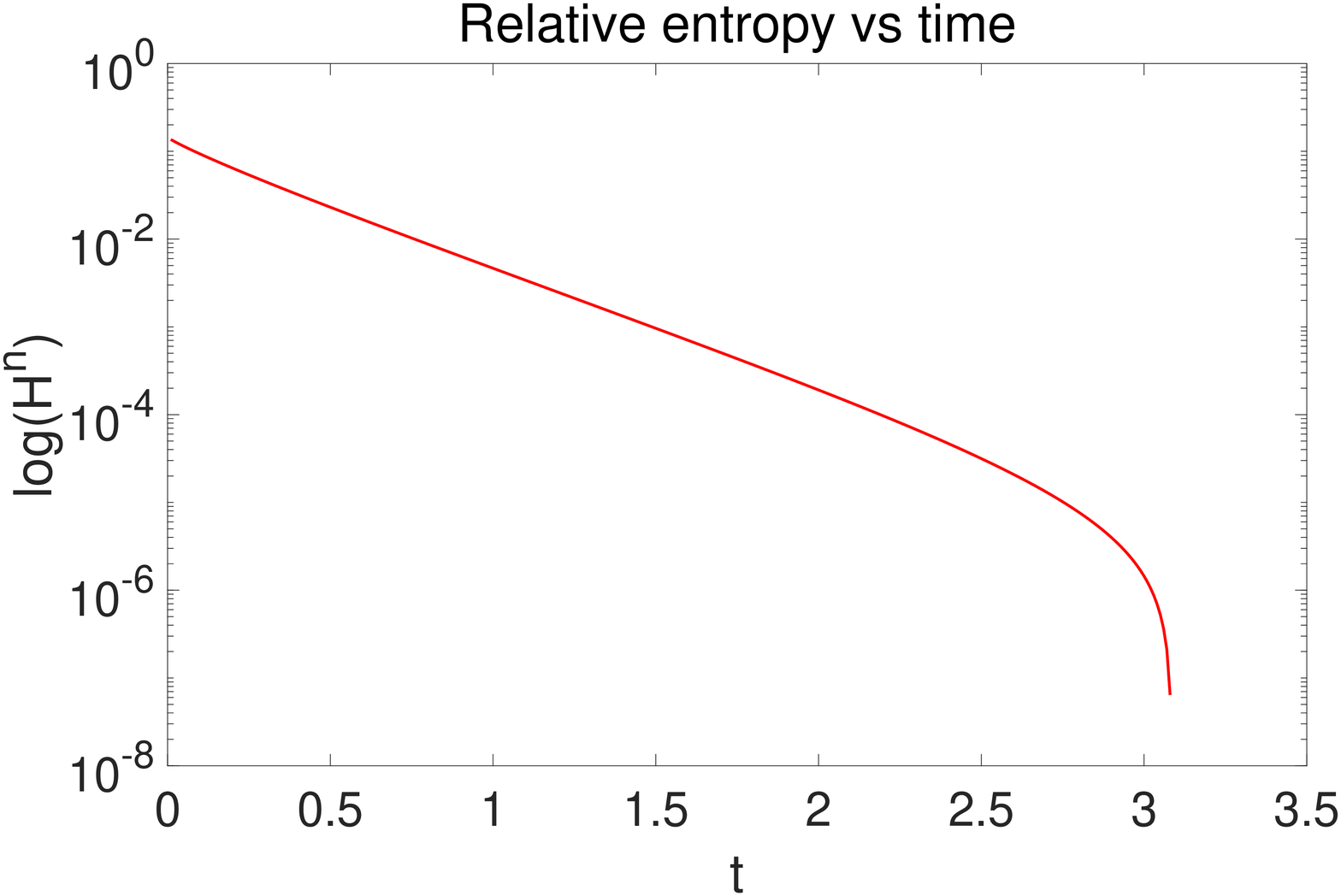}}
\caption{Computation of spatially homogeneous case \eqref{eqn: numeriacl_ex} for $s=0.6$ (top) and $s=0.8$ (bottom). Left column shows the tail behavior of the numerical equilibrium $f_{\infty}$ and right column shows exponential convergence of the relative entropy \eqref{eqnHn}. Here use $N_v=128$, $\Delta t =0.01$. The equilibrium is reached at $t=5.43$ and $t=4.1$ for $s=0.6$ and $s=0.8$, respectively.}
\label{fig:tail zp6 zp8}
\end{figure}

\subsubsection{Mass conservation}
In this section, we check the total mass versus time. As mentioned in Remark~\ref{rmk: mass}, the total mass is not exactly conserved by our scheme, but with proper choice of $L_v$ and $N_v$, the error can be controlled. Consider the same initial value problem as in \eqref{eqn: numeriacl_ex}, where the total mass is $M_0 = \pi^\half$, we define the following error in mass: 
\begin{equation} \label{masserror}
M^n = | \sum_{j=1}^{N_v} f_N^{n}(q_j) w_j \Delta q - M_0| , \qquad w_j = \frac{L_v}{(\sin (q_j))^2}\,,
\end{equation}
and plot it with $s=0.4,~0.6,~0.8$ in Fig.~\ref{fig:mass_error}. It is shown that, with the same choice of $L_v=3$, $N_v=128$, the mass is conserved better for larger $s$, which further indicates that slower decaying function is harder to compute numerically. We also check the mass error at the numerical equilibrium (defined in \eqref{Nequi}) for different $s$ with increasing $N_v$. The results are shown in Table~\ref{table:me}. As expected, with fixed $L_v$, larger $N_v$ leads to better conservation. 

\begin{figure}[H]
\centering
{\includegraphics[width=0.32\textwidth]{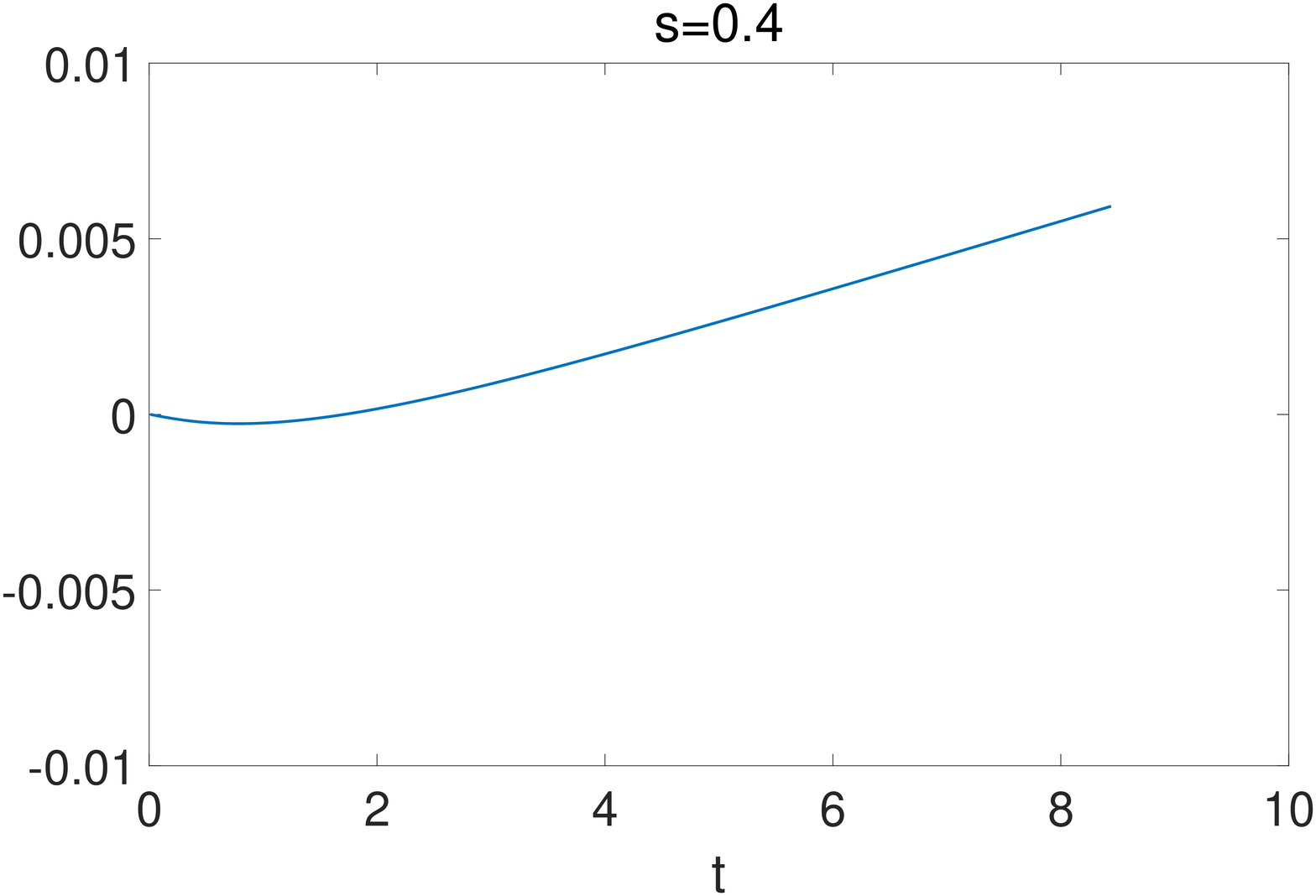}}
{\includegraphics[width=0.32\textwidth]{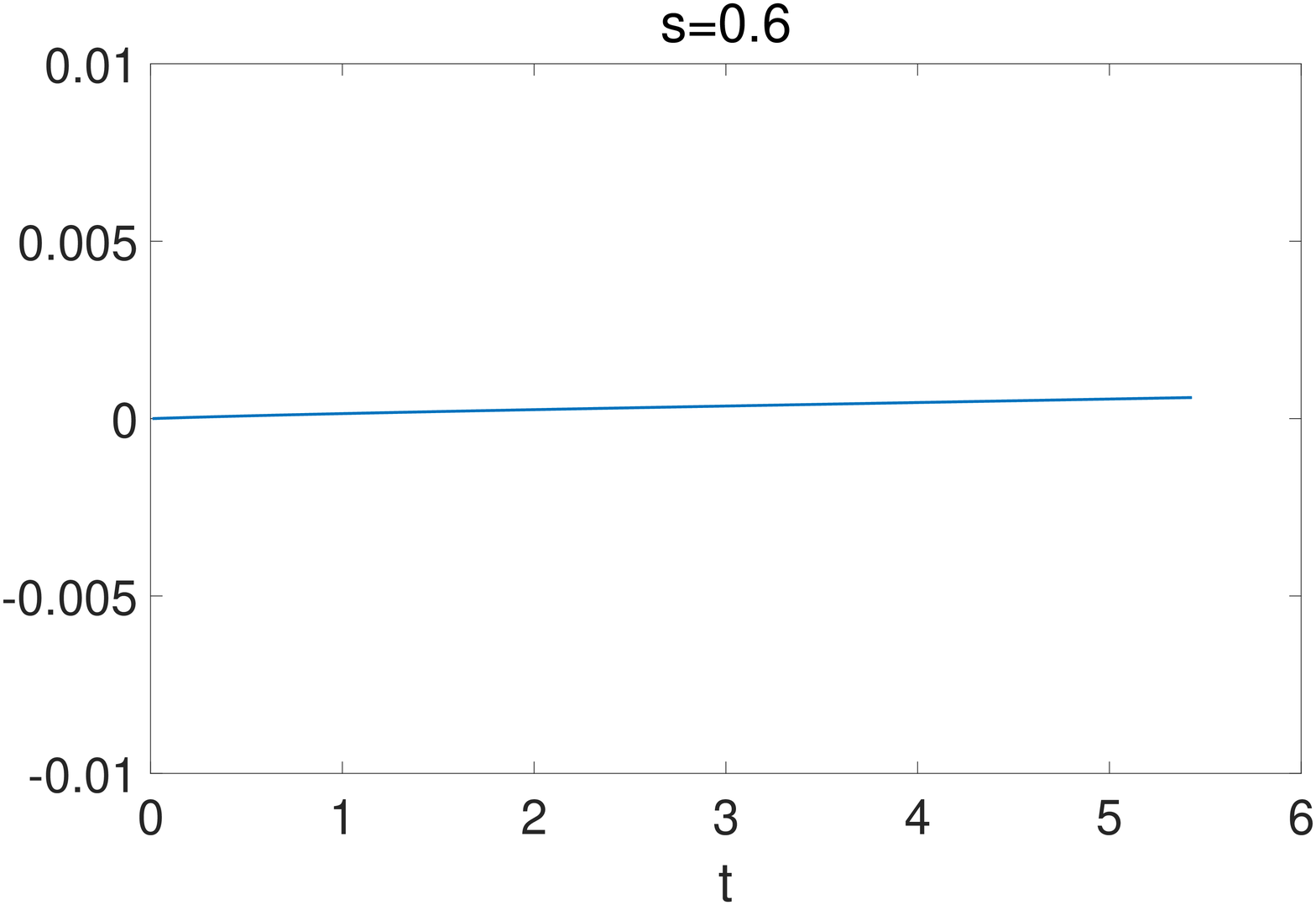}}
{\includegraphics[width=0.32\textwidth]{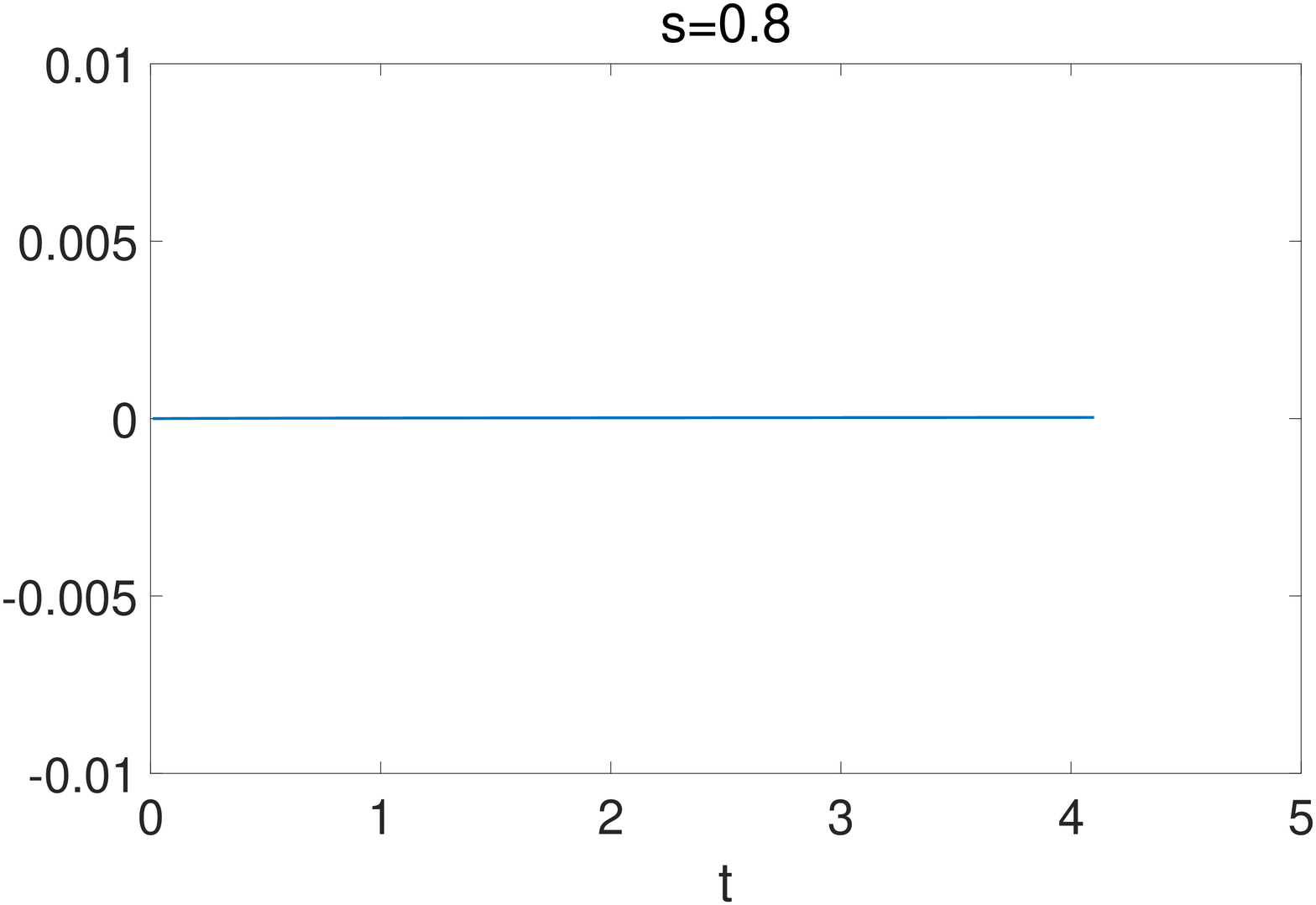}}
\caption{From top left to bottom right are mass error \eqref{masserror} over time for $s=0.4,0.6,0.8$ respectively, with $N_v = 128$, $L_v =3$. }
\label{fig:mass_error}
\end{figure}

\begin{table}[h]
\caption{Mass error at the numerical equilibrium $f_{\infty}^{N_v}$.} 
\centering 
\begin{tabular}{c c c c c} 
\hline\hline 
$N_v$   & 64 & 128 & 256 & 512\\ [0.5ex] 
\hline 
\text{$M^\infty$ with $s=0.4$} &7.8e-2  & 5.9e-3 & 2.1e-3 & 9.5e-4 \\ 
\text{$M^\infty$ with $s=0.5$} &5.7e-2  & 2.2e-3 &5.3e-4  & 1.3e-4 \\ 
\text{$M^\infty$ with $s=0.6$}  &3.3e-3  & 5.9e-4 &9.2e-5  & 5.3e-6 \\
\text{$M^\infty$ with $s=0.8$}  &2.9e-4  & 3.2e-5 & 9.4e-7 & 1.1e-6 \\[1ex] 
\hline 
\end{tabular}
\label{table:me} 
\end{table}

\subsection{Spatially inhomogeneous case}
Throughout this section, we compute $\eqref{eqn:111}$ with different choices of $\eps $. The following two initial conditions are considered: 
\begin{equation}\label{eqn: IC1}
f(0,x,v) = \pi^{-0.5}(1+\sin(\frac{\pi}{L_x}x))e^{-v^2}, \quad x\in [-\pi, \pi]\,,
\end{equation}
and
\begin{equation}\label{eqn: IC2}
f(0,x,v) = \pi^{-0.5}e^{-15x^2}e^{-v^2}\,, \quad  x \in [-5,5]\,.
\end{equation}
The equilibrium $\equilibrium$, except for $s=1/2$, is obtained numerically by running the spatially homogeneous solver until converge, as described in Section~\ref{sec:long time}.
Note that the periodic initial data \eqref{eqn: IC1} does not exactly fall into the assumption of Lemma~\ref{lemma: interchange}, but the conclusion there still holds. Indeed, for function $f(x,v)$ of the form \eqref{eqn: IC1}, one can write it into Fourier series, then it amounts to check $\int_{\mathbb R} \int_\Omega f(v,x) e^{-i \xi x}  \rd x \rd v = \int_\Omega \int_{\mathbb R}  f(v,x) e^{-i \xi x}  \rd v \rd x $, which is true as long as $f \in L^1(\mathbb{R} \times \Omega)$.

\subsubsection{Advantage of splitting in \eqref{eqn: semi_scheme}}
Before presenting specific numerical examples, we first verify the advantage of splitting and adding the term $-\gamma g^*$ in \eqref{semis1}. In particular, from \eqref{semis1},
one updates $g^*$ as $g^* = \Amat_{\gamma}^{-1} ( g^n - \Delta t \eps^{-2s}  I(\tr^n,\equilibrium))$, where $\Amat_{\gamma} = [\Imat-\Delta t \eps^{-2s} (\Lmat^s-\gamma \Imat )]$. Then the success of this step hinges on the condition number of $\Amat_{\gamma}$. In Table~\ref{table:cda}, we compute the condition number of $\Amat_\gamma$ with $\gamma = 0, \half, 1, 2$. The other parameters used are $L_x = \pi$, $N_x = 50$, $L_v = 3$, $N_v=64$ and $\Delta t = 0.1$. One sees that larger $\gamma$ leads to a better conditioned $\Amat_\gamma$, and this improvement is more pronounced for smaller $\eps$. Therefore, in the fractional diffusive regime, $\gamma$ plays a non-negligible role. 

\begin{table}[h]
\caption{Condition number of $\Amat_{\gamma}$.} 
\centering 
\begin{tabular}{c c c c c} 
\hline\hline 
   & $Cond(\Amat_{0})$ &  $Cond(\Amat_{\half})$ &  $Cond(\Amat_{1})$ &  $Cond(\Amat_{2})$ \\ [0.5ex] 
\hline 
 $s = 0.4$, $\eps = 1$ &4.73  & 4.54 & 4.36 & 4.06  \\ 
$s = 0.6$, $\eps = 1$ &4.93  &4.73  & 4.55 & 4.24 \\ 
$s = 0.8$, $\eps = 1$  &13.28  & 12.66 & 12.11 & 11.14 \\
$s = 0.4$, $\eps = 1e-3$ &7.74e3  & 158 & 52 & 20.90 \\ 
$s = 0.6$, $\eps = 1e-3$ &1.24e5 &187.42  & 63.44 & 25.45 \\ 
$s = 0.8$, $\eps = 1e-3$  &5.97e6  & 564.32 & 190.75 &75.81 \\
$s = 0.4$, $\eps = 1e-5$ &3.57e5  & 181.66 & 55.84 & 21.43  \\ 
$s = 0.6$, $\eps = 1e-5$ &3.15e7 &188.98  & 63.67 & 25.49  \\ 
$s = 0.8$, $\eps = 1e-5$  &9.46e9  & 564.62 & 190.79 & 75.82 \\[1ex] 
\hline 
\end{tabular}
\label{table:cda} 
\end{table}

\subsubsection{Uniform accuracy in time}
In this section, we show the first order accuracy in time by computing the following error
\[
e_{\Delta t} = \sum_{i}^{N_x} \sum_{j}^{N_v} |f^{\Delta t}_{i,j}(T) - f^{\Delta t/2}_{i,j}(T)|w_j \Delta q \Delta x .
\]
In Fig.~\ref{fot_acc}, first order accuracy in time is shown among different choices of $eps$: $\eps=1$, $\eps=1e-3$ and $\eps=1e-5$.
\begin{figure}[!ht]
\centering
{\includegraphics[width=0.45\textwidth]{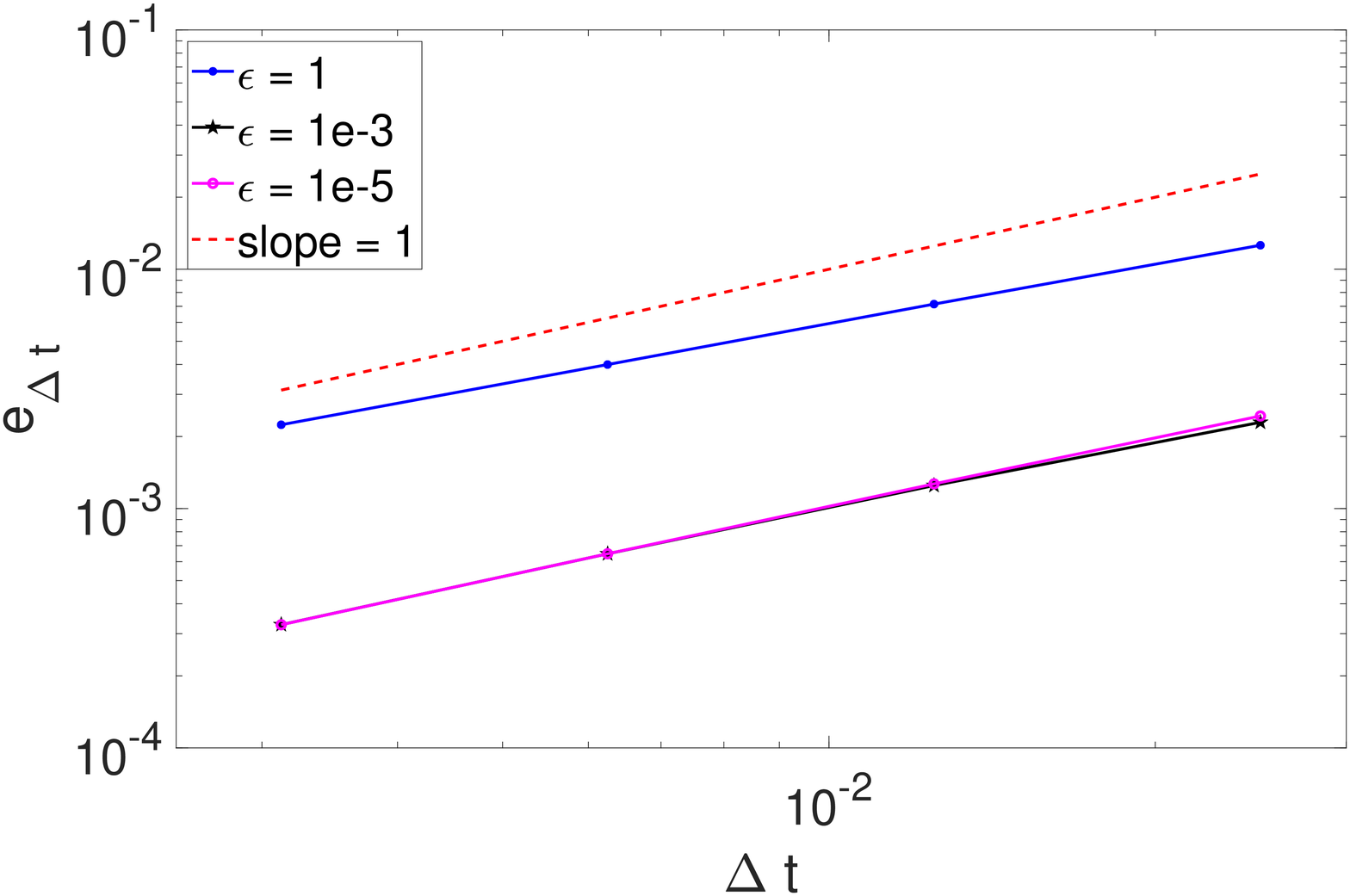}}
{\includegraphics[width=0.45\textwidth]{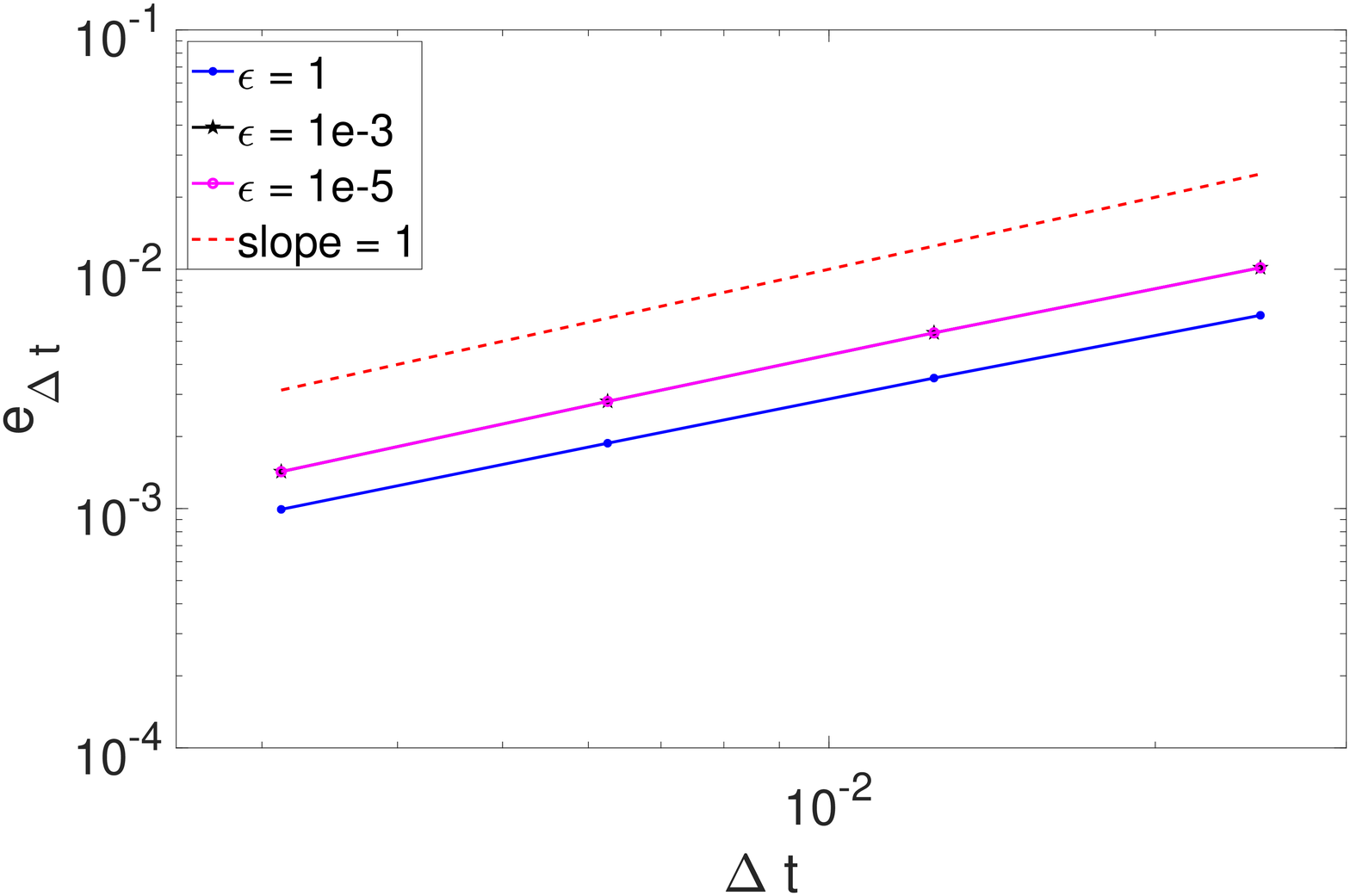}}
\caption{ First order accuracy in time for \eqref{eqn:111} with initial condition $\eqref{eqn: IC2}$. The left is for $s=0.4$, with $L_x = 5$, $N_x = 200$, $L_v = 3$ and $N_v=128$, right is for $s=0.8$, with $L_x = 5$, $N_x = 200$, $L_v = 3$ and $N_v=128$. For both cases, run up to $T=0.1$ with time step sizes $\Delta t = 0.025, 0.0125, 0.00625, 0.003125, 0.0015625$. }
\label{fot_acc}
\end{figure}

\subsubsection{Energy stability}
We compute the total energy, and the individual energy corresponding to the macro and micro part in this section. More precisely, the numerical approximation to \eqref{Ef} and \eqref{Eetag} are
\[
E_f = \sum_{i=1}^{N_x} \sum_{j=1}^{N_v} \frac{ f_{i,j}^2}{\equilibrium_j}  w_j \Delta q \Delta x\,, \quad 
E_g = \sum_{i=1}^{N_x} \sum_{j=1}^{N_v} \frac{ g_{i,j}^2}{\equilibrium_j}  w_j \Delta q \Delta x\,, \quad 
E_\eta = \sum_{i=1}^{N_x} \sum_{j=1}^{N_v} \eta_{i,j}^2\equilibrium_j  w_j \Delta q \Delta x\,.
\]
In Fig.~ \ref{fig: energy}, we compute \eqref{eqn:111} with initial condition $\eqref{eqn: IC2}$ and various $\eps$, ranging from $1$ to $1e-5$. The solution is computed up to $T=0.1$. As shown, $E_f$ decays with time, and $E_g$ and $E_\eta$ are uniformly bounded in time, which coincides with the Propositions ~\ref{prop:energy1} and \ref{prop:energy2} and confirms the stability of our scheme. When $\eps=1e-5$, the sudden change at the first step is due to fact that \eqref{eqn: IC2} is not at the equilibrium.

\begin{figure}[H]
\centering
{\includegraphics[width=0.32\textwidth]{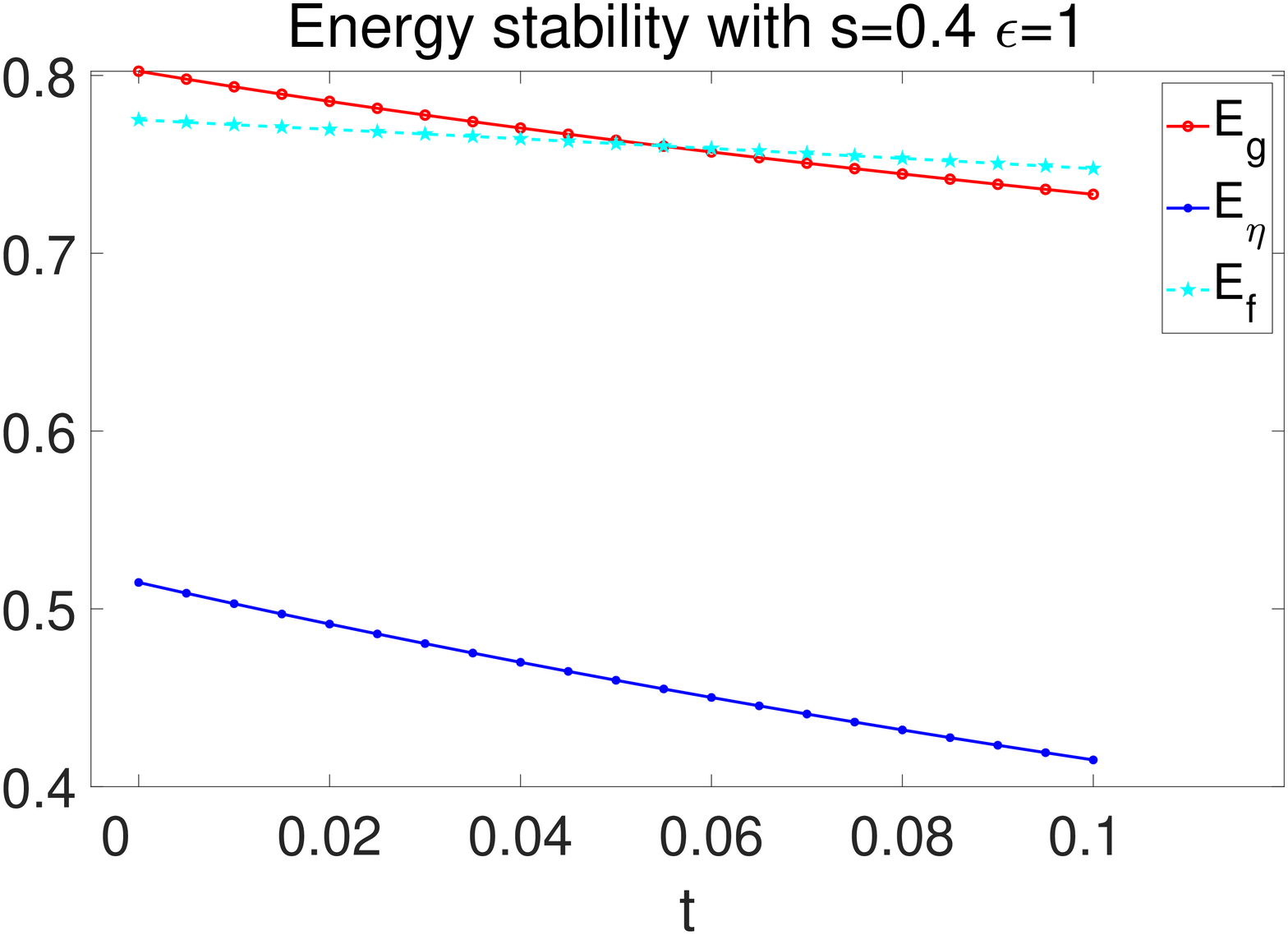}}
{\includegraphics[width=0.32\textwidth]{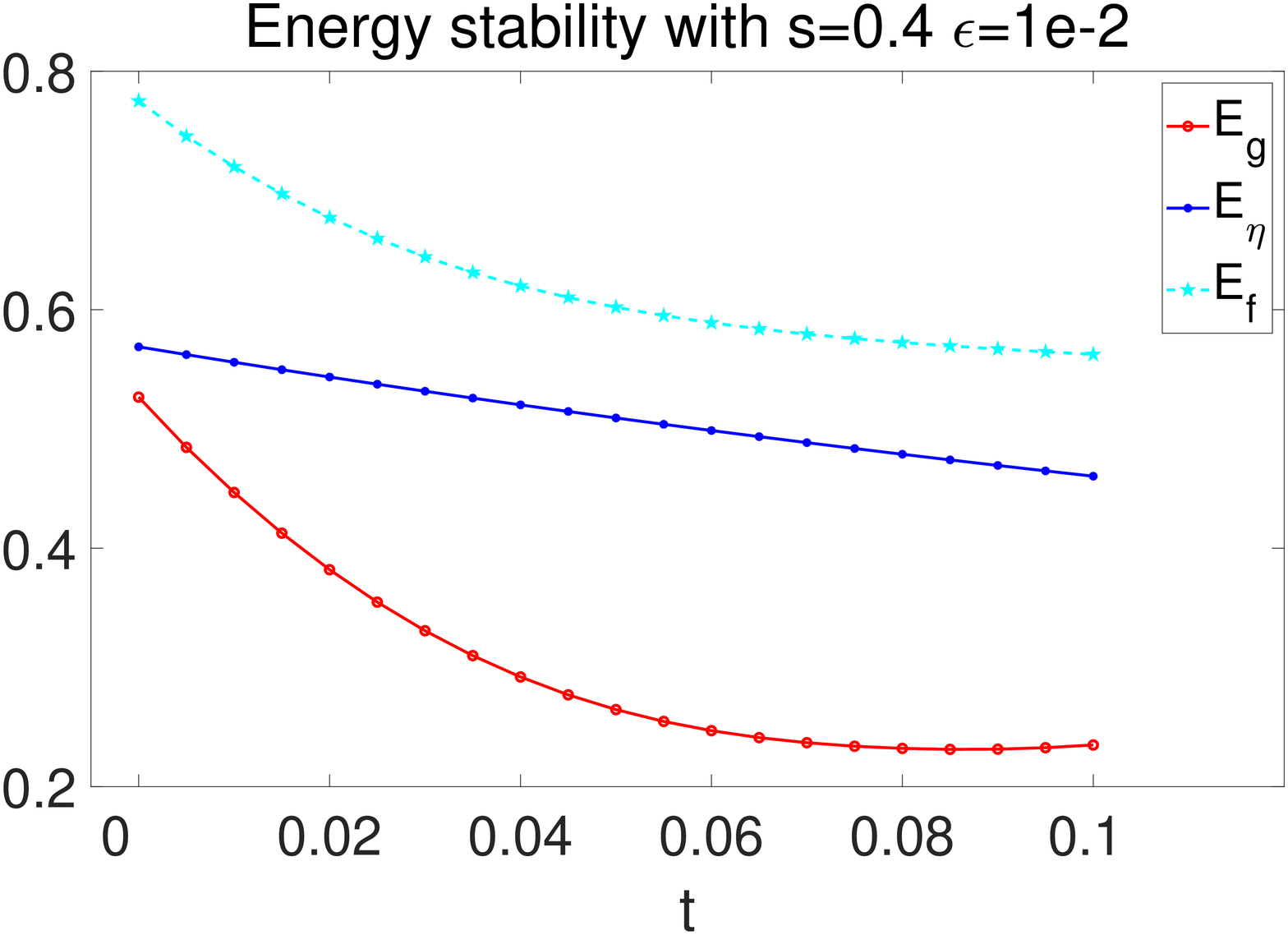}}
{\includegraphics[width=0.32\textwidth]{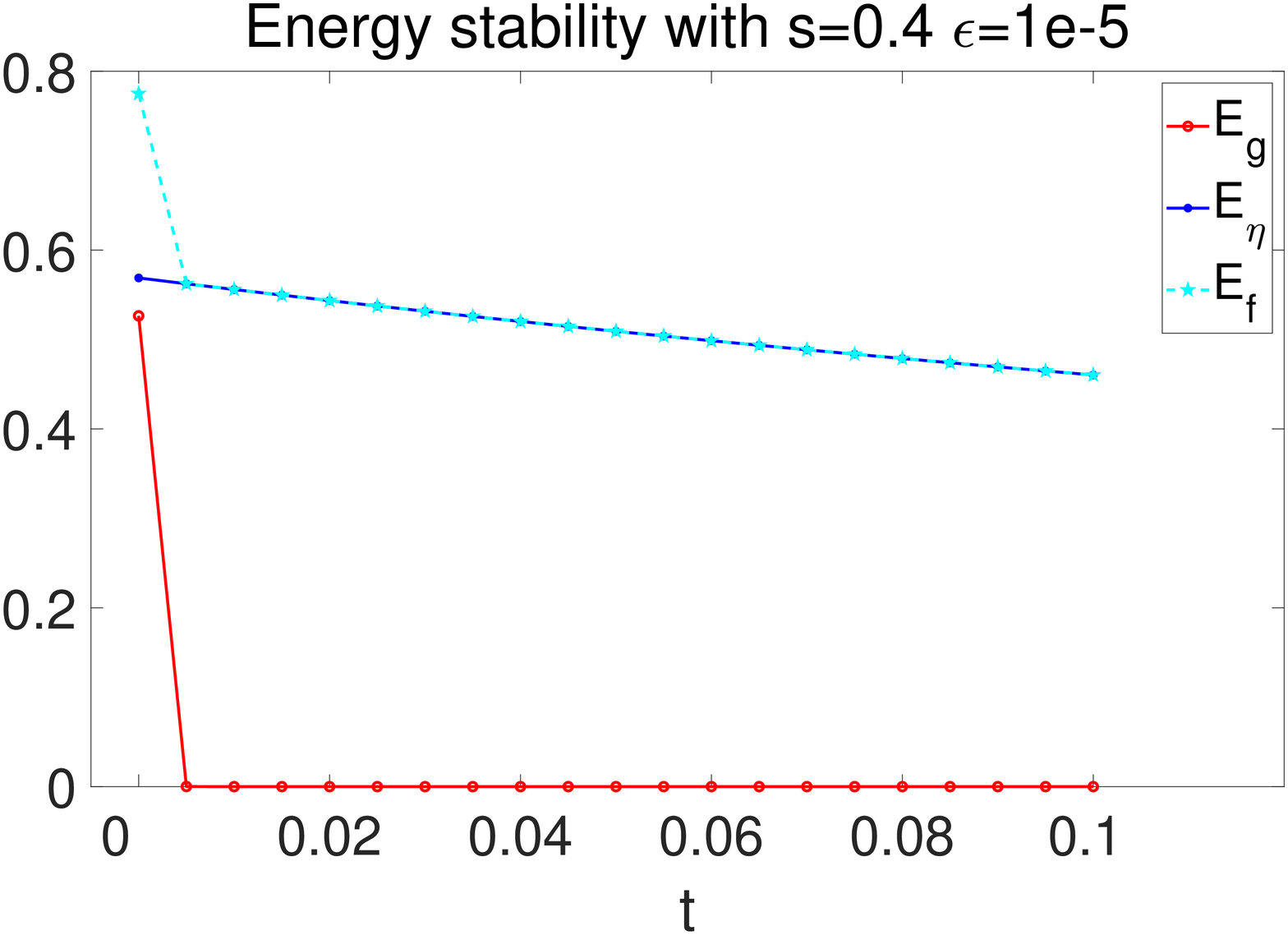}}
\caption{Computation of $E_f$, $E_g$ and $E_\eta$ for \eqref{eqn:111} with initial condition $\eqref{eqn: IC2}$. The numerical parameters for all three different $\eps$ cases are:  $L_x=5$, $L_v=3$, $N_x=100$, $N_v=128$, $\Delta t=0.01$, and the solution is computed up to $T=0.1$.}
\label{fig: energy}
\end{figure}

\subsubsection{Kinetic regime $\eps = 1$}
We then check the performance of our scheme in the kinetic regime with $\eps=1$. As a comparison, the following implicit-explicit (IMEX) scheme is used on finer mesh to produce the reference solution:
\begin{align*} 
\begin{cases}{}
\frac{f^{*} - f^n}{\Delta t} +  v \partial_x f^{n} = 0,
\\ \frac{f^{n+1} - f^*}{\Delta t} = \partial_v (v f^{n+1}) - (-\Delta_v)^s f^{n+1}.
\end{cases}
\end{align*}
Here Fig.~\ref{two_scheme_comp_IC_2} and Fig.~\ref{two_scheme_comp_IC_1} correspond to initial conditions  \eqref{eqn: IC1} and \eqref{eqn: IC2}, respectively. Different choices of $s$ are considered. One sees that the numerical solution from our AP scheme agrees well with the reference solution.

\begin{figure}[!ht]
\centering
{\includegraphics[width=0.32\textwidth]{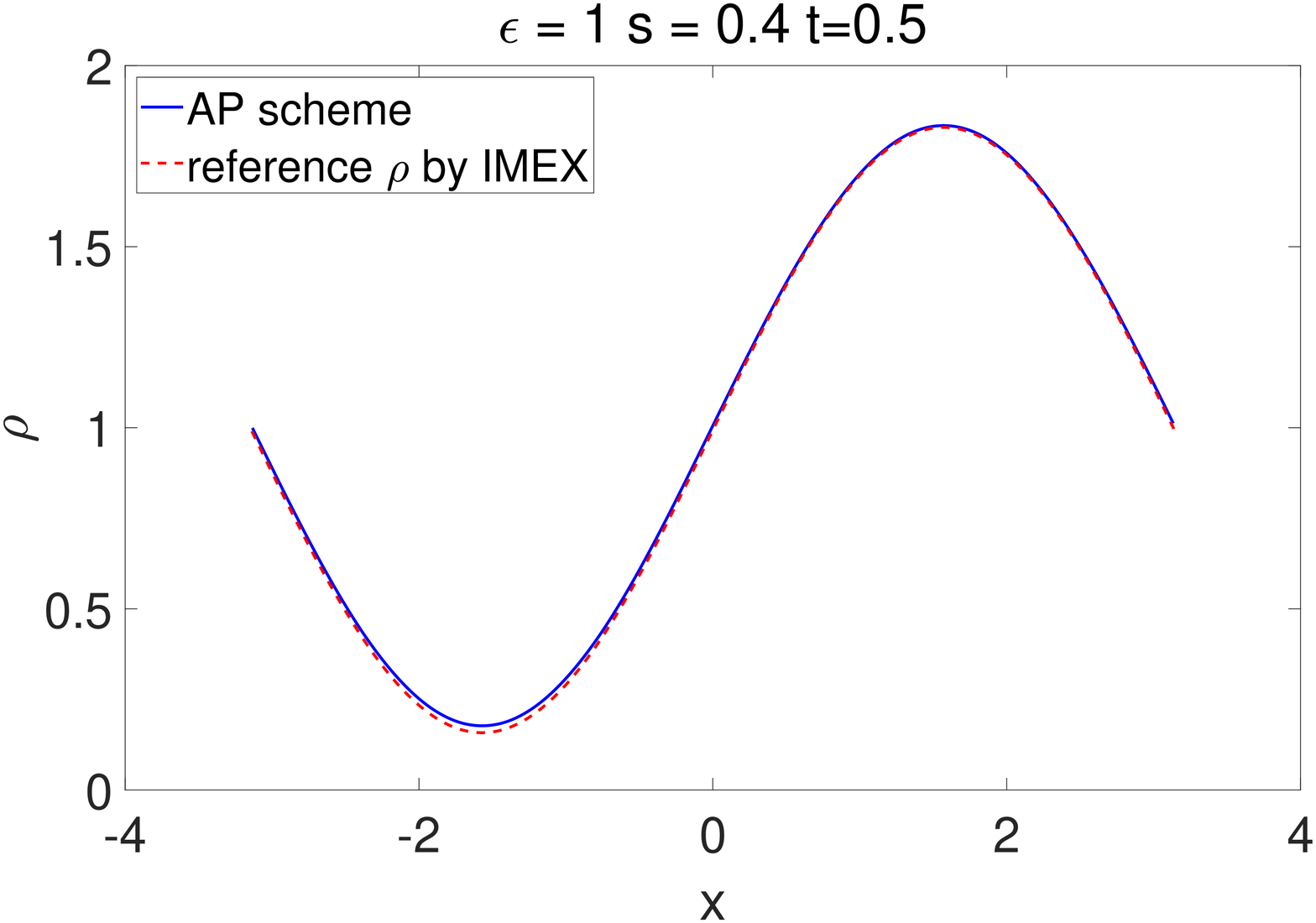}}
{\includegraphics[width=0.32\textwidth]{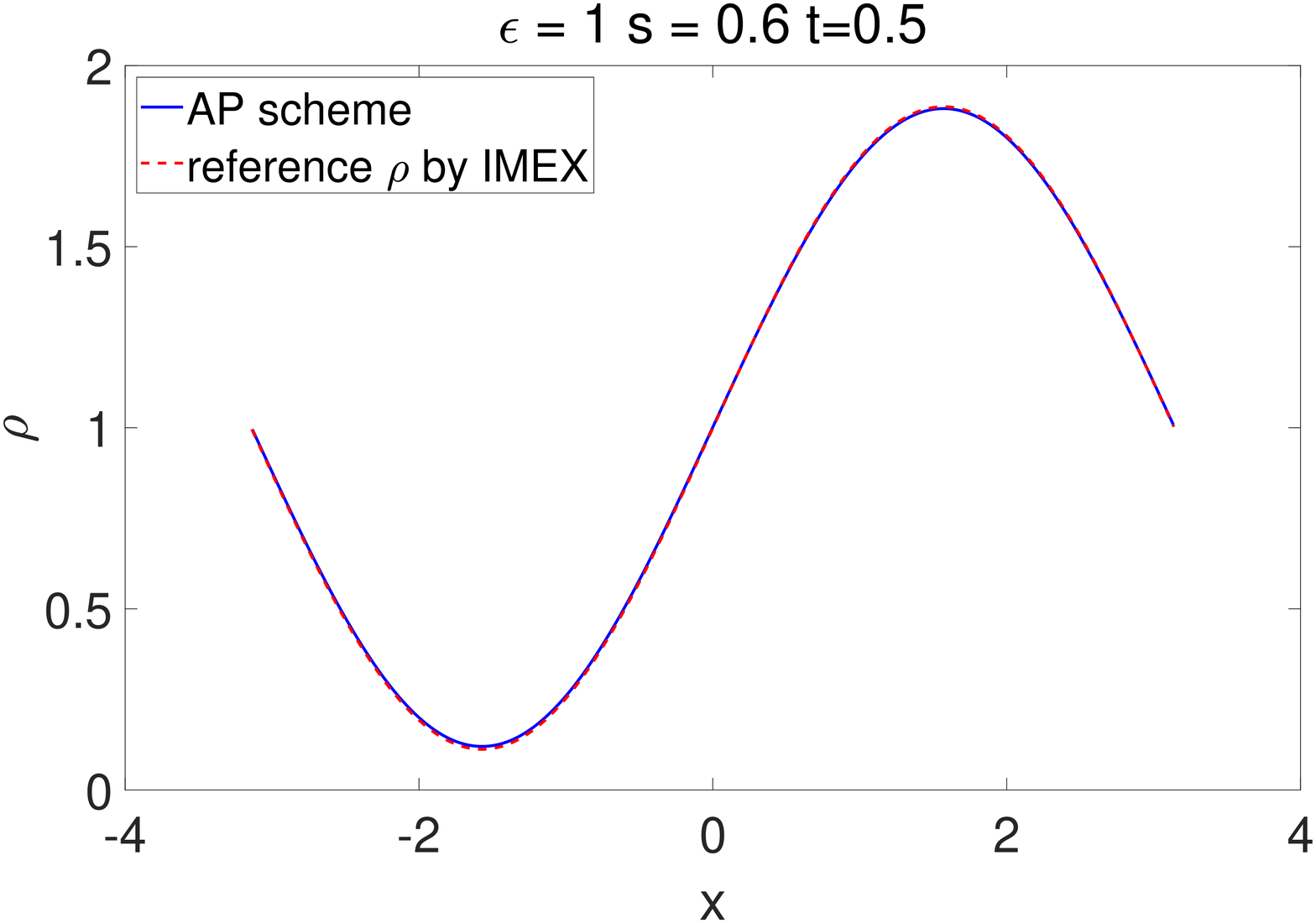}}
{\includegraphics[width=0.32\textwidth]{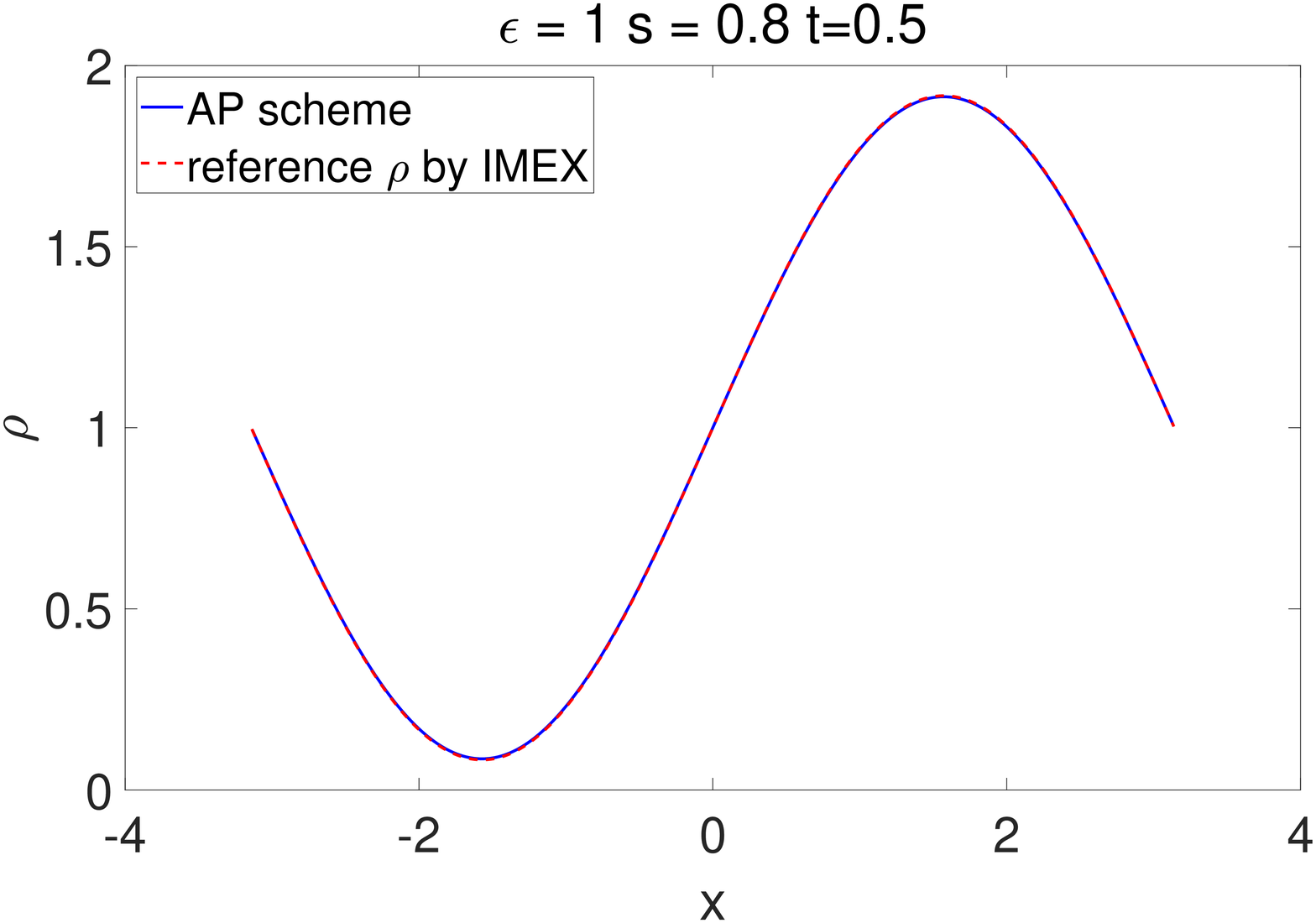}}
\caption{ Plot of solution at $t=0.5$ for \eqref{eqn:111} with initial condition \eqref{eqn: IC1}. $L_x = \pi$. For our AP scheme, $N_x = 50$, $N_v=64$, $\Delta t = 0.05$. For the reference solution, $N_x = 100$, $N_v=128$, $\Delta t = 1e-4$.}
\label{two_scheme_comp_IC_2}
\end{figure}

\begin{figure}[!ht]
\centering
{\includegraphics[width=0.32\textwidth]{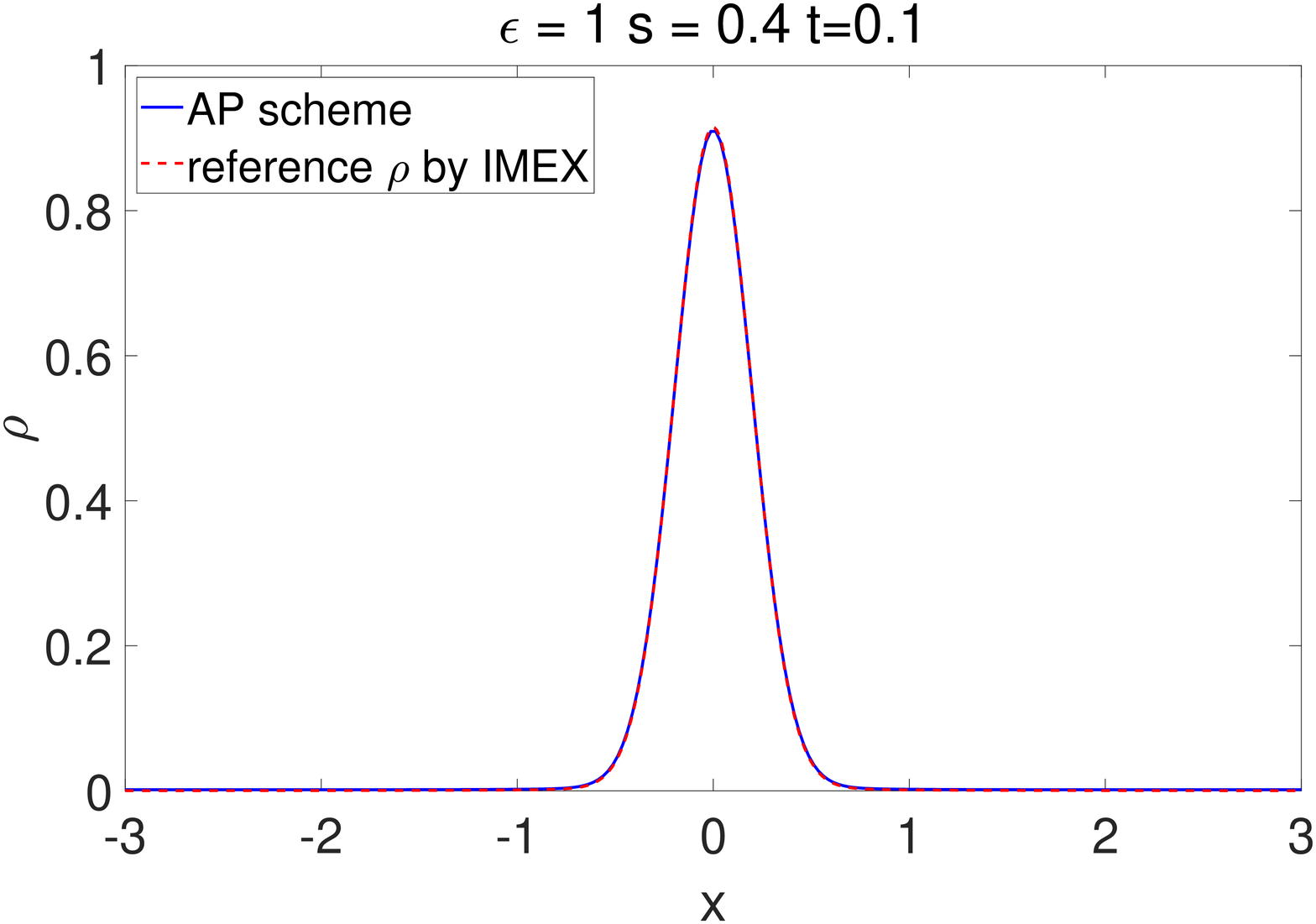}}
{\includegraphics[width=0.32\textwidth]{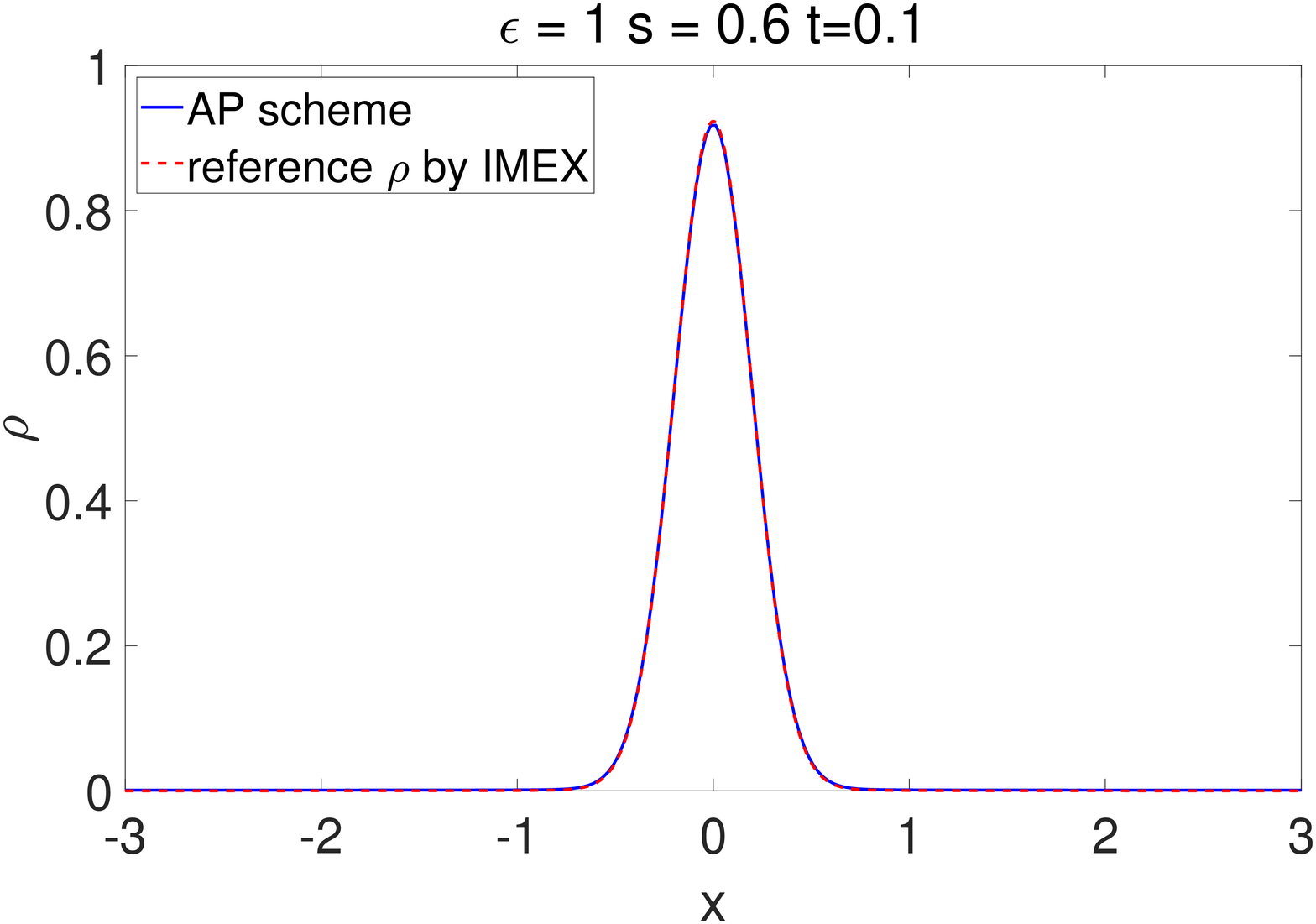}}
{\includegraphics[width=0.32\textwidth]{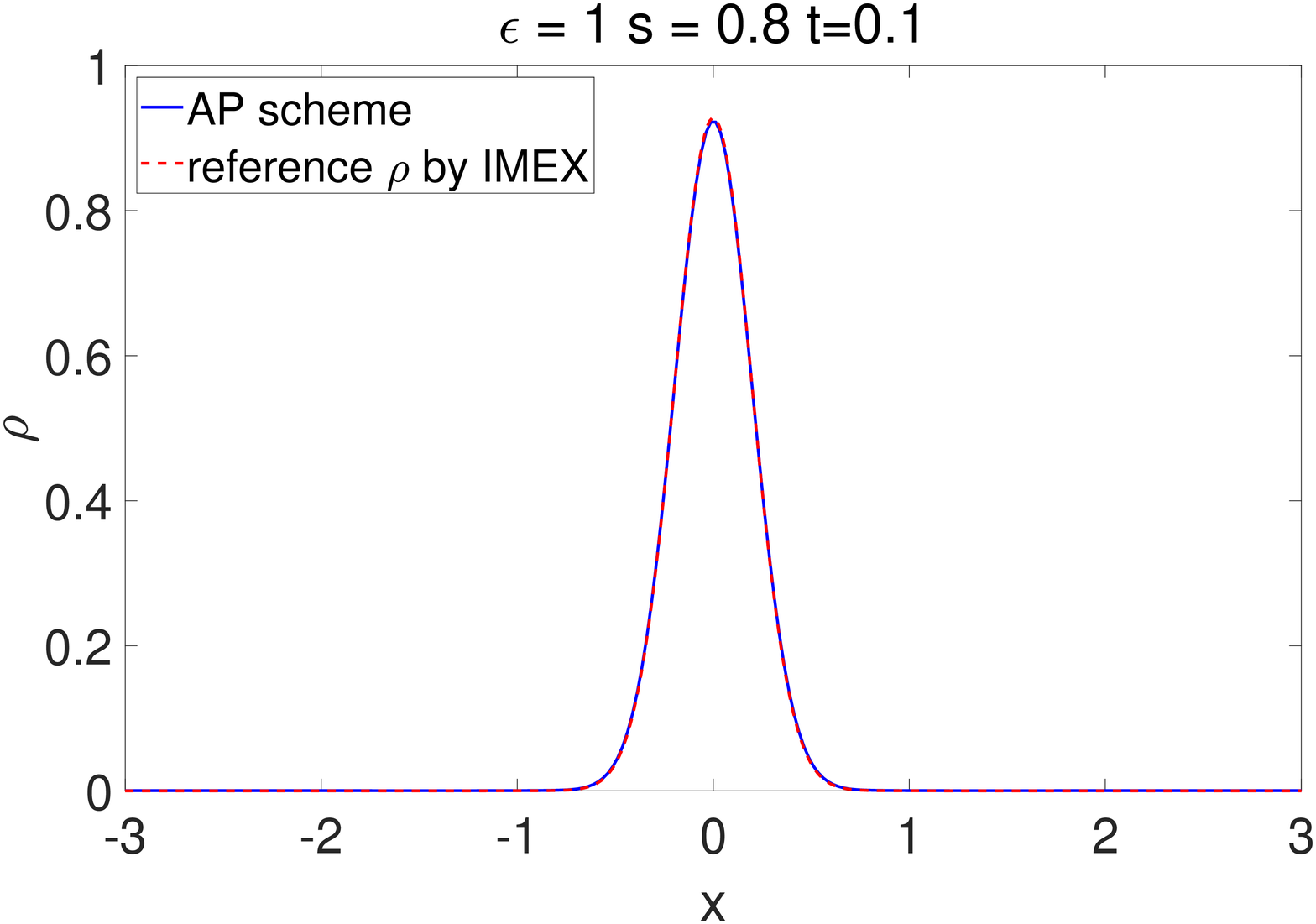}}
\caption{Plot of solution at $t=0.5$ for \eqref{eqn:111} with initial condition \eqref{eqn: IC2}. $L_x=5$. For our AP scheme, $N_x = 200$, $N_v=128$, $\Delta t = 0.01$. For reference solution, $N_x = 800$, $N_v=256$, $\Delta t = 1e-4$.}
\label{two_scheme_comp_IC_1}
\end{figure}

\subsubsection{AP property and diffusive regime}
In this section, we check the AP property of the scheme and test its performance in the diffusive regime. To check the AP property, we compute the following asymptotic error 
\begin{equation} \label{APerror}
Error = \|\rho^{\eps}\equilibrium - {f}^{\eps}\|_1  = \sum_{i=1}^{N_x} \sum_{j=1}^{N_v} | \rho^{\eps}_i \equilibrium_j - {f^{\eps}_{i,j}}| w_j  \Delta q \Delta x\,,
\end{equation}
where $\equilibrium$ is the equilibrium and $\rho^{\eps} = \int_{\RR} f^{\eps} \rd v$.

When $\eps = 1e-5$, we compare the solution to our AP scheme with the solution to the diffusive equation \eqref{eqn: limit_system}, which is computed using the Fourier spectral method. 

The results are collected in Fig.~\ref{fig: ap_check_IC_1} and Fig.~\ref{fig: ap_check_IC_2}, for initial data \eqref{eqn: IC1} and \eqref{eqn: IC2}, respectively. The left column of each figure represents the asymptotic error \eqref{APerror} in time with different choice of $\eps$. It is clearly that the error decreases with vanishing $\eps$, and it is at a magnitude of approximately $\mathcal O(\eps)$. On the right, a good match between the solution to the AP scheme with the solution to the fractional limit is observed. Two different $s$ are tested for each cases.  
\begin{figure}[!ht]
\centering
{\includegraphics[width=0.45\textwidth]{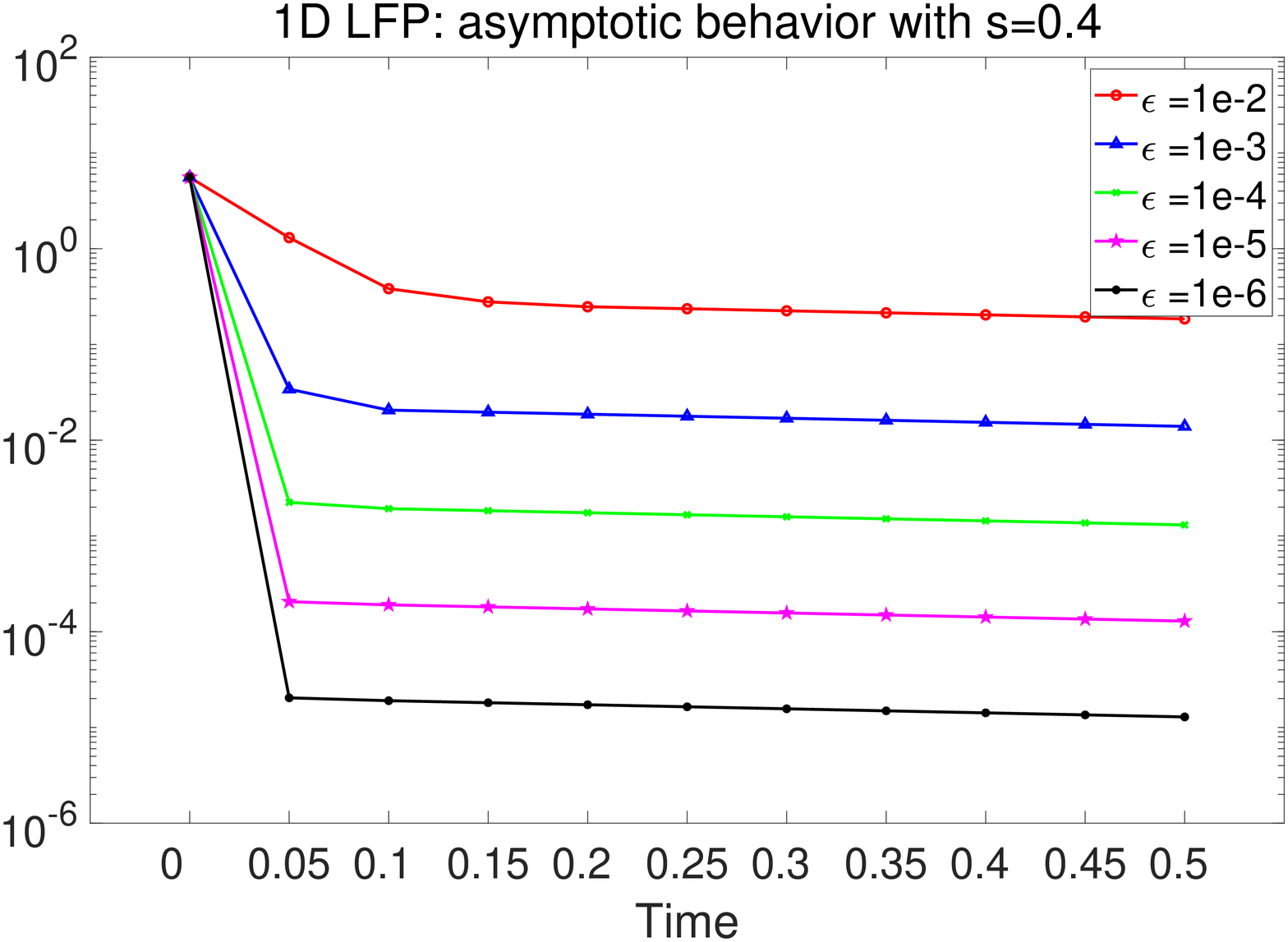}}
{\includegraphics[width=0.45\textwidth]{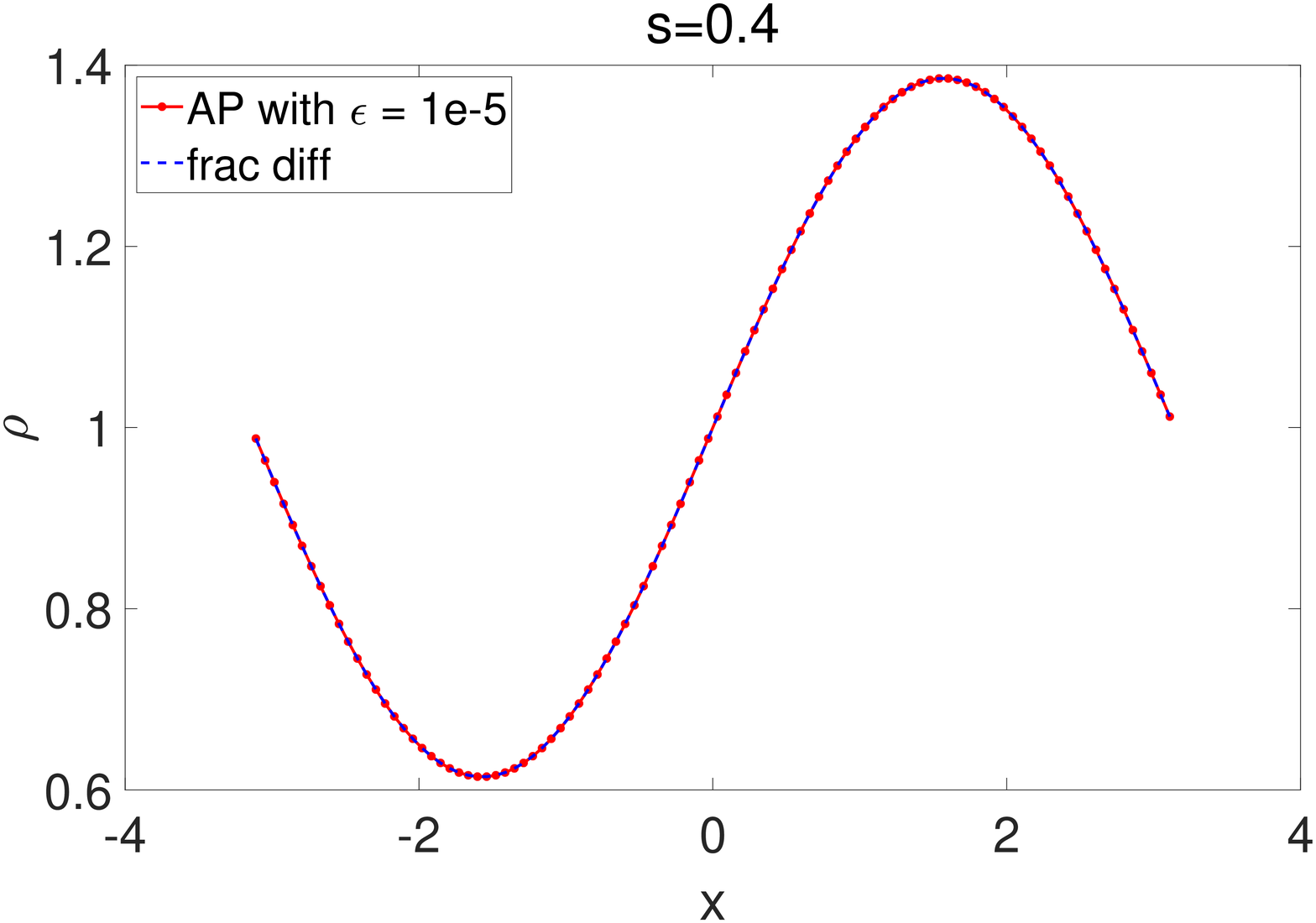}}
{\includegraphics[width=0.45\textwidth]{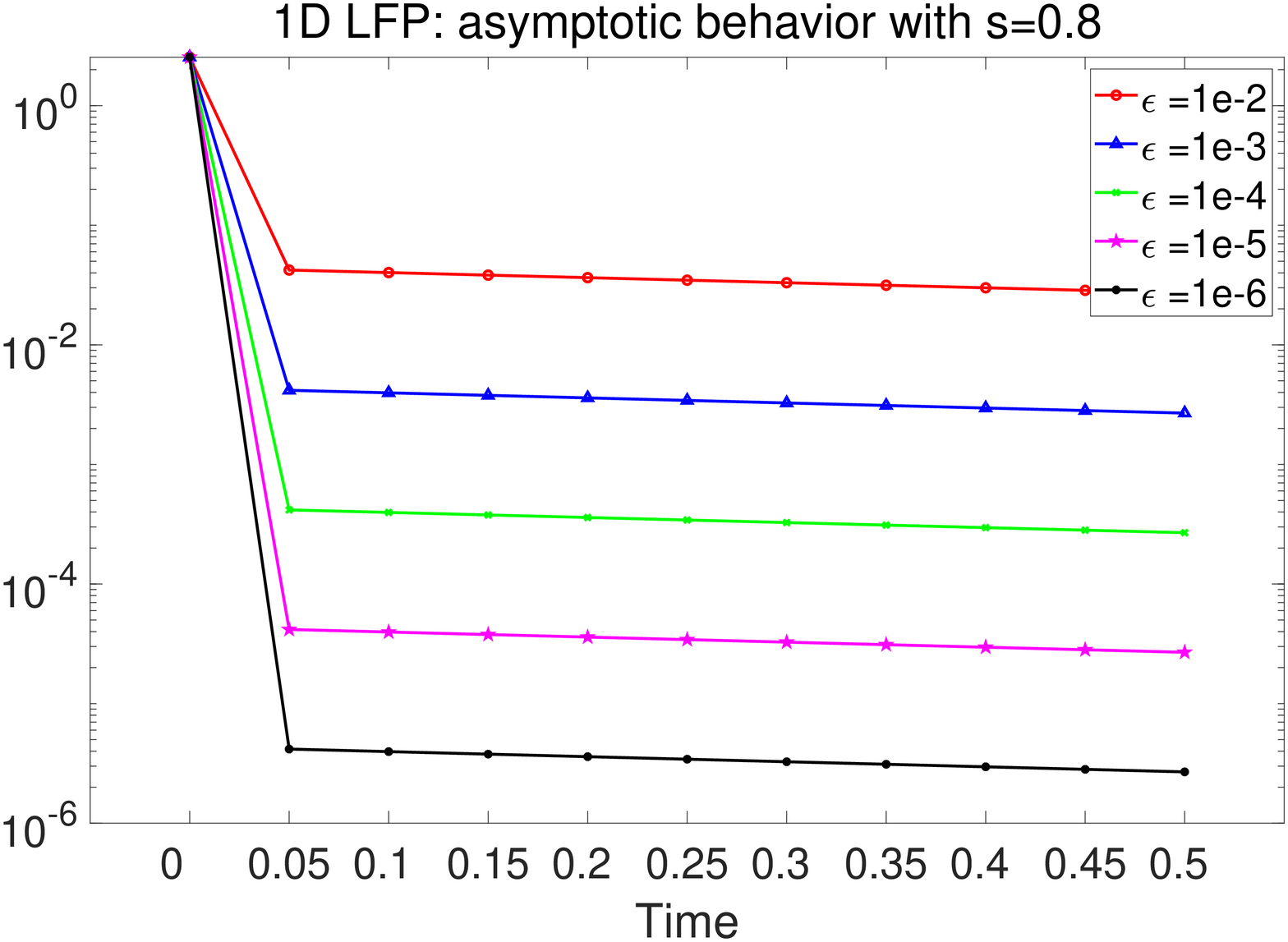}}
{\includegraphics[width=0.45\textwidth]{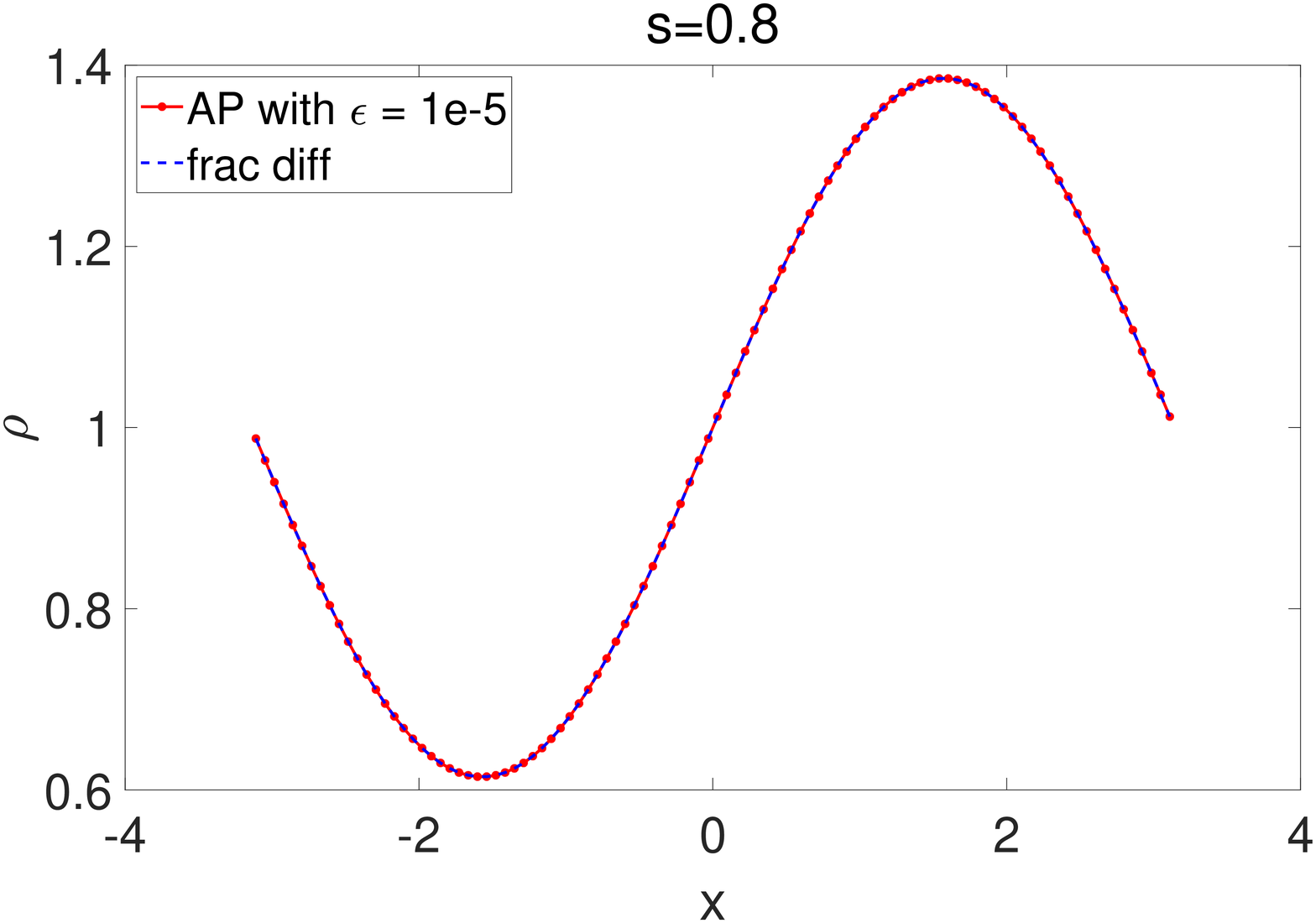}}
\caption{Computation of \eqref{eqn:111} with initial condition \eqref{eqn: IC1}. Left: asymptotic error \eqref{APerror} versus time. Right: plot of solution at $T=1$. $L_x = \pi$, $N_x=100$, $\Delta t = 0.1$ are used in both AP scheme and the Fourier spectral method in computing $\eqref{eqn: limit_system}$. $N_v=128$, $L_v=3$ are used for velocity variable.}
\label{fig: ap_check_IC_1}
\end{figure}

\begin{figure}[!ht]
\centering
{\includegraphics[width=0.45\textwidth]{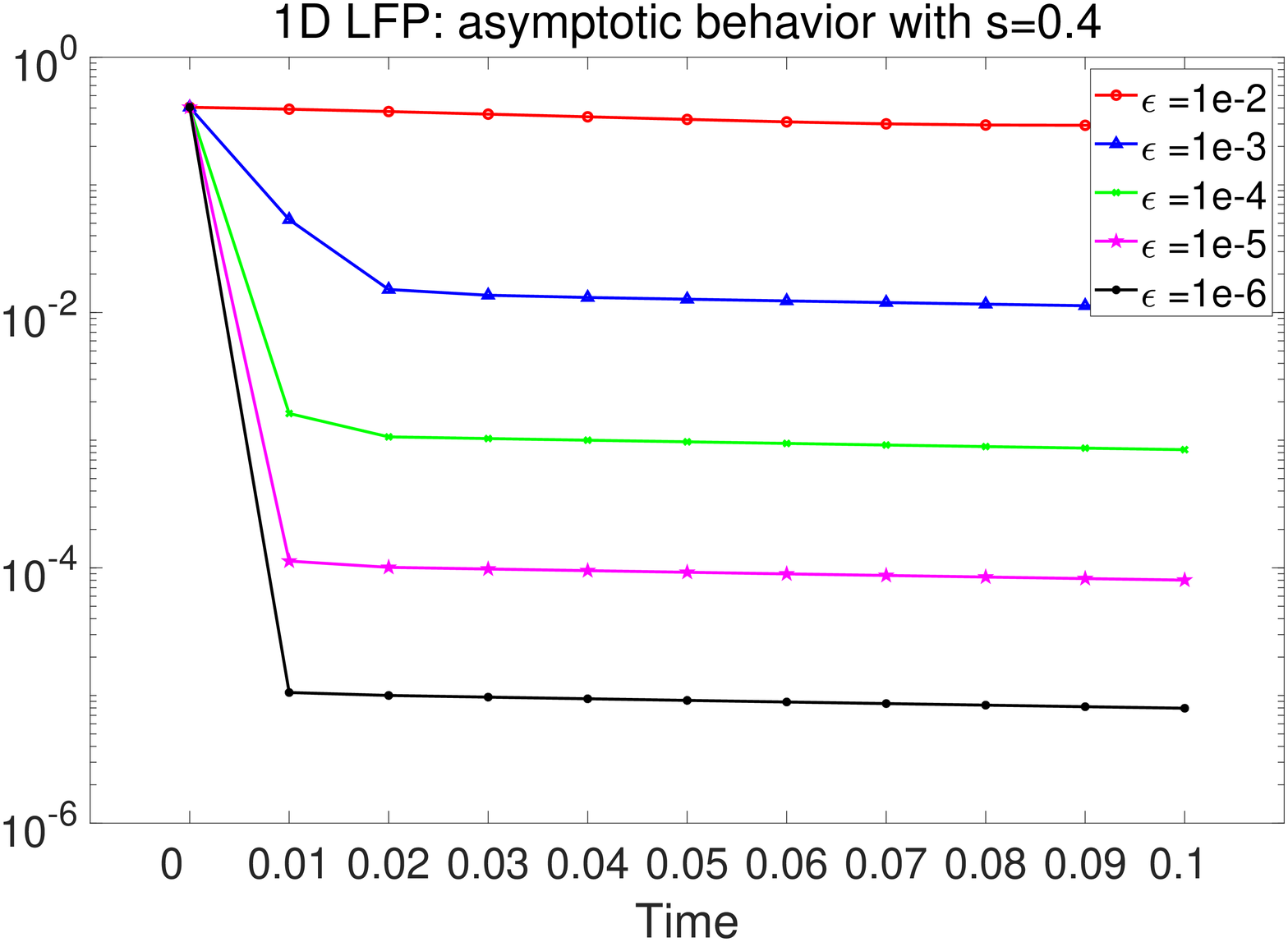}}
{\includegraphics[width=0.45\textwidth]{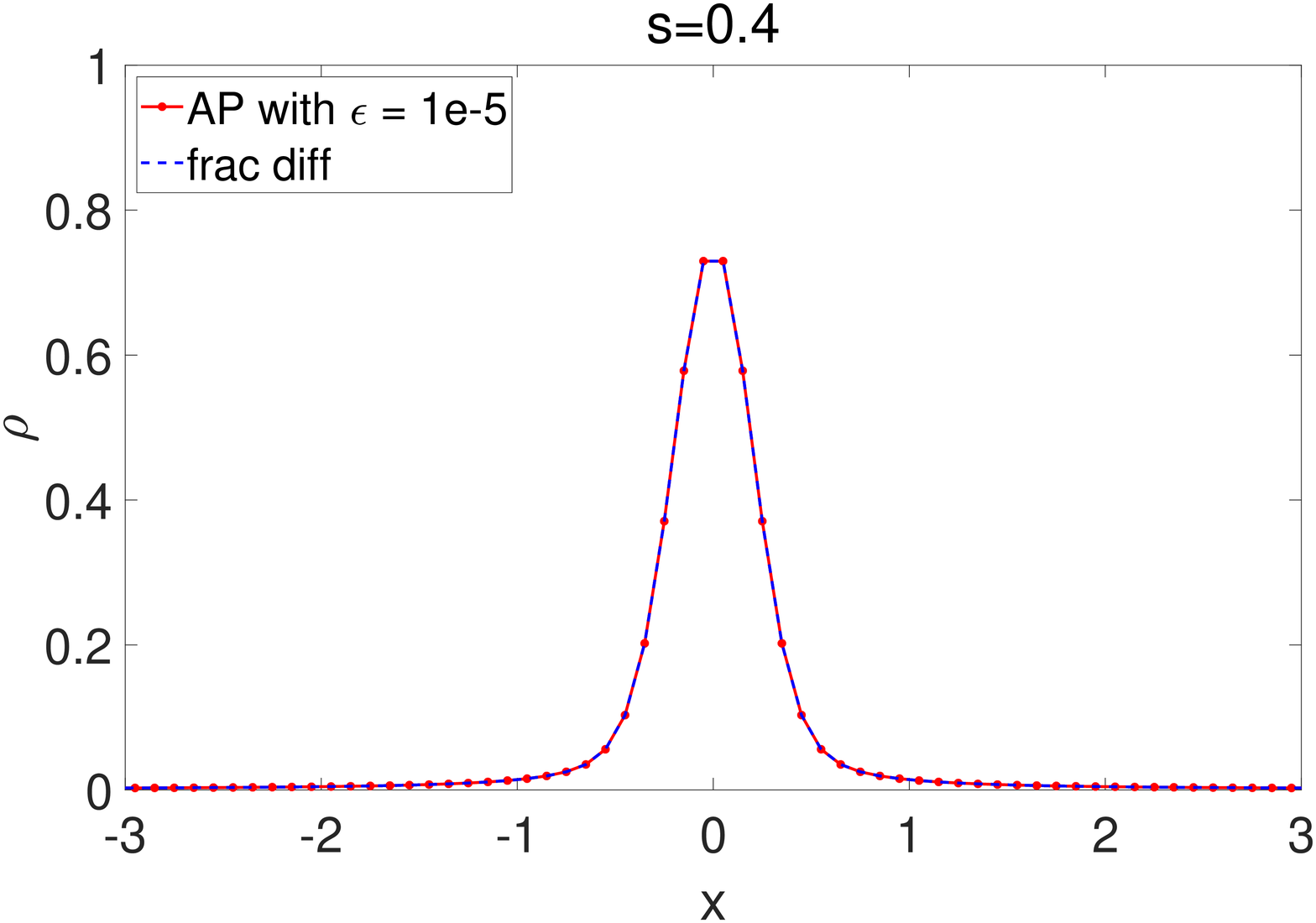}}
{\includegraphics[width=0.45\textwidth]{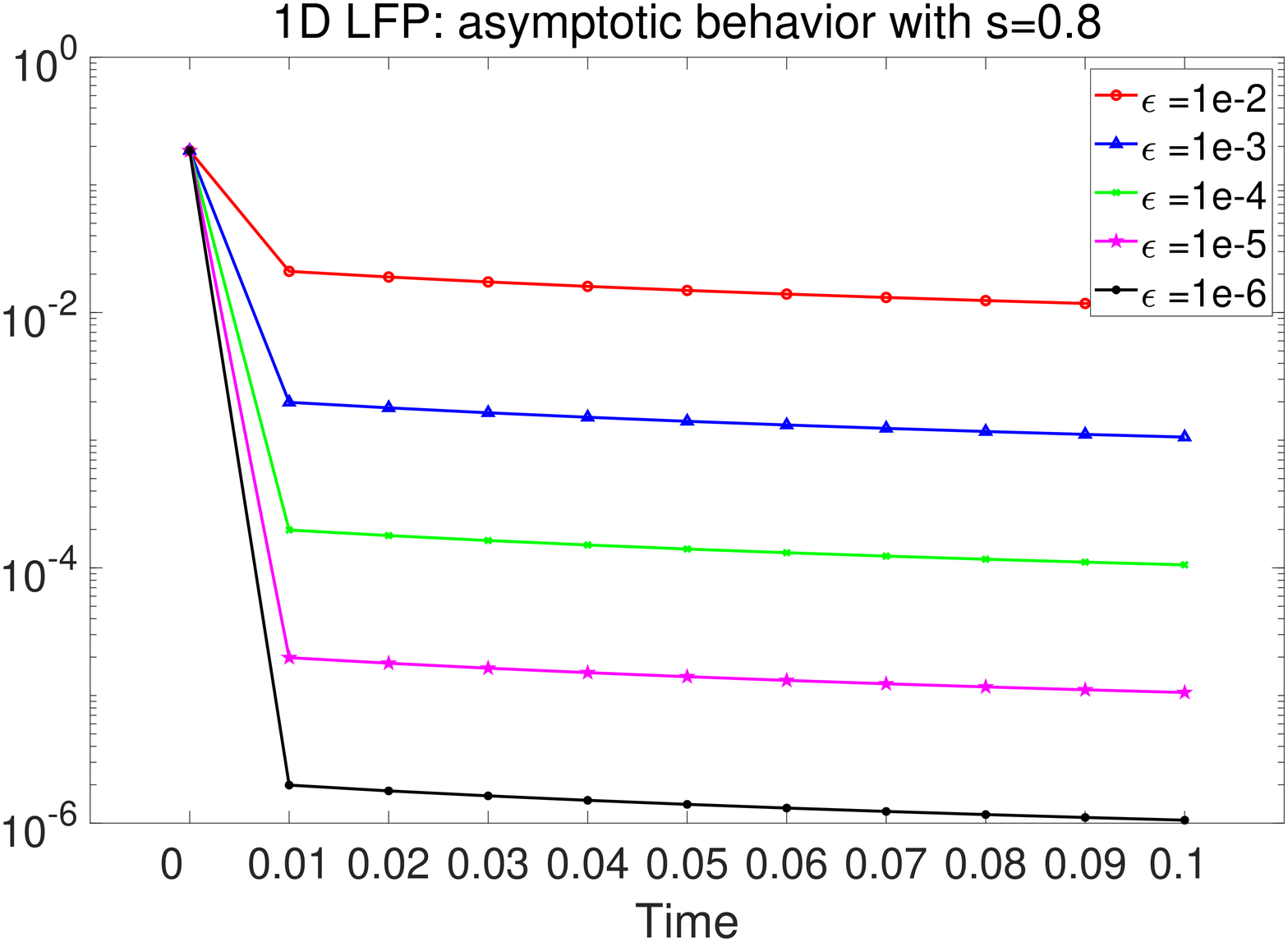}}
{\includegraphics[width=0.45\textwidth]{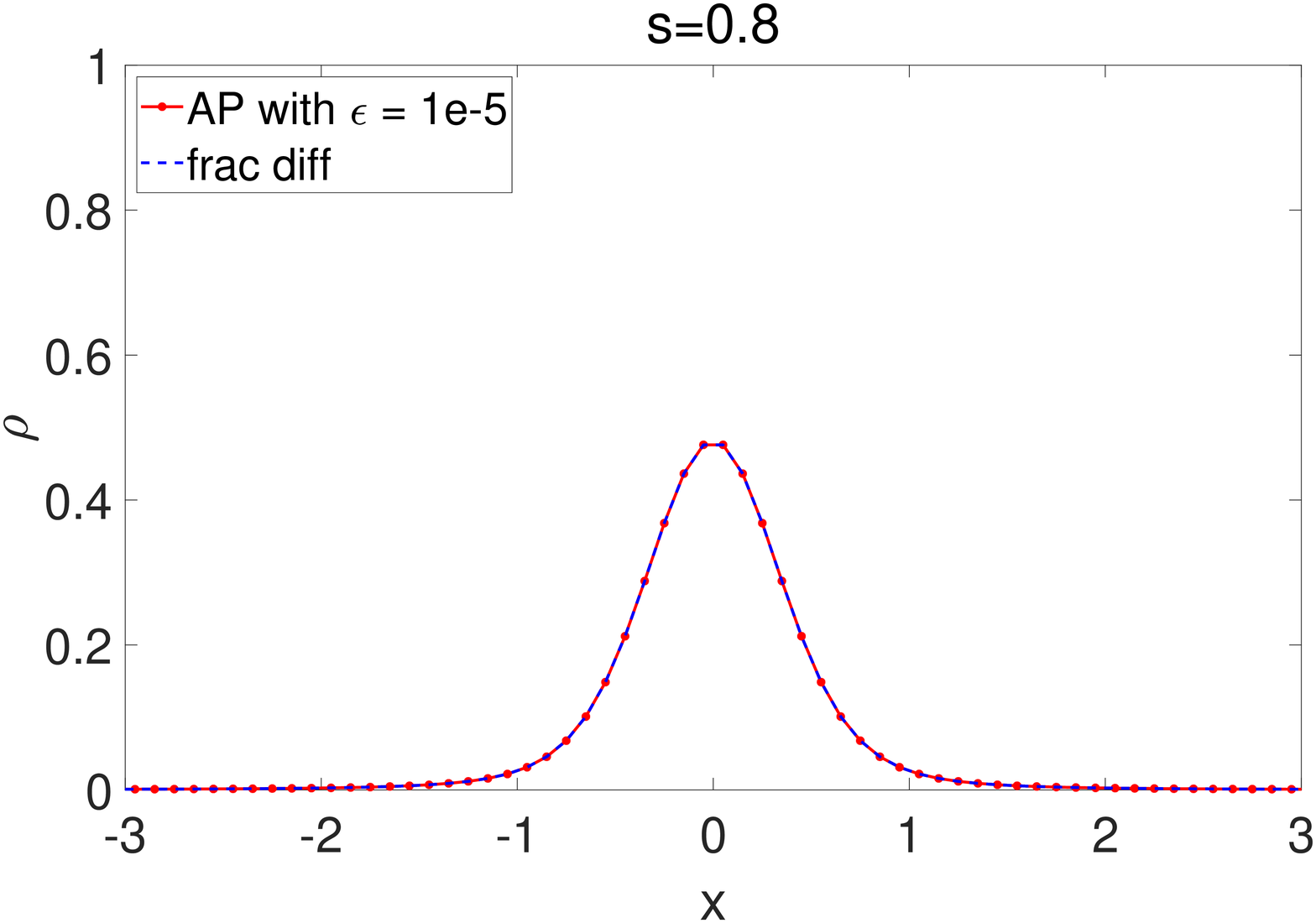}}
\caption{Computation of \eqref{eqn:111} with initial condition \eqref{eqn: IC2}. Left: asymptotic error \eqref{APerror} versus time. Right: plot of solution at $T=0.1$. $L_x = 5$, $N_x=100$, $\Delta t = 0.01$ are used in both AP scheme and the Fourier spectral method in computing $\eqref{eqn: limit_system}$. $N_v=128$, $L_v=3$ are used for velocity variable.}
\label{fig: ap_check_IC_2}
\end{figure}

\section{Conclusion}
We designed an asymptotic preserving scheme for L\'evy-Fokker-Planck equation with fractional diffusion limit. This limit emerges due to the fat tail equilibrium of the L\'evy-Fokker-Planck operator, which breaks down the classical diffusion limit as it renders the diffusion matrix unbounded. Similar anomalous diffusion was considered for the linear Boltzmann case \cite{mellet2011fractional, abdallah2011fractional}, for which asymptotic preserving schemes have been designed \cite{wang2016asymptotic, crouseilles2016numerical, CHL16, wang2019asymptotic}. 
Comparing to the linear Boltzmann case, there are two major difficulties here in constructing numerical schemes. One is that the fat tail equilibrium does not appear explicitly in the collision operator, but exits implicitly as the kernel of the collision operator. Therefore, the idea of truncating the infinite domain into a finite computational one with a tail compensation can not be directly apply, as the tail behavior is not known unless the solution reaches the equilibrium. The other comes from the derivation of the fractional diffusion limit. In the linear Boltzmann case, a reshuffled Hilbert expansion is performed to show the strong convergence of the kinetic equation to the anomalous diffusion limit and it is the building block of the design of AP scheme. In contrast, only a weak convergence is known for our case. To resolve the first difficulty, we adopt a pseudo spectral method based on rational Chebyshev polynomial, which transforms an infinite domain into a finite one, and therefore no domain decomposition is needed. For the second difficulty, we propose a novel macro-micro decomposition, with a unique macro part that is inspired by the special choice of of the test function in proving the weak convergence. The stability of the split system is obtained. We also propose an operator splitting discretization to the split system, which removes the ill-posedness due to the stiffness and reduces the computational cost from a direct implicit treatment. A rigorous asymptotic preserving property of our scheme is established.  Several numerical results are carried out to demonstrate the properties of our scheme, including asymptotic-preservation, uniform accuracy and energy dissipation.

\vspace{1cm}

{\bf Acknowledgment}: This work is partially supported by NSF grant DMS-1846854. L.W. would like to thank Dr. Min Tang and Dr. Jingwei Hu for the discussion on computing the fractional Laplacian operator.

\bibliography{reference}
\bibliographystyle{siam}

\end{document}